\documentclass{mcom-l}

\usepackage{amssymb}

\usepackage[utf8]{inputenc}
\usepackage{bbm}
\usepackage[bbgreekl]{mathbbol}
\usepackage{graphicx,subfig,tikz} 
\usepackage{comment}

\usepackage{tikz}
\usepackage{tikz-cd}
\usetikzlibrary{arrows,matrix}

\usepackage{mathtools}
\mathtoolsset{showonlyrefs=true}

\newtheorem{theorem}{Theorem}[section]
\newtheorem{corollary}[theorem]{Corollary}
\newtheorem{lemma}[theorem]{Lemma}

\theoremstyle{definition}

\newtheorem{assumption}[theorem]{Assumption}

\theoremstyle{remark}
\newtheorem{remark}[theorem]{Remark}

\numberwithin{equation}{section}

\input{macros}

\usepackage{hyperref}

\begin{document}

\title[Bounded commuting projections for non-matching interfaces]%
	{Bounded commuting projections for multipatch spaces with non-matching interfaces}


\author{Martin Campos Pinto}
\address{Max-Planck-Institut für Plasmaphysik, Boltzmannstr. 2, 85748 Garching, Germany}
\email{martin.campos-pinto@ipp.mpg.de}

\author{Frederik Schnack}
\address{Max-Planck-Institut für Plasmaphysik, Boltzmannstr. 2, 85748 Garching, Germany}
\email{frederik.schnack@ipp.mpg.de}


\date{}

\dedicatory{}


\begin{abstract}
	We present commuting projection operators on de Rham sequences of 
	two-dimensional multipatch spaces
	with local tensor-product para{\-}metrization and 
	non-matching interfaces. 
	Our construction yields projection operators which are local and stable 
	in any $L^p$ norm with $p \in [1,\infty]$:
	it applies to shape-regular spline patches with different mappings and 
	local refinements, under the assumption that neighboring patches 
	have nested resolutions and that interior vertices are shared by exactly 
	four patches.
	It also applies to de Rham sequences with homogeneous boundary conditions.
	Following a broken-FEEC approach, we first consider tensor-product commuting 
	projections on the single-patch de Rham sequences, and modify the resulting 
	patch-wise operators so as to enforce their conformity and commutation with the 
	global derivatives, while preserving their projection and stability properties 
	with constants independent of both the diameter and inner 
	resolution of the patches.
\end{abstract}

\maketitle

\tableofcontents

\section{Introduction}

Mixed finite element spaces which preserve the de Rham structure 
offer a flexible and powerful framework for the approximation of partial
differential equations.
This discretization paradigm has been extensively studied in the scope of 
electromagnetic modelling
\cite{Bossavit.1988.IEE-A, Bossavit.1998.ap, Hiptmair.2002.anum}
and has given rise to an elegant body of theoretical work which guarantees that 
compatible spaces of nodal, edge, face and volume type
lead to stable and accurate approximations to various 
differential operators 
in domains with non-smooth or non-connected boundaries
\cite{Hiptmair.1999.mcomp, Compatible.2006.IMA, Boffi.2010.anum, boffi2011discrete}. 

A notable step has been the unifying analysis of Finite Element Exterior Calculus (FEEC)
\cite{Arnold.Falk.Winther.2006.anum, Arnold.Falk.Winther.2010.bams}
developed in the general framework of Hilbert complexes.
There, the existence of bounded cochain projections, i.e.~sequences of commuting projection operators with uniform stability properties, is identified as a key ingredient for the stability, spectral accuracy and structure preservation of the discrete problems.
In parallel, $L^2$ stable commuting projection operators based on composition of 
finite element interpolation and smoothing operators have been proposed
by 
Schöberl \cite{schoberl_multilevel_2005, schoberl_posteriori_2007} for 
sequences of compatible Lagrange, N\'ed\'elec, Raviart-Thomas and discontinuous finite element spaces, 
and by Christiansen, Arnold, Falk and Winther in
 \cite{christiansen_smoothed_2007,Arnold.Falk.Winther.2006.anum,%
christiansen_stability_2007} for simplicial finite element spaces of 
differential forms in arbitrary dimensions.
These constructions have later been refined by 
Ern and Guermond \cite{ern_mollification_2016} who introduced shrinking-based mollifiers
to avoid technical difficulties with the domain boundaries, and derived
commuting projections stable in any $L^p$ norm, $p \in [1, \infty]$. 
Local commuting projection operators have also been proposed: first by 
Falk and Winther \cite{falk_local_2014} with uniform stability properties
in the domain spaces ($H^1$, $H(\curl)$, \dots) and by 
Arnold and Guzm\'an \cite{arnold_local_2021} with uniform stability in $L^2$.

Important extensions of these works have been carried out in the scope of 
isogeometric analysis methods \cite{Hughes.Cottrell.Bazilevs.2005.cmame}, 
with structure-preserving spline finite element spaces on multipatch domains 
proposed by Buffa, Sangalli, Rivas and V\'asquez in
\cite{buffa_isogeometric_2010,buffa_isogeometric_2011}.
These discretizations involve compatible sequences of 
tensor-product spline spaces defined on a Cartesian parametric domain
and transported on mapped subdomains (the patches) using pullback operators
such as contravariant and covariant Piola transformations.
The parametric tensor-product structure is attractive as it enables fast 
algorithms at the numerical level, and with the elegant construction 
of \cite{buffa_isogeometric_2011} it admits a variety of commuting projection 
operators starting from general projections for the first space of the sequence.
In particular, this process leads to commuting projections 
with uniform stability properties on single-patch spline spaces.

A difficulty, however, regards the construction of stable commuting
projection operators on multipatch spline spaces.
Because the tensor-product structure breaks down at the patch interfaces
the construction of \cite{buffa_isogeometric_2011} does not apply,
and it is unclear whether the smoothing projection approach of
\cite{schoberl_multilevel_2005,christiansen_smoothed_2007,Arnold.Falk.Winther.2006.anum}
can yield projections which are uniformly stable with respect to the inner grid
resolution of the patches, due to the non-locality of spline interpolation operators.
Although optimal convergence results for multipatch spline approximations
have been established in \cite{buffa_approximation_2015,Buffa_Doelz_Kurz_Schoeps_Vazquez_Wolf_2019},
up to our knowledge no $L^2$ stable commuting projections have been proposed for these spaces.

Another difficulty regards the extension of these constructions to locally refined
spaces. A typical configuration is when adjacent patches are discretized 
with spline spaces using different knot sequences or polynomial degrees.
Then the patches are non-matching in the sense of \cite{buffa_approximation_2015}
and the existence of commuting projection operators, let alone
stable ones, seems to be an open question.
More generally the preservation of the de Rham structure at the discrete level 
is an active research topic when locally refined splines are involved: let us cite
\cite{buffa_isogeometric_2014,johannessen_divergence-conforming_2015} 
on the construction of discrete de Rham sequences of T-spline and locally refined B-splines,
\cite{evans_hierarchical_2020} where sufficient and necessary conditions 
are proposed for the exactness of discrete de Rham sequences on hierarchical 
spline discretizations, and \cite{toshniwal_isogeometric_2021,patrizi_isogeometric_2021} 
for de Rham sequences of splines with multiple degrees and mapped domains with polar singularities.
We note, however, that none of these works propose commuting projection operators
for spline spaces with local refinement.

In this article, we provide a first answer to these questions in the 2D setting,
by constructing $L^p$ stable commuting projection operators on multipatch 
spaces which non-matching interfaces, for any $1 \le p \le \infty$.
Under the assumption that the multipatch decomposition is geometrically conforming,
that local resolutions across patch interfaces must be nested and that interior 
vertices are shared by exactly four patches,
our construction applies to general discretizations 
involving parametric tensor-product spaces with locally stable bases.
Our commuting projection operators are also local, 
and their stability holds with constants independent of both the size 
of the patches and the resolution of the individual patch discretizations.

Our construction follows a broken-FEEC approach 
reminiscent of \cite{campos-pinto_broken-feec_2022,guclu_broken_2022},
where the multipatch finite element spaces are seen as 
the maximal conforming subspaces of broken spaces defined 
as the juxtaposition of the single-patch ones.
The commuting projections are then obtained by a two-step process:
Applying the tensor-product construction of \cite{buffa_isogeometric_2011}
on the individual single-patch spaces 
-- which consists of composing antiderivative operators, stable projections
and local derivatives -- we first obtain stable projection operators on the 
broken space which commute with the patch-wise differential operators.
The second step is to modify these patch-wise projections 
close to the patch interfaces so as to enforce the conformity conditions
and their commuting properties,
while preserving their projection and stability properties. 
On the first space of the sequence where 
the conformity amounts to continuity conditions across patch interfaces,
this is done by composing the patch-wise commuting projection with a 
local discrete conforming projection which essentially consists of averaging
interface degrees of freedom. On the next spaces the patch-wise 
commuting projection are modified with additive correction 
terms which rely on carefully crafted antiderivative, local projection 
and derivative operators associated with the edge and vertex interfaces.
Our main finding is that this constructive process indeed produces
local commuting projection operators on the conforming spaces, with 
uniform $L^p$ stability properties. Moreover, our construction also applies
to de Rham sequences with homogeneous boundary conditions.
A by-product of this analysis is the optimal convergence
and spectral correctness of Hodge-Laplace operators on multipatch spaces
with local refinements. 

The outline is as follows: in Section~\ref{sec:approach} we present
the form of our commuting projection operators and state their main properties.
The structure of the broken and conforming multipatch spaces are respectively 
described in Section~\ref{sec:broken} and \ref{sec:conf}, together with 
our assumptions on the multipatch geometry and the local stability of the bases.
In Section~\ref{sec:antider} we define and study stable antiderivative operators 
associated with patches, edge and vertex interfaces, and 
our construction is finalized in Section~\ref{sec:cpo}
with a statement of our main results. Section~\ref{sec:proof} is then devoted
to proving these results, using the preliminary properties established
for the various intermediate operators. We conclude with some perspectives.

\section{Broken-FEEC approach and main result}
\label{sec:approach}

In this article, we consider several 2D de Rham sequences, 
namely the $\nabla$-$\curl$ sequence
\begin{equation} \label{dR}
	V^{0} = H^1(\Omega)
	~ \xrightarrow{ \nabla }  ~
	V^{1} = H(\curl;\Omega) 
	~ \xrightarrow{ \curl}  ~
	V^{2} = L^2(\Omega) 
\end{equation}
the $\bcurl$-$\Div$ sequence
\begin{equation} \label{cD} V^{0} = H(\bcurl;\Omega)
	~ \xrightarrow{ \bcurl}  ~
	V^{1} = H(\Div;\Omega) 
	~ \xrightarrow{ \Div}  ~
	V^{2} = L^2(\Omega)
\end{equation}
and their counterparts with homogeneous boundary conditions, namely
\begin{equation} \label{dR0} V^{0} = H^1_0(\Omega)
	~ \xrightarrow{ \nabla }  ~
	V^{1} = H_0(\curl;\Omega) 
	~ \xrightarrow{ \curl}  ~
	V^{2} = L^2(\Omega)
\end{equation}
and
\begin{equation}\label{cD0}V^{0} = H_0(\bcurl;\Omega)
	~ \xrightarrow{ \bcurl}  ~
	V^{1} = H_0(\Div;\Omega) 
	~ \xrightarrow{ \Div}  ~
	V^{2} = L^2(\Omega).
\end{equation}
We refer to \cite{Arnold.Falk.Winther.2006.anum,  Arnold.Falk.Winther.2010.bams} for their description as $L^2(\Omega)$ Hilbert complexes. To present our construction, we 
shall consider the first two sequences \eqref{dR} and \eqref{dR0} which can 
be treated in the almost same way.
The sequences \eqref{cD} and \eqref{cD0} will be handled with a standard
rotation argument in Section~\ref{sec:main}. 

Thus, we consider a sequence of finite element spaces in 2D
\begin{equation} \label{dRh}
	V^{0}_{h} 
	~ \xrightarrow{ \nabla }  ~
	V^{1}_{h} 
	~ \xrightarrow{ \curl}  ~
	V^{2}_{h}	
\end{equation}
included in the continuous spaces \eqref{dR} 
(and in \eqref{dR0} for the homogeneous case),
and defined on a multipatch domain of the form
\begin{equation} \label{Om}
	\Omega = \Int\big(\bigcup_{k \in \cK} \bar \Omega_k\big) 
	\quad \text{ with } \quad \Omega_k = F_k(\hat \Omega) 
\end{equation}
with disjoint, geometrically conforming subdomains $\Omega_k$
corresponding to smooth mappings $F_k$
defined on a reference domain $\hat \Omega = \openint{0}{1}^2$.
We further assume that each patch 
is equipped with a local sequence 
\begin{equation} \label{dR_k}
	V^{0}_{k} 
	~ \xrightarrow{ \nabla }  ~
	V^{1}_{k} 
	~ \xrightarrow{ \curl}  ~
	V^{2}_{k} 
\end{equation}
where $V^{\ell}_k = \cF^\ell_k(\hat V^{\ell}_k)$ is defined 
as the $\ell$-degree pushforward of a logical space $\hat V^{\ell}_k$ 
with a locally stable tensor-product basis
that will be described in the next section.

To allow for local refinements, we further allow neighboring patches 
with different logical spaces,
under nestedness assumptions which will also be specified later on. 
The global finite element spaces are then defined as
\begin{equation} \label{Vh}
	V^{\ell}_h = 
		\{ v \in V^\ell(\Omega) : v|_{\Omega_k} \in V^{\ell}_k  \text{ for } k \in \cK\}
\end{equation}
where again, the spaces $V^\ell(\Omega)$ are given either by \eqref{dR}
or \eqref{dR0} in the homogeneous case.

Our objective is then to design $L^p$ stable projection operators on these discrete 
spaces that yield a commuting diagram:
\begin{equation} \label{cd}
	\begin{tikzpicture}[ampersand replacement=\&, baseline] 
	\matrix (m) [matrix of math nodes,row sep=3em,column sep=4em,minimum width=2em] {
					~~ V^0 ~ \bbb
							\& ~~ V^1 ~ \bbb
									\& ~~ V^2 ~ \bbb
		\\
		~~ V^0_h ~ \bbb
				\& ~~ V^1_h ~ \bbb
						\& ~~ V^2_h ~ \bbb
		\\
	};
	\path[-stealth]
	(m-1-1) edge node [above] {$\nabla$} (m-1-2)
					edge node [right] {$\Pi^0$} (m-2-1)
	(m-1-2) edge node [above] {$\curl$} (m-1-3)
					edge node [right] {$\Pi^1$} (m-2-2)
	(m-1-3) edge node [right] {$\Pi^2$} (m-2-3)
	(m-2-1) edge node [above] {$\nabla$} (m-2-2)
	(m-2-2) edge node [above] {$\curl$} (m-2-3)
	;
	\end{tikzpicture}
\end{equation}

On a single patch $\Omega_k$, 
the approach of \cite{buffa_isogeometric_2011} starts 
from a general tensor-product projection $\hat \Pi^0_k : L^p(\hat \Omega) \to \hat V^0_k$ 
on the first logical space, and defines
projections on the next spaces of the form 
\begin{equation}
	\label{Pi12k}
	\hat \Pi^1_k \hat \bu \coloneqq \sum_{d \in \{1,2\}} \hat \nabla_d \hat \Pi^0_k \hat \Phi_d(\hat \bu)
	\quad \text{ and } \quad
	\hat \Pi^2_k \hat f \coloneqq \hat \partial_1\hat \partial_2 \hat \Pi^0_k \hat \Psi(\hat f)
\end{equation}
with directional gradient operators $\hat \nabla_d$ 
and antiderivative operators
\begin{equation*}
	\left\{\begin{aligned}
	\hat \Phi_1(\hat \bu)(\hbx) \coloneqq \int_{0}^{\hx_1} \hat u_1(z, \hx_2) \dd z
	\\
	\hat \Phi_2(\hat \bu)(\hbx) \coloneqq \int_{0}^{\hx_2} \hat u_2(\hx_1, z) \dd z	
\end{aligned}\right.
	\quad \text{ and } \quad
	\hat \Psi(\hat f)(\hbx) \coloneqq \int_{0}^{\hx_1}\int_{0}^{\hx_2} \hat f(z_1,z_2) \dd z_2\dd z_1.
\end{equation*}
The projections \eqref{Pi12k} commute with the logical differential operators
thanks to the tensor-product structure of $\hat \Pi^0_k$, and they preserve 
its stability due to the intrinsic integrability of the antiderivative operators 
and a localization argument that relies on the tensor-product product structure,
as will be explained below.
On the mapped spaces the projections are defined through pullbacks and pushforwards,
\begin{equation*}
	\Pi^{\ell}_k = \cF^\ell_k \hat \Pi^{\ell}_k (\cF^\ell_k)^{-1} :
	\quad L^p(\Omega_k) \to V^\ell_k.
\end{equation*}
Their stability and commuting properties respectively follow from the smoothness of the mapping $F_k$ and from the fact that the pullbacks commute with the differential operators, see e.g. \cite{Hiptmair.2002.anum,buffa_isogeometric_2011}.

At patch interfaces where the parametric tensor-product structure breaks down, this construction must be adapted.
Our approach is to first consider the patch-wise projections
\begin{equation} \label{Pipw}
\Pi^\ell_\pw = \sum_{k \in \cK} \Pi^\ell_k : \quad L^p(\Omega) \to V^{\ell}_\pw
\end{equation}
on the broken patch-wise spaces 
\begin{equation} \label{Vpw}
	V^{\ell}_\pw \coloneqq 
		\{ v \in L^2(\Omega) : v|_{\Omega_k} \in V^{\ell}_k  \text{ for } k \in \cK\}
\end{equation}
which are fully discontinuous at the patch interfaces,
and to modify the former so as to enforce the conformity conditions at the patch interfaces.
On the first space of \eqref{dR} where the $H^1$ conformity amounts to continuity conditions, 
in the sense that $V^0_h = V^0_\pw \cap C^0(\Omega)$,
this is done by applying a conforming projection 
$P: V^0_\pw \to V^0_h$
which averages interface degrees of freedom: we thus set  
\begin{equation}
\label{Pi0_intro}
\Pi^0 \coloneqq P \Pi^0_\pw.
\end{equation}
On the next spaces our modification takes the form of additive correction terms
associated with the patch interfaces. The global projection on $V^1_h$ has the form
\begin{equation}
	\label{Pi1_intro}
	\Pi^1 \coloneqq \Pi^1_\pw 
			+ \sum_{e \in \cE} \tilde \Pi^1_e
					+ \sum_{\bv \in \cV} \tilde \Pi^1_\bv
						+ \sum_{\bv \in \cV, e \in \cE(\bv)} \tilde \Pi^1_{e,\bv}
\end{equation}
with correction terms 
that are localized on patch edges and vertices. Like the single-patch
projections \eqref{Pi12k}, they involve antiderivative operators,
local (patch-wise) projections $\Pi^0_\pw$ and partial derivative operators. In addition, they
also involve local projection operators which vanish on conforming functions. 
Specifically, our correction terms take the following form:
\begin{equation}
	\label{tPi1_intro}
	\left\{\begin{aligned}
	\tilde \Pi^1_e \bu &\coloneqq \sum_{d \in \{\parallel, \perp\}} \nabla^e_d (P^e - I^e) \Pi^0_\pw \Phi^e_{d}(\bu)
	\\
	\tilde \Pi^1_\bv \bu &\coloneqq \nabla_\pw (P^{\bv} - \bar I^{\bv}) \Pi^0_\pw \Phi^\bv(\bu)
	\\[5pt]
	\tilde \Pi^1_{e,\bv} \bu &\coloneqq \sum_{d \in \{\parallel, \perp\}} 
						\nabla^e_d (\bar I^e_{\bv} - P^e_{\bv})\Pi^0_\pw \Phi^{\bv,e}_{d}(\bu).
	\end{aligned}
\right.
\end{equation}
Here, $\nabla_\pw$ is the patch-wise gradient operator and 
$\nabla^e_d$, $d \in \{\parallel,\perp\}$, 
are patch-wise gradients along the logical parallel and perpendicular 
directions relative to a given edge $e$: they will be defined in Section~\ref{sec:broken}.
The various operators $P^g$, $I^g$, $\bar I^g, \dots$ are 
discrete projections on local conforming and broken subspaces associated to patch
edges (for $g = e$), vertices (for $g = \bv$) and edge-vertex pairs (for $g = (\bv,e)$). 
These local projection operators will be designed so as to guarantee the grad-commuting properties 
of $\Pi^1$ and $\Pi^0$, and to vanish on continuous functions: they 
will be described in Section~\ref{sec:conf}.
Finally, the $\Phi^g$ are antiderivative operators associated 
with edges and vertices: they will be studied in Section~\ref{sec:antider}. 

Similarly, the projection on $V^2_h$ reads
\begin{equation}
	\label{Pi2_intro}
	\Pi^2 \coloneqq
		\Pi^2_\pw 
			+ \sum_{e \in \cE} \tilde \Pi^2_e
						+ \sum_{\bv \in \cV, e \in \cE(\bv)} \tilde \Pi^2_{e,\bv}
\end{equation}
with interface correction terms of the form
\begin{equation}
	\label{tPi2_intro}
	\left\{\begin{aligned}
	\tilde \Pi^2_e f & \coloneqq D^{2,e} (P^e - I^e) \Pi^0_\pw \Psi^{e}(f)
	\\
	\tilde \Pi^2_{e,\bv} f & \coloneqq 
					D^{2,e} 
						(\bar I^e_{\bv} - P^e_{\bv})\Pi^0_\pw \Psi^{\bv,e}(f).	
	\end{aligned}
\right.
\end{equation}
Here, $D^{2,e}$ is a second order patch-wise derivative and 
$\Psi^{e}$, $\Psi^{\bv,e}$ are bivariate antiderivatives: they will be 
described in Section~\ref{sec:broken} and \ref{sec:antider}.

Our findings can be summarized as follows. 

\begin{theorem}
	The operators $\Pi^\ell_h$ 
	are local projections on the spaces $V^\ell_h$, $\ell = 0, 1, 2$.
	They yield a commuting diagram \eqref{cd}
	and they are uniformly $L^p$ stable
	with respect to the 
	size and inner resolution
	of the patches.
\end{theorem}

%
This result will be formally stated and proven in Section~\ref{sec:cpo}.
A precise meaning of the uniform stability will be given by 
listing discretization parameters 
$\kappa_1$, $\kappa_2, \dots$ on which our estimates depend. Throughout the article, we will then write 
\begin{equation*}
f \lesssim g
\end{equation*}
to mean that $f \le C g$ holds for a constant that only depends on these constants $\kappa_m$ 
while $f \sim g$ indicates that both $f \lesssim g$ and $g \lesssim f$ hold.


\section{Broken multipatch spaces}
\label{sec:broken}

In this section, we describe in more detail the multipatch domains 
and the finite element spaces to which our construction applies.

\subsection{Multipatch geometry}
\label{sec:geom}

As described above, we consider a domain $\Omega$ 
of the form \eqref{Om}, made of disjoint open patches
$\Omega_k = F_k(\hat \Omega)$ with smooth mappings
$F_k$, $k \in \cK$.
We denote by $H_k$ the diameter of patch $\Omega_k$, 
and assume that the mappings are $C^1$ diffeomorphisms
with Jacobian matrices satisfying
\begin{equation*}
\norm{DF_k} \le \kappa_1 H_k 
\quad \text{ and } \quad
\norm{(DF_k)^{-1}} \le \kappa_2  (H_k)^{-1}
\qquad \text{ for all } k \in \cK.
\end{equation*}
In particular, 
\begin{equation}  \label{boundsDF}
\norm{DF_k} \sim H_k, 
\quad 
\norm{(DF_k)^{-1}} \sim (H_k)^{-1}, 
\quad
\det (DF_k) \sim H_k^2 
\end{equation}
hold for all patch $k \in \cK$.
We make the following assumptions: 
\begin{itemize}
	\item[(i)] 
	the patch decomposition is geometrically conforming,
	\item[(ii)] across any interior edge, the patch discretization spaces are nested,
	\item[(iii)] vertices are shared by at most four patches 
	(exactly four in the case of interior vertices)
	with specific nestedness properties. 
\end{itemize}
Here, Assumption (i) amounts to saying that the intersection of two closed patches 
is either empty, or a common vertex, or a common edge. 
In addition, we assume that the mappings are continuous on the patch edges, in the sense that
both sides provide the same parametrization up to a possible change in orientation.
Assumption (ii) is standard for locally refined spaces and Assumption (iii)
may be summarised by the fact that every point in the patches contiguous to a vertex
can be connected to a {\em coarse edge} by a curve that intersects two patches at most
and not the vertex itself.
This will be specified by Assumption~\ref{as:e_nested} and \ref{as:v_nested} in Section~\ref{sec:ev}.

\subsection{Tensor-product logical spaces} 
\label{tensorprodspaces}
Following \cite{Hiptmair.2002.anum,buffa_isogeometric_2011,perse_geometric_2021}, 
we consider discrete spaces on each patch which are obtained by pushing forward 
tensor-product de Rham sequences on the logical Cartesian domain 
$\hat \Omega = \openint{0}{1}^2$.
Thus, for a patch $\Omega_k$, $k \in \cK$, we consider a logical discrete de Rham sequence on $\hat \Omega$, 
\begin{equation*}
\hat V^{0}_{k}
~ \xrightarrow{ \nabla }  ~
\hat V^{1}_{k} 
~ \xrightarrow{ \curl}  ~
\hat V^{2}_{k} 
\end{equation*}
with tensor-product spaces of the form
\begin{equation*}
\hat V^{0}_{k}  \coloneqq \VV^{0}_k \otimes \VV^{0}_k,
\qquad
\hat V^{1}_{k}  \coloneqq \begin{pmatrix}
  \VV^{1}_{k} \otimes \VV^{0}_k
  \\ \noalign{\smallskip}
  \VV^{0}_k\otimes \VV^{1}_{k}
\end{pmatrix},
\qquad
\hat V^{2}_{k}  \coloneqq \VV^{1}_{k} \otimes \VV^{1}_{k}.
\end{equation*}
The univariate spaces must form de Rham sequences on the reference interval, 
i.e.
\begin{equation} \label{1D_seq}
\VV^{0}_{k} \subset W^{1,1}([0,1])
\quad \xrightarrow{ \mbox{$~ \partial_\hx ~$}}  \quad 
\VV^{1}_{k} \subset L^1([0,1])
\end{equation} 
and antiderivative operators must map back to the first space,
\begin{equation} \label{ad_1d}
	\VV^{0}_{k}
	\quad \xleftarrow{ \mbox{$~ \int^{\hx} ~$}}  \quad 
	\VV^{1}_{k}
\end{equation}
for arbitrary integration constants,
which also implies that constants belong to $\VV^{0}_{k}$.
An important particular case is provided by spline spaces 
\begin{equation*}
	\VV^{0}_{k} = \mathbbm{S}^p_{\bs\alpha},  \qquad
	\VV^{1}_{k} = \mathbbm{S}^{p-1}_{{\bs\alpha} - 1} 
\end{equation*}
where $p$ and ${\bs\alpha}$ are the degree and regularity vector of the first space, 
as described in \cite{buffa_isogeometric_2011}. This also includes
the case of polynomial spaces $\VV^{0}_{k} = \PP^p$ and $\VV^{1}_{k} = \PP^{p-1}$. 
To simplify the matching of functions across patch interfaces we further assume that 
the univariate spaces are invariant by a change of orientation, namely
\begin{equation} \label{sym_VV0}
  \vp \in \VV^{0}_{k} \quad \implies \quad \vp \circ \eta \in \VV^{0}_{k}
	\quad \text{where} \quad \eta(s) = 1-s.
\end{equation}

Our next assumption is that the first space is equipped with basis functions 
\begin{equation*}
\VV^{0}_{k} = \Span(\{\lambda^{k}_{i} : i = 0, \dots, n_k\})
\end{equation*}
with the following properties:
\begin{itemize}
	\item an interpolation property at the endpoints,
	\begin{equation} \label{ipe}
		\lambda^{k}_{i}(0) = \delta_{i,0}
		\qquad \text{ and } \qquad
		\lambda^{k}_{i}(1) = \delta_{i,n_k}.
	\end{equation}

	\item bounded overlapping and quasi-uniformity:
	the open supports of $\lambda^{k}_i$, which we denote by $\hat S^{k}_i$,
	are intervals of diameter 
	\begin{equation} \label{S_size}
		\kappa_3^{-1} \hat h_k \le \diam(\hat S^{k}_i) \le \kappa_3 \hat h_k
		\qquad \text{with} \qquad 
		\hat h_k \coloneqq (n_k+1)^{-1}.
	\end{equation}
	and they overlap in a bounded way, i.e.
	\begin{equation} \label{overlap}
		\#\big(\{ j : \hat S^{k}_j \cap \hat S^{k}_i \neq \emptyset \}\big)
		\le \kappa_4
		\qquad
		\text{ for } i = 0, \dots, n_k.
	\end{equation}

	\item inverse estimate:	
	assuming an $L^\infty$ normalization, it holds
	\begin{equation} \label{l_norm}
		\norm{\lambda^{k}_{i}}_{L^\infty} \le 1,
			\qquad 		
		\norm{\partial_\hx \lambda^{k}_{i}}_{L^\infty} 
			\le \kappa_5 (\hat h_k)^{-1}.
		\end{equation}

	\item local stability:
	there exists dual basis functions $\theta^{k}_{i}$ 
	with the same supports $\hat S^{k}_{i}$ (for simplicity), such that
	\begin{equation*}
	\int_0^1 \theta^{k}_{i}\lambda^{k}_j \dd \hx= \delta_{i,j}
	\end{equation*}
	holds for all $i, j \in \{0, \dots, n_k\}$, as well as
	the dual normalization 
	\begin{equation} \label{t_norm}
		\norm{\theta^{k}_{i}}_{L^\infty} \le \kappa_6 (\hat h_k)^{-1}.
	\end{equation}

\end{itemize}
The basis functions for $\hat V^0_k$, as well as the dual functions, are 
then defined as 
\begin{equation*}
\left\{\begin{aligned}
	\hat \Lambda^{k}_{\bi}(\hbx) \coloneqq \lambda^{k}_{i_1}(\hx_1) \lambda^{k}_{i_2}(\hx_2)
	\\
	\hat \Theta^{k}_{\bi}(\hbx) \coloneqq \theta^{k}_{i_2}(\hx_1) \theta^{k}_{i_2}(\hx_2)
\end{aligned}\right.
\qquad \text{ for } ~
\hbx \in \hat \Omega, \quad \bi \in \cI^k \coloneqq \{0, \dots, n_k\}^2.
\end{equation*}
Both functions have the same supports 
\begin{equation} \label{hSki}
\hat S^{k}_{\bi} = \hat S^{k}_{i_1} \times \hat S^{k}_{i_2}
\end{equation}
which, according to \eqref{overlap}, also overlap in a bounded way.
From \eqref{S_size} and the normalization \eqref{l_norm}, \eqref{t_norm} we infer
\begin{equation} \label{hLT_L2}
	\norm{\hat \Lambda^{k}_{\bi}}_{L^p(\hat S^k_\bi)} 
	\lesssim \hat h_k^{2/p},
	\qquad
	\norm{\hat \Theta^{k}_{\bi}}_{L^q(\hat S^k_\bi)} 
	\lesssim \hat h_k^{2/q-2} 
\end{equation}
where we have denoted $\frac 1q = 1-\frac 1p$ for $p \in [1, \infty]$. 
Note that in the case of splines, such dual functionals are standard: 
they can be obtained from the perfect B-spline of de Boor \cite{de_boor_local_1976}
or as piecewise polynomials \cite{buffa_quasi-interpolation_2016,campos_pinto_moment_2020}.
Using these dual functions, we define a projection operator
\begin{equation} \label{hPi0k}
\hat \Pi^0_k: L^p(\hat \Omega) \to \hat V^0_k, \qquad 
\hat \phi \mapsto \sum_{\bi \in \cI^k} 
	\sprod{\hat \Theta^{k}_{\bi}}{\hat \phi} \hat \Lambda^{k}_{\bi}
\end{equation}
where $\sprod{\cdot}{\cdot}$ is the usual $L^q$-$L^p$ duality product in $\hat \Omega$.
Our assumptions classically imply that both the basis and 
$\hat \Pi^0_k$ are locally stable. 
Specifically, if we define the {\em extension} of a domain 
$\hat \omega \subset \hat \Omega$ as
\begin{equation} \label{hEk_ext}
	\hat E_k(\hat \omega) \coloneqq \bigcup_{\bi \in \cI^k(\hat \omega)} \hat S^k_\bi
	\quad \text{ where  } \quad
	\cI^k(\hat \omega) \coloneqq
		\{\bi \in \cI^k : \hat S^k_\bi \cap \hat \omega \neq \emptyset\},
\end{equation}
then we have 
\begin{equation} \label{stab_hPi0}
	\norm{\hat \Pi^0_k \hat \phi}_{L^p(\hat \omega)} 
		\lesssim \norm{\hat \phi}_{L^p(\hat E_k(\hat \omega))}		
		\qquad \text{ for } ~~ \hat \phi \in L^p(\hat \Omega).
\end{equation}
To verify the stability of the basis, take 
$\hat \phi_h = \sum_{\bi \in \cI^k} \phi^k_\bi \hat \Lambda^{k}_\bi \in \hat V^0_k$, 
and use \eqref{hLT_L2}: 
\begin{equation} \label{stab_hbasis}
\begin{aligned} 
	 \abs{\phi^k_\bi} 
	&=  \abs{\sprod{\hat \Theta^{k}_\bi}{\hat \phi_h}} 
		\le  \norm{\hat \Theta^{k}_\bi}_{L^q(\hat S^k_\bi)}\norm{\hat \phi_h}_{L^p(\hat S^k_\bi)}
		\\
	&\lesssim  \hat h_k^{2/q - 2} \norm{\hat \phi_h}_{L^p(\hat S^k_\bi)} 
		\le   \hat h_k^{2/q - 2} \sum_{\bj \in \cI^k(\hat S^k_\bi)} \abs{\phi^k_\bj} \norm{\hat \Lambda^{k}_\bj}_{L^p(\hat S^k_\bi)}\\
	& \lesssim \hat h_k^{2/q+ 2/p - 2}
	 \sum_{\bj \in \cI^k(\hat S^k_\bi)} \abs{\phi^k_\bj} = \sum_{\bj \in \cI^k(\hat S^k_\bi)} \abs{\phi^k_\bj}.
\end{aligned}
\end{equation}
Similarly, we compute for $\hat \phi \in L^p(\hat \Omega)$
\begin{equation*}
	\norm{\hat \Pi^0_k \hat \phi}_{L^p(\hat \omega)}
	\le \sum_{\bi \in \cI^k(\hat \omega)} 
	\abs{\sprod{\hat \Theta^{k}_\bi}{\hat \phi}} \norm{\hat \Lambda^{k}_\bi}_{L^p} 
	\lesssim \sum_{\bi \in \cI^k(\hat \omega)}  \norm{\hat \phi}_{L^p(\hat S^k_\bi)}
\end{equation*}
so that \eqref{stab_hPi0} follows from the bounded overlapping property 
\eqref{overlap} of the supports. 
Using \eqref{l_norm}, we further infer 
$\norm{\hat \nabla \hat \Lambda^{k}_\bi}_{L^p} \lesssim \hat h_k^{2/p - 1}$
which yields the inverse estimate 
\begin{equation} \label{inverse_est}
	\begin{aligned}
	\norm{\hat \nabla \hat \phi}_{L^p(\hat \omega)} 
			&\lesssim \sum_{\bi \in \cI^k(\hat \omega)} \abs{\phi_\bi} \norm{\hat \nabla \hat \Lambda^{k}_\bi}_{L^p} \\
			& \lesssim \hat h_k^{2/p -1} \sum_{\bi \in \cI^k(\hat \omega)} \norm{\hat \Theta^{k}_\bi}_{L^q} \norm{\hat \phi}_{L^p} \\
			&\lesssim \hat h_k^{2/p + 2/q - 3} \sum_{\bi \in \cI^k(\hat \omega)} \norm{\hat \phi}_{L^p} 
			\lesssim \hat h_k^{-1}\norm{\hat \phi}_{L^p(E_h(\hat \omega))}. 
	\end{aligned}
\end{equation}

A key property of the tensor-product structure is the 
preservation of directional invariance.

\begin{lemma} \label{lem:dd_Pi0k}
	If $\hat \partial_d \hat \phi = 0$ for some dimension $d \in \{1,2\}$,
	then $\hat \partial_d \hat \Pi^0_k \hat \phi = 0$.
\end{lemma}
\begin{proof}
	Without loss of generality, assume $d=1$. Then 
	$\hat \phi(\hbx) = \hat \phi(\hx_2)$ and 
	\begin{equation*}
	\begin{aligned}
		\hat \Pi^0_k \hat \phi(\hbx)
			= \sum_{\bi \in \cI^k} \sprod{\hat \Theta^{k}_{\bi}}{\hat \phi} 
			\hat \Lambda^{k}_{\bi}(\hbx)
			= 
				\underbrace{\Big(\sum_{i_1} \sprod{\theta^{k}_{i_1}}{1} \lambda^{k}_{i_1}(\hx_1) \Big)}_{=~1}
				\Big(\sum_{i_2} \sprod{\theta^{k}_{i_2}}{\hat \phi} \lambda^{k}_{i_2}(\hx_2)\Big)
	\end{aligned}
	\end{equation*}
	where the univariate duality products are over $[0,1]$ and the last equality follows from the fact that 
	the first term in the product is the projection of a constant which belongs to the univariate
	spaces $\hat \VV^{0}_{k}$. This shows that $\hat \partial_1 \hat \Pi^0_k \hat \phi = 0$
	holds indeed.
\end{proof}

%

\subsection{pushforward spaces on the mapped patches}

On each patch $\Omega_k = F_k(\hat \Omega)$,
the local spaces are defined as the image of the logical ones by 
the pushforward operators associated with $F_k$, i.e. 
$V^{\ell}_k \coloneqq \cF^{\ell}_k (\hat V^{\ell}_{k})$.
In 2D these operators read
\begin{equation} \label{pf_gc}
  \left\{
  \begin{aligned}
  &\cF^0_k : \hat \phi \mapsto \phi \coloneqq \hat \phi \circ F_k^{-1}
  \\
  &\cF^1_k : \hat \bu \mapsto \bu \coloneqq  \big(DF_k^{-T} \hat \bu \big)\circ F_k^{-1}
  \\
  &\cF^2_k : \hat f \mapsto f \coloneqq  \big(J_{F_k}^{-1} \hat f \big)\circ F_k^{-1}
  \end{aligned}
  \right.
\end{equation}
where $DF_k = \big(\partial_b (F_k)_a(\hat \bx)\big)_{1 \le a,b \le 2}$
is the Jacobian matrix of $F_k$, and $J_{F_k}$ 
its (positive) metric determinant, 
see e.g. \cite{Hiptmair.2002.anum,kreeft2011mimetic,Buffa_Doelz_Kurz_Schoeps_Vazquez_Wolf_2019,%
Holderied_Possanner_Wang_2021}.
These operators define isomorphisms between $L^p(\hat \omega)$ and 
$L^p(\omega)$ for any open domain $\hat \omega \subset \hat \Omega$ 
with image $\omega = F_k(\hat \omega)$. Specifically, for all 
$\hat \phi$, $\hat \bu$, $\hat f \in L^p(\hat \Omega)$
the pushforwards
$\phi \coloneqq \cF^0_k \hat \phi$, $\bu \coloneqq \cF^1_k \hat \bu$ and  $f \coloneqq \cF^2_k \hat f$
satisfy
\begin{equation} \label{L2_scale_pf}
	\left\{
  \begin{aligned}
  &\norm{\phi}_{L^p(\omega)} \sim H_k^{2/p} \norm{\hat \phi}_{L^p(\hat \omega)},
  \\
  &\norm{\bu}_{L^p(\omega)} \sim H_k^{2/p - 1} \norm{\hat \bu}_{L^p(\hat \omega)},
  \\
  &\norm{f}_{L^p(\omega)} \sim H_k^{2/p - 2} \norm{\hat f}_{L^p(\hat \omega)}.
  \end{aligned}
  \right.
\end{equation}
A key property is the commutation with the differential operators, 
namely
\begin{equation*}
\nabla \cF^0_k \hat \phi = \cF^1_k \hat \nabla \hat \phi,
\qquad 
\curl \cF^1_k \hat \bu = \cF^2_k \widehat \curl \hat \bu
\end{equation*}
which holds for all $\hat \phi \in H^1(\hat \Omega)$ and 
$\hat \bu \in H(\hcurl; \hat \Omega)$.
In particular, the mapped spaces also form de Rham sequences
\eqref{dR_k}.

\subsection{Broken basis functions and patch-wise projection on $V^0_\pw$} 
\label{sec:bbf}

Basis functions for the single-patch spaces $V^\ell_k = \cF_k^\ell \hat V^\ell_k$ 
are obtained by pushing forward the reference basis functions.
Outside $\Omega_k$, we implicitly extend $\cF^\ell_k$ by zero, so that
these functions also provide a basis for the patch-wise spaces \eqref{Vpw}.
Of particular importance to us is the corresponding basis for $V^{0}_\pw$,
\begin{equation*}
\Lambda^{k}_{\bi}(\bx) \coloneqq \cF^0_k(\hat \Lambda^{k}_{\bi})
	=
	\begin{cases}
		\hat \Lambda^{k}_{\bi}(\hbx^k) ~ &\text{ on } \Omega_k
		\\
		0 &\text{ elsewhere}
	\end{cases}
	\quad \text{ for } k \in \cK, \quad \bi \in \cI^k = \{0, \dots, n_k\}^2
\end{equation*} 
where $\hbx^k \coloneqq (F_k)^{-1}(\bx)$.
These functions have local supports mapped from \eqref{hSki},
\begin{equation} \label{Ski}
S^{k}_\bi \coloneqq \supp(\Lambda^{k}_{\bi}) = F_k(\hat S^{k}_\bi) \subset \Omega_k.
\end{equation} 
The stability of the basis follows from \eqref{stab_hbasis}, \eqref{L2_scale_pf}:
for $\phi_h = \sum_{k, \bi} \phi^k_\bi \Lambda^{k}_\bi \in V^0_k$, 
\begin{equation} \label{stab_basis}
	 \abs{\phi^k_\bi}
		\lesssim  h_k^{- 2/p} \norm{\phi_h}_{L^p(S^k_\bi)} 
			\lesssim  \sum_{\bj \in \cI^k(\hat S^k_\bi)} \abs{\phi^k_\bj}		
\end{equation}
holds with  
\begin{equation} \label{hk}
	h_k \coloneqq H_k \hat h_k = \frac{\diam(\Omega_k)}{n_k + 1}. 
\end{equation}
The single-patch projection operator is then defined as
\begin{equation} \label{Pi0k}
	\Pi^0_k = \cF^0_k \hat \Pi^0_k (\cF^0_k)^{-1}
\end{equation}
see \eqref{hPi0k},
and the patch-wise projection on $V^0_\pw$ consists of 
\begin{equation} \label{Pi0pw}
	\Pi^0_\pw \coloneqq \sum_{k \in \cK} \Pi^0_k : \quad L^p(\Omega) \to V^0_\pw.
\end{equation}

\begin{lemma} \label{lem:stab_Pi0k}
	On a domain $\omega \subset \Omega_k$, $k \in \cK$,
	the single-patch projection satisfies
	\begin{equation} \label{stab_Pi0k}
		\norm{\Pi^0_k \phi}_{L^p(\omega)} 
			\lesssim \norm{\phi}_{L^p(E_k(\omega))}		
			\qquad \text{ for } ~~ \phi \in L^p(\omega),
	\end{equation}
 	where
	\begin{equation} \label{Ek_ext}
		E_k(\omega)			
			\coloneqq F_k(\hat E_k(F_k^{-1}(\omega)))
			= \bigcup_{\bi \in \cI^k: \, S^k_\bi \cap \omega \neq \emptyset} S^k_\bi,
	\end{equation}
	with $\hat E_k(\omega)$ the single-patch extension corresponding to \eqref{hEk_ext}, which is visualized in Figure \ref{fig:dom_ext_patch}.
	In particular, for any constant $c$ we have
	\begin{equation} \label{c_Pi0}
	\phi = c \text{ on } E_k(\omega) \quad \implies \quad \Pi^0_k \phi = c \text{ on } \omega
	\end{equation}
	and the patch-wise projection satisfies the global bound
	\begin{equation} \label{stab_Pi0_pw}
		\norm{\Pi^0_\pw \phi}_{L^p(\Omega)} \lesssim \norm{\phi}_{L^p(\Omega)}.		
	\end{equation}
\end{lemma}
\begin{proof}
	The bound \eqref{stab_Pi0k} follows by applying \eqref{stab_hPi0} 
	to the pullback $\hat \phi = (\cF_k^0)^{-1}\phi$ 
	and using the scaling properties \eqref{L2_scale_pf}.
	To show \eqref{c_Pi0}, we apply this bound
	to the function $\psi = \phi - c$ and use the fact that $\Pi^0_k c = c$
	holds because the constants belong to $\hat V^0_k$, hence also to $V^0_k$.
	(In passing note that \eqref{c_Pi0} holds for all functions $c \in V^0_k$,
	not only constants.)
	Estimate \eqref{stab_Pi0_pw} then follows 
	by summing the local bounds \eqref{stab_Pi0k} over $k \in \cK$ 
	with $\omega = \Omega_k$ (whose domain extension is $E_k(\Omega_k) = \Omega_k$)
	and by using the fact that the patches form a partition the domain $\Omega$.
\end{proof}

\begin{figure}[ht]
	\centering 
	\fontsize{18pt}{22pt}\selectfont
	 \resizebox{!}{0.25\textheight}{
\begingroup%
  \makeatletter%
  \providecommand\color[2][]{%
    \errmessage{(Inkscape) Color is used for the text in Inkscape, but the package 'color.sty' is not loaded}%
    \renewcommand\color[2][]{}%
  }%
  \providecommand\transparent[1]{%
    \errmessage{(Inkscape) Transparency is used (non-zero) for the text in Inkscape, but the package 'transparent.sty' is not loaded}%
    \renewcommand\transparent[1]{}%
  }%
  \providecommand\rotatebox[2]{#2}%
  \newcommand*\fsize{\dimexpr\f@size pt\relax}%
  \newcommand*\lineheight[1]{\fontsize{\fsize}{#1\fsize}\selectfont}%
  \ifx\svgwidth\undefined%
    \setlength{\unitlength}{315.62072754bp}%
    \ifx\svgscale\undefined%
      \relax%
    \else%
      \setlength{\unitlength}{\unitlength * \real{\svgscale}}%
    \fi%
  \else%
    \setlength{\unitlength}{\svgwidth}%
  \fi%
  \global\let\svgwidth\undefined%
  \global\let\svgscale\undefined%
  \makeatother%
  \begin{picture}(1,0.54)%
    \lineheight{1}%
    \setlength\tabcolsep{0pt}%
    \put(0,0){\includegraphics[width=\unitlength,page=1]{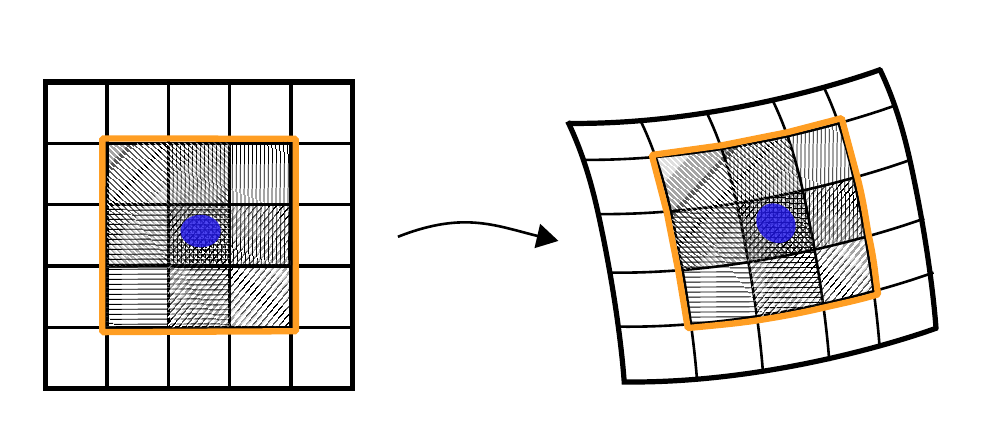}}%
    \put(0.45876556,0.24012577){\makebox(0,0)[lt]{\lineheight{1.25}\smash{\begin{tabular}[t]{l}$F_k$\end{tabular}}}}%
    \put(0.79406708,0.01062641){\makebox(0,0)[lt]{\lineheight{1.25}\smash{\begin{tabular}[t]{l}$\Omega_k$\end{tabular}}}}%
    \put(0.18716763,0.0){\makebox(0,0)[lt]{\lineheight{1.25}\smash{\begin{tabular}[t]{l}$\hat \Omega$\end{tabular}}}}%
    \put(0.37632183,0.31296674){\makebox(0,0)[lt]{\lineheight{1.25}\smash{\begin{tabular}[t]{l}$\hat\omega$\end{tabular}}}}%
    \put(0.95571533,0.32023428){\makebox(0,0)[lt]{\lineheight{1.25}\smash{\begin{tabular}[t]{l}$\omega$\end{tabular}}}}%
    \put(0,0){\includegraphics[width=\unitlength,page=2]{plots/dom_ext_p.pdf}}%
  \end{picture}%
\endgroup%
}
	\caption{Domain extensions of a domain 
	$\hat \omega \in \hat \Omega$ (left) 
	and its mapped image $\omega = F_k(\hat \omega)$ on a patch $\Omega_k$ (right), 
	as defined by equations \eqref{hEk_ext} and \eqref{Ek_ext}. 
	In different shading, we show the overlapping basis function supports 
	$\hat S^{k}_\bi$ and $S^{k}_\bi$ that have a non-empty intersection with 
	the respective $\hat \omega$ and $\omega$. The domain extensions 
	$\hat E_k(\hat \omega)$ and $E_k(\omega)$ are 
	delimited by
	the orange boundaries.
	}
	\label{fig:dom_ext_patch}
\end{figure}

\subsection{Edges and vertices}
\label{sec:ev}

In this section, we introduce some notation relative to edges and vertices,
and specify the nestedness assumptions (ii) and (iii) mentioned in Section~\ref{sec:geom}.

We first denote by $\cE$ the set of patch edges, and for a given $e \in \cE$
we gather the indices of contiguous patches in
\begin{equation*}
\cK(e) = \{k \in \cK: e \subset \partial \Omega_k\}.
\end{equation*}
and define the corresponding neighborhood as 
\begin{equation} \label{Omega_e}
	\Omega(e) \coloneqq \Int\big(\cup_{k \in \cK(e)} \bar\Omega_k\big).
\end{equation}
Due to the geometric conformity of the patch decomposition, 
$\cK(e)$ contains two patches if $e$ is an interior edge, 
and only one patch if it is a boundary edge ($e \subset \partial \Omega$).
Our edge-nestedness assumption (ii) from Section~\ref{sec:geom} then reads:
\begin{assumption} \label{as:e_nested}
	For any interior edge, $\cK(e)$ consists of two patches 
	$k^-(e)$, $k^+(e)$ of nested resolutions, in the sense that 
	\begin{equation} 
		\label{-+_assumption}
		\VV^{0}_{k^-(e)} \subset \VV^{0}_{k^+(e)}.
	\end{equation}
\end{assumption}
For boundary edges it will be convenient to denote the 
single patch of $\cK(e)$ by
\begin{itemize}
	\item $k^-(e)$ if we are in the inhomogeneous case \eqref{dR},
	\item $k^+(e)$ if we are in the homogeneous case \eqref{dR0}.
\end{itemize}

Moreover, we assume that two adjacent patches across an edge $e$ have similar
diameters and resolutions, i.e.
\begin{equation}\label{sdr}
	\kappa_7^{-1} \le \frac{H_{k^-(e)}}{H_{k^+(e)}} \le \kappa_7
	\quad \text{ and } \quad
	\kappa_8^{-1} \le \frac{n_{k^-(e)}}{n_{k^+(e)}} \le \kappa_8.
\end{equation}
Given an edge $e$ and a point $\bx \in \Omega_k$, $k \in \cK(e)$,
we denote by $(\hx^k_{\parallel}, \hx^k_{\perp})$
(with implicit dependence on $e$)
the components of the reference variable 
$\hbx^k = (\hx^k_1, \hx^k_2) \coloneqq F_k^{-1}(\bx)$ 
in the directions parallel and perpendicular to the 
reference edge $\hat e^k \coloneqq (F_k)^{-1}(e)$. 
In other terms, we set
\begin{equation} \label{ppe}
(\hx^k_{\parallel}, \hx^k_{\perp}) = (\hx^k_{\parallel(e)}, \hx^k_{\perp(e)}) 
\coloneqq \begin{cases}
	(\hx^k_1, \hx^k_2) & \text{ if $\hat e^k$ is parallel to the $\hx_1$ axis}
	\\
	(\hx^k_2, \hx^k_1) & \text{ if $\hat e^k$ is parallel to the $\hx_2$ axis}
\end{cases}
\end{equation}
and it will be convenient to denote the corresponding reordering function by 
\begin{equation}
	\label{Xke}
	\hat X^k_e : (\hx^k_\parallel,\hx^k_\perp) \mapsto \hbx^k 
	\quad \text{ and } \quad 
	X^k_e \coloneqq F_k (\hat X^k_e) : (\hx^k_\parallel,\hx^k_\perp) \mapsto \bx .
\end{equation}
Using this notation, a logical edge is always of the form 
\begin{equation} \label{e_perp}
	\hat e^k = \{\hbx \in \hat \Omega : \hx_\perp = \hat e^k_\perp\}
\end{equation}
where $\hat e^k_\perp \in \{0,1\}$ is constant for any $e$ and $k \in \cK(e)$.

Next, denoting $\cV$ the set of all patch vertices, 
we gather for any $\bv \in \cV$ the 
indices of the contiguous patches in the set
\begin{equation*}
\cK(\bv) = \{k \in \cK: \bv \in \partial \Omega_k\}. 
\end{equation*}
We define the corresponding neighborhood as the union of the contiguous patches,
\begin{equation} \label{Omega_v}
	\Omega(\bv) \coloneqq \Int\big(\cup_{k \in \cK(\bv)} \bar\Omega_k\big), 
\end{equation}
and also denote by $\cE(\bv)$ the set of contiguous edges.

\begin{figure}[ht]
	\centering 
	 {
	\fontsize{18pt}{22pt}\selectfont
	 \resizebox{!}{0.25\textheight}{
\begingroup%
  \makeatletter%
  \providecommand\color[2][]{%
    \errmessage{(Inkscape) Color is used for the text in Inkscape, but the package 'color.sty' is not loaded}%
    \renewcommand\color[2][]{}%
  }%
  \providecommand\transparent[1]{%
    \errmessage{(Inkscape) Transparency is used (non-zero) for the text in Inkscape, but the package 'transparent.sty' is not loaded}%
    \renewcommand\transparent[1]{}%
  }%
  \providecommand\rotatebox[2]{#2}%
  \newcommand*\fsize{\dimexpr\f@size pt\relax}%
  \newcommand*\lineheight[1]{\fontsize{\fsize}{#1\fsize}\selectfont}%
  \ifx\svgwidth\undefined%
    \setlength{\unitlength}{315.62072754bp}%
    \ifx\svgscale\undefined%
      \relax%
    \else%
      \setlength{\unitlength}{\unitlength * \real{\svgscale}}%
    \fi%
  \else%
    \setlength{\unitlength}{\svgwidth}%
  \fi%
  \global\let\svgwidth\undefined%
  \global\let\svgscale\undefined%
  \makeatother%
  \begin{picture}(1,0.65519165)%
    \lineheight{1}%
    \setlength\tabcolsep{0pt}%
    \put(0,0){\includegraphics[width=\unitlength,page=1]{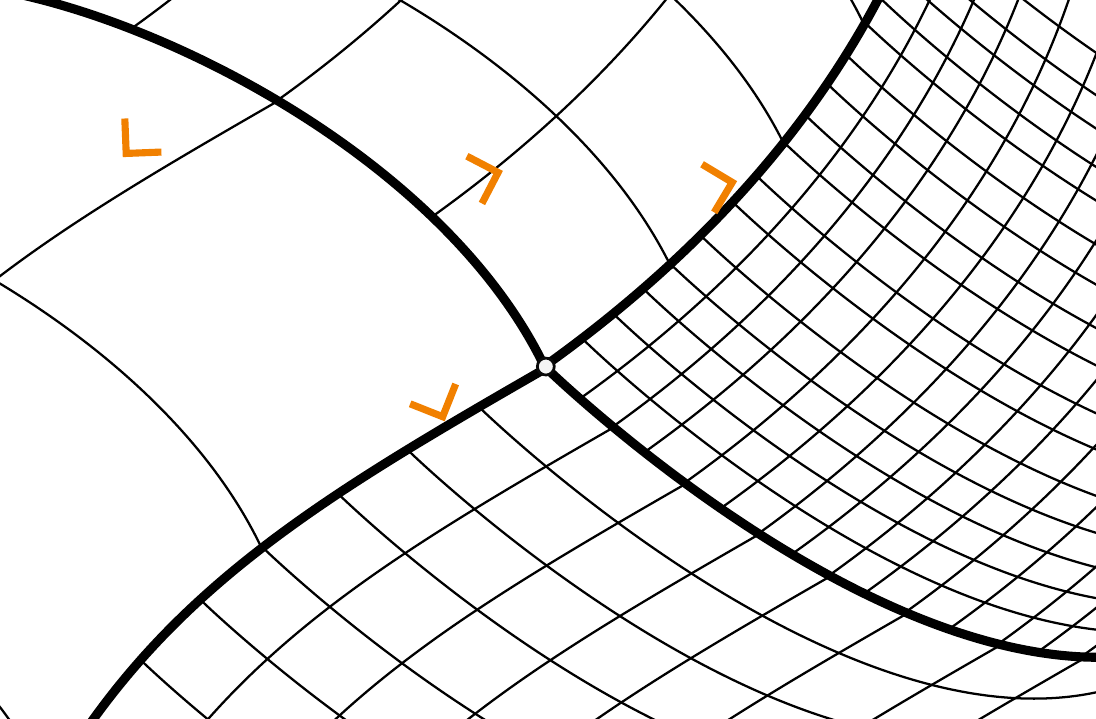}}%
    \put(0.48873267,0.3519984){\makebox(0,0)[lt]{\lineheight{1.25}\smash{\begin{tabular}[t]{l}$\bv$\end{tabular}}}}%
    \put(0.134382,0.35254007){\makebox(0,0)[lt]{\lineheight{1.25}\smash{\begin{tabular}[t]{l}$\colorbox{white}{\boxed{k^l_1(\bv)}}$\end{tabular}}}}%
    \put(0.43268163,0.55221533){\makebox(0,0)[lt]{\lineheight{1.25}\smash{\begin{tabular}[t]{l}$\colorbox{white}{\boxed{k^r_1(\bv)}}$\end{tabular}}}}%
    \put(0.70505233,0.31200439){\makebox(0,0)[lt]{\lineheight{1.25}\smash{\begin{tabular}[t]{l}\colorbox{white}{$\boxed{k^r_2(\bv)}$}\end{tabular}}}}%
    \put(0.01496137,0.57390765){\makebox(0,0)[lt]{\lineheight{1.25}\smash{\begin{tabular}[t]{l}$e^*(\bv)$\end{tabular}}}}%
    \put(0.41705911,0.07230441){\makebox(0,0)[lt]{\lineheight{1.25}\smash{\begin{tabular}[t]{l}$\colorbox{white}{\boxed{k^l_2(\bv)}}$\end{tabular}}}}%
    \put(0,0){\includegraphics[width=\unitlength,page=2]{plots/vertex_partition.pdf}}%
  \end{picture}%
\endgroup%
}
		}
	\caption{Adjacent nested patches around a vertex $\bv$
	corresponding to the decomposition \eqref{cK_seqs}, with
	$n(\bv,l) =  n(\bv,r) = 2$ since $\bv$ is an interior vertex,
	and dashed curves connecting arbitrary points $\bx \in \Omega(\bv)$
	to a coarse edge $e^*(\bv)$, according to Assumption~\ref{as:v_nested}.
	Observe that here the edge shared by patches $k_1^l$ and $k_2^l$ could 
	also be used as a coarse edge.
	The plotted cells correspond to the minimal intersections of the overlapping
	supports $S^{k}_\bi$ defined in \eqref{Ski}.
	}
	\label{fig:ver_part}
\end{figure}

To specify the vertex-nestedness assumption (iii) from Section~\ref{sec:geom},
we say that a point $\bx$ can be connected 
to an edge $e$ with a {\em monotonic curve of length $L$} if there exists a continuous
curve $\gamma: [0,1] \to \Omega$ such that $\gamma(0) \in e$, $\gamma(1) = \bx$, 
$\gamma$ intersects $L$ patches without touching any vertex and 
for any $s_1 < s_2$, the patches $\Omega_{k_i} \ni \gamma(s_i)$ satisfy
$
\VV^{0}_{k_1} \subset \VV^{0}_{k_2}.
$
Our nestedness assumption on vertices then reads as follows.
\begin{assumption} \label{as:v_nested}
	For any interior vertex $\bv$, $\cK(\bv)$ contains exactly four patches 
	and there exists an edge $e^*(\bv) \in \cE(\bv)$ such that 
	any $\bx \in \Omega(\bv)$ can be connected to $e^*(\bv)$ with a 
	monotonic curve of length $L \le 2$: 
	we call such an edge a {\em coarse edge} of $\bv$.
	For any boundary vertex, $\cK(\bv)$ contains no more than four patches, and
	\begin{itemize}
		\item in the inhomogeneous case \eqref{dR}, there also exists a coarse
		edge $e^*(\bv)$ in the sense just defined,
		\item in the homogeneous case \eqref{dR0}, any 
		$\bx \in \Omega(\bv)$ can be connected to some 
		boundary edge $e =e(\bx) \in \cE(\bv)$ with a monotonic curve of length $L \le 2$.
	\end{itemize}
\end{assumption}

\begin{figure}[ht]
	\subfloat[inhomogeneous boundary]
	{
   \fontsize{18pt}{22pt}\selectfont
	\resizebox{!}{0.25\textheight}{
\begingroup%
  \makeatletter%
  \providecommand\color[2][]{%
    \errmessage{(Inkscape) Color is used for the text in Inkscape, but the package 'color.sty' is not loaded}%
    \renewcommand\color[2][]{}%
  }%
  \providecommand\transparent[1]{%
    \errmessage{(Inkscape) Transparency is used (non-zero) for the text in Inkscape, but the package 'transparent.sty' is not loaded}%
    \renewcommand\transparent[1]{}%
  }%
  \providecommand\rotatebox[2]{#2}%
  \newcommand*\fsize{\dimexpr\f@size pt\relax}%
  \newcommand*\lineheight[1]{\fontsize{\fsize}{#1\fsize}\selectfont}%
  \ifx\svgwidth\undefined%
    \setlength{\unitlength}{315.62072754bp}%
    \ifx\svgscale\undefined%
      \relax%
    \else%
      \setlength{\unitlength}{\unitlength * \real{\svgscale}}%
    \fi%
  \else%
    \setlength{\unitlength}{\svgwidth}%
  \fi%
  \global\let\svgwidth\undefined%
  \global\let\svgscale\undefined%
  \makeatother%
  \begin{picture}(1,0.65519165)%
    \lineheight{1}%
    \setlength\tabcolsep{0pt}%
    \put(0,0){\includegraphics[width=\unitlength,page=1]{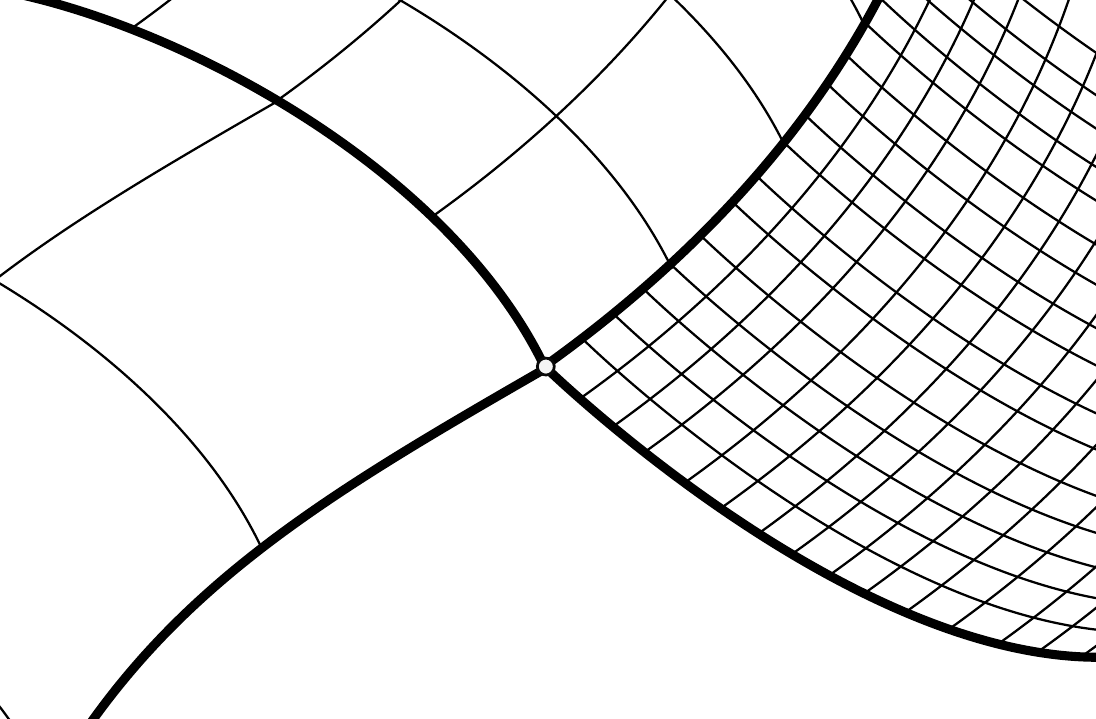}}%
    \put(0.48873267,0.3519984){\makebox(0,0)[lt]{\lineheight{1.25}\smash{\begin{tabular}[t]{l}$\bv$\end{tabular}}}}%
    \put(0.45873267,0.23019984){\makebox(0,0)[lt]{\lineheight{1.25}\smash{\begin{tabular}[t]{l}$\partial\Omega$\end{tabular}}}}%
    \put(0.134382,0.35254007){\makebox(0,0)[lt]{\lineheight{1.25}\smash{\begin{tabular}[t]{l}$\colorbox{white}{\boxed{k^l_1(\bv)}}$\end{tabular}}}}%
    \put(0.43268163,0.54221533){\makebox(0,0)[lt]{\lineheight{1.25}\smash{\begin{tabular}[t]{l}$\colorbox{white}{\boxed{k^r_1(\bv)}}$\end{tabular}}}}%
    \put(0.70505233,0.31200439){\makebox(0,0)[lt]{\lineheight{1.25}\smash{\begin{tabular}[t]{l}$\colorbox{white}{\boxed{k^r_2(\bv)}}$\end{tabular}}}}%
    \put(0.04496137,0.56390765){\makebox(0,0)[lt]{\lineheight{1.25}\smash{\begin{tabular}[t]{l}$e^*(\bv)$\end{tabular}}}}%
    \put(0,0){\includegraphics[width=\unitlength,page=2]{plots/vertex_partition_inhom_bc.pdf}}%
  \end{picture}%
\endgroup%
}
	   }
	   \\
   \subfloat[homogeneous boundary]	
   {
   \fontsize{18pt}{22pt}\selectfont
   \resizebox{!}{0.25\textheight}{
\begingroup%
  \makeatletter%
  \providecommand\color[2][]{%
    \errmessage{(Inkscape) Color is used for the text in Inkscape, but the package 'color.sty' is not loaded}%
    \renewcommand\color[2][]{}%
  }%
  \providecommand\transparent[1]{%
    \errmessage{(Inkscape) Transparency is used (non-zero) for the text in Inkscape, but the package 'transparent.sty' is not loaded}%
    \renewcommand\transparent[1]{}%
  }%
  \providecommand\rotatebox[2]{#2}%
  \newcommand*\fsize{\dimexpr\f@size pt\relax}%
  \newcommand*\lineheight[1]{\fontsize{\fsize}{#1\fsize}\selectfont}%
  \ifx\svgwidth\undefined%
    \setlength{\unitlength}{315.62072754bp}%
    \ifx\svgscale\undefined%
      \relax%
    \else%
      \setlength{\unitlength}{\unitlength * \real{\svgscale}}%
    \fi%
  \else%
    \setlength{\unitlength}{\svgwidth}%
  \fi%
  \global\let\svgwidth\undefined%
  \global\let\svgscale\undefined%
  \makeatother%
  \begin{picture}(1,0.65519165)%
    \lineheight{1}%
    \setlength\tabcolsep{0pt}%
    \put(0,0){\includegraphics[width=\unitlength,page=1]{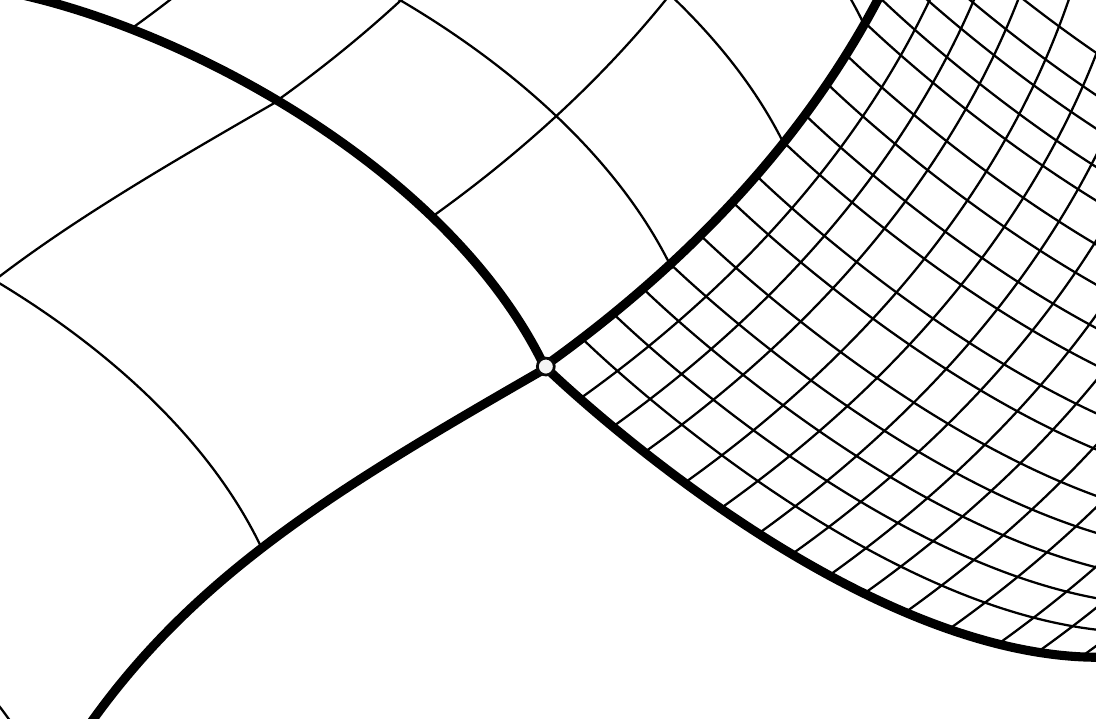}}%
    \put(0.48873267,0.3519984){\makebox(0,0)[lt]{\lineheight{1.25}\smash{\begin{tabular}[t]{l}$\bv$\end{tabular}}}}%
    \put(0.45873267,0.23019984){\makebox(0,0)[lt]{\lineheight{1.25}\smash{\begin{tabular}[t]{l}$\partial\Omega$\end{tabular}}}}%
    \put(0.134382,0.35254007){\makebox(0,0)[lt]{\lineheight{1.25}\smash{\begin{tabular}[t]{l}$\colorbox{white}{\boxed{k^r_1(\bv)}}$\end{tabular}}}}%
    \put(0.43268163,0.54221533){\makebox(0,0)[lt]{\lineheight{1.25}\smash{\begin{tabular}[t]{l}$\colorbox{white}{\boxed{k^r_2(\bv)}}$\end{tabular}}}}%
    \put(0.70505233,0.31200439){\makebox(0,0)[lt]{\lineheight{1.25}\smash{\begin{tabular}[t]{l}$\colorbox{white}{\boxed{k^l_1(\bv)}}$\end{tabular}}}}%
    \put(0,0){\includegraphics[width=\unitlength,page=2]{plots/vertex_partition_hom_bc.pdf}}%
  \end{picture}%
\endgroup%
}
   \label{fig:type_hom}
   }

	 \caption{Adjacent patches around a boundary vertex $\bv$
	 and dashed curves connecting arbitrary points $\bx \in \Omega(\bv)$
	 to (\textsc{a}) one coarse edge $e^*(\bv)$ in the inhomogeneous case, and
	 (\textsc{b}) boundary edges in the homogeneous case,
	 according to Assumption~\ref{as:v_nested}.
	 In each case, we show a decomposition of the form \eqref{cK_seqs}
	 for the adjacent nested patches.
	 }
   \label{fig:bound_ver_part}
\end{figure}

An illustration is provided in Figure~\ref{fig:ver_part} for interior
vertices, and Figure~\ref{fig:bound_ver_part} for boundary vertices.

This assumption allows to decompose any $\cK(\bv)$ in two
sequences of adjacent patches with nested resolutions 
(one of which may be empty), namely
\begin{equation} \label{cK_seqs}
	\cK(\bv) = \{k^l_i(\bv) : 1 \le i \le n(\bv,l)\} \cup 
	\{k^r_i(\bv) : 1 \le i \le n(\bv,r)\}
\end{equation}
(rotating left and right, for instance)
with integers $0 \le n(\bv,l), n(\bv,r) \le 2$, such
that if $n(\bv,s) = 2$ for $s = l$ or $r$, the patches 
$\Omega_{k^s_1(\bv)}$ and $\Omega_{k^s_2(\bv)}$ 
are adjacent and
their FEM spaces satisfy
\begin{equation*}
\VV^{0}_{k^s_{1}(\bv)} \subset \VV^{0}_{k^s_{2}(\bv)}. 
\end{equation*}
If $\bv$ is an interior vertex, we further have $n(\bv,l)= n(\bv,r) = 2$ 
and the coarse, resp. fine patches of both sequences  
(i.e. the patches $k^l_i(\bv)$ and $k^r_i(\bv)$ for $i = 1$, resp. $i = 2$) must be adjacent.
If $\bv$ is a boundary vertex then different configurations may occur:
if $\bv$ is shared by one or two patches only, then 
we can decompose $\cK(\bv)$ in a single sequence of nested, adjacent patches.
If $\bv$ is shared by three or four patches then both sequences are non-empty, and:
\begin{itemize}
	\item in the inhomogeneous case \eqref{dR}, the coarser patches 
	($k^l_1(\bv)$ and $k^r_1(\bv)$) must be adjacent 
	but the finer ones ($k^l_{n(\bv,l)}(\bv)$ and $k^r_{n(\bv,r)}(\bv)$) are not.
	\item in the homogeneous case \eqref{dR0}, the finer patches must be adjacent, 
	while the coarser ones are not.
\end{itemize}
%

For interior and inhomogeneous boundary vertices, 
we denote for later purpose by $k^*(\bv)$ the index of 
one (coarse) patch adjacent to the coarse edge $e^*(\bv)$.
For boundary vertices in the homogeneous case \eqref{dR0} 
these coarse edges and patches are left undefined.

\subsection{Patch-wise differential operators}

Using the pushforward and pullback we also define patch-wise gradient operators.
Given $d \in \{1,2\}$, we define the single-patch directional gradients as
\begin{equation} \label{nabla_kd}
	\nabla^k_d : V^{0}_k \to V^{1}_k, 
	\qquad \phi \mapsto \cF^1_k \big( \hat \btau_d \hat \partial_d (\cF^0_k)^{-1}(\phi)
\end{equation}
where $\hat \btau_d$ is the unit vector of $\RR^2$ along $\hx_d$,
and we set $\nabla^k \coloneqq \nabla^k_1 + \nabla^k_2$.
Like $\cF^1_k$, these single-patch gradients are implicitly extended 
by zero outside their patch. The patch-wise (broken) gradient is then
\begin{equation} \label{nabla_pw}
	\nabla_\pw : V^{0}_\pw \to V^{1}_\pw, 
	\qquad \phi \mapsto 
	\sum_{k \in \cK} \nabla^k (\phi|_{\Omega_k}).
\end{equation}
Finally for an edge $e \in \cE$, we define broken gradients
along the parallel and perpendicular directions: for $d \in \{\parallel,\perp\}$,
\begin{equation} \label{nabla_ed}
	\nabla^e_d : V^{0}_\pw \to V^{1}_\pw,
	\quad
	\phi \mapsto \sum_{k\in \cK(e)} \cF^1_k \big( \hat \btau^k_d \hat \partial_d (\cF^0_k)^{-1}(\phi|_{\Omega_k})\big)
\end{equation}
where $\hat \btau^k_d$ is the unit vector of $\RR^2$ along $\hx^k_d$, the 
parallel or perpendicular logical directions respective to $e$ according to \eqref{ppe}.
Observe that these operators satisfy
\begin{equation} \label{sum_nabla_ked}
	\nabla^e_\parallel + \nabla^e_\perp = \sum_{k \in \cK(e)} \nabla^k 
	= \nabla_\pw
	\quad 
	\text{ on } \quad \Omega(e), 
\end{equation}
see \eqref{Omega_e}.
To design our commuting projection on $V^2_h$, we will need 
broken second order (mixed) derivative operators:
a single-patch operator 
\begin{equation} \label{D2k}
	D^{2,k} \coloneqq \cF_k^2 \hat \partial_1 \hat \partial_2 (\cF_k^0)^{-1}
	: \quad V^0_k \to V^2_k
\end{equation}
and a patch-wise operator $D^{2,e} : V^0_\pw \to V^2_\pw$ 
associated with any edge $e \in \cE$.
We define the latter by its values on the local domain $\Omega(e)$, as 
\begin{equation} \label{D2e}
	D^{2,e} \phi|_{\Omega_k} 
		\coloneqq (\det \hat X^k_e) \cF_k^2\big(\hat \partial^k_\parallel \hat \partial^k_\perp \hat \phi^k \big)
\end{equation}
where we have denoted $\hat \phi^k \coloneqq (\cF_k^0)^{-1}\phi \in \hat V^0_k$.
Outside $\Omega(e)$, we set $D^{2,e} \phi = 0$.
Note that $\det \hat X^k_e = \pm 1$ so that $D^{2,e} \phi|_{\Omega_k}$
indeed belongs to $V^2_k$.
This operator satisfies the following relation.
\begin{lemma} \label{lem:D2e}
	Let $\psi_\parallel$, $\psi_\perp$ be functions in $V^0_\pw$ which vanish
	outside $\Omega(e)$.
	We have 
	\begin{equation*}
	\curl_\pw \Big(\sum_{d \in \{\parallel, \perp\}} \nabla^e_d \psi_d \Big)
	= D^{2,e} (\psi_\perp - \psi_\parallel).
	\end{equation*}
\end{lemma}
\begin{proof}
	Let $\bw \coloneqq \sum_{d \in \{\parallel, \perp\}} \nabla^e_d \psi_d$.
	On each patch $\Omega_k$, $k \in \cK(e)$, its pullback 
	 reads
	\begin{equation*}
	\hat \bw^k \coloneqq (\cF_k^1)^{-1} \bw
	= \sum_{d \in \{\parallel, \perp\}} \hat \btau^k_d \hat \partial^k_d \hat \psi^k_d
	= \begin{cases}
		(\hat \partial^k_\parallel \hat \psi^k_\parallel, 
			\hat \partial^k_\perp \hat \psi^k_\perp ) & \text{ if } \det \hat X^k_e = 1
			\\
		(	\hat \partial^k_\perp \hat \psi^k_\perp,
		  \hat \partial^k_\parallel \hat \psi^k_\parallel ) & \text{ if } \det \hat X^k_e = -1
	\end{cases}
	\end{equation*}
	where we have denoted $\hat \psi^k_d \coloneqq (\cF_k^0)^{-1} \psi_d$. 
	In particular, we have
	\begin{equation*}
	\begin{aligned}
		\curl_\pw \bw |_{\Omega_k}
		= \cF_k^2\big( \hcurl \hat \bw^k \big)
		= (\det \hat X^k_e) \cF_k^2\big(\hat \partial^k_\parallel \hat \partial^k_\perp (\hat \psi^k_\perp - \hat \psi^k_\parallel)\big)
		= D^{2,e} (\psi_\perp - \psi_\parallel) |_{\Omega_k}.
	\end{aligned}
	\end{equation*}	
\end{proof}


\section{Conforming multipatch spaces}
\label{sec:conf}

In this section, we specify a basis for the first conforming space 
$V^0_h = V^0_\pw \cap V^0$ 
and we construct several projection operators on local broken and conforming 
spaces that will play a central role in the construction of our commuting operators,
as presented in Section~\ref{sec:approach}. 

\subsection{Conforming constraints on patch interfaces}

The finite element spaces \eqref{Vh} are the maximal conforming subspaces
of the broken spaces \eqref{Vpw}, namely
\begin{equation*} 
	V^0_h = V^{0}_\pw \cap H^1(\Omega),
	\qquad 
	V^1_h = V^{1}_\pw \cap H(\curl;\Omega),
	\qquad 
	V^2_h = V^{2}_\pw \cap L^2(\Omega) = V^2_\pw
\end{equation*}
(in this subsection we do not consider boundary conditions for simplicity).
Since each local space $V^\ell_k$ consists of continuous functions,
the conforming subspaces 
are characterized by continuity constraints on the patch interfaces.
Specifically, a function $\phi \in V^{0}_\pw$ 
belongs to $H^1(\Omega)$, and hence to $V^0_h$, 
if and only if we have
\begin{equation} \label{V0_confcond}
	\phi|_{\Omega_{k^-}} = \phi|_{\Omega_{k^+}}
	\qquad \text{ on every } e = \partial \Omega_{k^-} \cap \partial \Omega_{k^+}
\end{equation}
and a function $\bu \in V^{1}_\pw$ belongs to $H(\curl;\Omega)$, and hence to $V^1_h$,
if and only if
\begin{equation} \label{V1_confcond}
	\btau \cdot \bu|_{\Omega_{k^-}} = \btau \cdot \bu|_{\Omega_{k^+}} 
	\qquad \text{ on every } e = \partial \Omega_{k^-} \cap \partial \Omega_{k^+}
\end{equation}
where $\btau$ denotes an arbitrary vector tangent to the interface.
For the last space of the sequence there are no constraints since $V^2_h = V^2_\pw$.

It is possible to reformulate these interface constraints on the pullback fields.
To do so, we consider the parametrization of an edge $e \in \cE$
according to the $k^-$ and $k^+$ patches, 
namely
\begin{equation} \label{x_ek}
	\bx^k_e : [0,1] \to e, \quad z \mapsto F_{k}(\hat \bx_e^k(z)) 
	\quad	\text{ with } \quad 
	\hat \bx_e^k(z) \coloneqq \hat X^k_e(z, \hat e^k_\perp)
\end{equation}
where $\hat e^k_\perp$ is as in \eqref{e_perp}.
We remind that the continuity assumption on the mappings (see Section~\ref{sec:geom})
implies that these parametrizations coincide up to a possible change in orientation,
namely an affine bijection $\eta_e : [0,1] \mapsto [0,1]$ such that
\begin{equation} \label{cont_param}
	\bx_e^-(z) = \bx_e^+(\eta_e(z))
	\quad \text{ where } \quad 
	\eta_e(z) \coloneqq \begin{cases}
		z &\text{ if orientations coincide}
		\\
		1-z &\text{ if they differ.}
	\end{cases}	
\end{equation}
For later purpose, we denote
\begin{equation} \label{eta}
	\eta_e^-(z) \coloneqq z
	\quad \text{ and } \quad 
	\eta_e^+(z) \coloneqq \eta_e(z).
\end{equation}
The continuity condition \eqref{V0_confcond} 
expressed on the pullbacks $\hat \phi^k \coloneqq (\cF^0_k)^{-1}(\phi|_{\Omega_k})$ then reads
\begin{equation} \label{V0_confcond_ref}
 \hat \phi^-(\hbx_e^-(z)) = \hat \phi^+(\hbx_e^+(\eta_e(z))), \quad z \in [0,1].
\end{equation}

To specify the curl-conforming condition \eqref{V1_confcond},
we consider the tangent vectors to $e$ oriented according to the
$k^-$ and $k^+$ patches, namely
\begin{equation*}
\btau^k_e(\bx) \coloneqq \frac{\dd \bx^k_e(z)}{\dd z} = DF_{k}(\hbx^k_e(z)) \hat \btau^k_\parallel
\end{equation*}
where $\bx = \bx^k_e(z) = F_{k}(\hbx^k_e(z)) \in e$ and 
$\hat \btau^k_\parallel$ is the 
positive unit vector parallel to the reference edge $\hat e^k = F_{k}^{-1}(e)$.
According to \eqref{cont_param}, these vectors 
coincide up to their orientation, namely
\begin{equation} \label{tau+-}
	\btau^+_e(\bx) = (\eta_e)' \btau^-_e(\bx)
	\quad \text{ with } \quad (\eta_e)' = \pm 1.	
\end{equation}
%
%
Expressed on the pullbacks 
$\hat \bu^k(\hbx^k) \coloneqq (\cF^1_k)^{-1}(\bu|_{\Omega_k})(\hbx^k) 
= DF_k^T(\hbx^k) \bu|_{\Omega_{k}}(\hbx)$, 
the curl-conformity condition \eqref{V1_confcond} then reads
\begin{equation} \label{V1_confcond_ref}
\hat \btau^-_\parallel \cdot \hat \bu^-(\hbx^-_e(z))
	= (\eta_e)' \hat \btau^+_\parallel \cdot \hat \bu^+(\hbx^+_e(z)), \quad z \in [0,1].
\end{equation}

\subsection{Conforming basis functions in $V^0_h$}
\label{sec:conf_bf}

In the inhomogeneous case \eqref{dR}, a basis for the conforming space $V^0_h$ 
can be obtained as a collection of 
\begin{itemize}
	\item single-patch basis functions associated with interior indices
	\begin{equation} \label{Lk0}
		\Lambda^{k}_{\bi} \in H^1(\Omega): ~ k \in \cK, ~~ \bi \in \{1, \dots, n_k-1\}^2 
	\end{equation}
	which are supported in $\Omega_k$ and vanish on $\partial \Omega_k$,
	\item edge-continuous basis functions
	\begin{equation} \label{Le0}
	\Lambda^{e}_{i} \in H^1(\Omega): ~ e \in \cE, ~~ i \in \{1, \dots, n_e-1\}
	\end{equation}
	which are supported in $\Omega(e)$ and vanish on $\partial \Omega(e) \setminus e$,
	see \eqref{Omega_e},
	\item vertex-continuous basis functions 
	\begin{equation} \label{Lv}
		\Lambda^{\bv} \in H^1(\Omega): ~ \bv \in \cV
	\end{equation}
	which are supported in 
	$\Omega(\bv)$ and vanish on $\partial \Omega(\bv) \cap \Omega$, see \eqref{Omega_v}.
\end{itemize}
Here the single-patch basis functions $\Lambda^k_\bi$
have been defined in Section~\ref{sec:bbf}. We now define the edge and 
vertex-continuous ones.

\begin{remark}
	In the homogeneous case \eqref{dR0}, the same basis can be used
	without the functions $\Lambda^{e}_{i} $ and $\Lambda^{\bv}$ associated 
	with boundary edges and vertices.
\end{remark}

\subsubsection{Edge-based conforming basis functions.}

		For $e \in \cE$ which is either an interior edge or a boundary edge 
		in the inhomogeneous case \eqref{dR},
		and for an index $i \in \{0, \dots, n_e\}$ with $n_e \coloneqq n_{k^-(e)}$,
		we define the edge-continuous function 
		\begin{equation} \label{Lambda_ei}
			\Lambda^{e}_{i}(\bx) \coloneqq \begin{cases}
				\hat \lambda^{-}_{i}(\hx^-_\parallel)\hat \lambda^{-}_{i^-_\perp(e)}(\hx^-_\perp)
				&\text{ on $\Omega_k$ with $k = k^-(e)$}
				\\
				\hat \lambda^{-}_{i}(\eta_e(\hx^+_\parallel))\hat \lambda^{+}_{i^+_\perp(e)}(\hx^+_\perp)
				&\text{ on $\Omega_k$ with $k = k^+(e)$}
				\\
				0 &\text{ elsewhere}
			\end{cases}
		\end{equation}
		where $X^k_e(\hx^k_\parallel, \hx^k_\perp) = \bx$ for $k \in \cK(e)$, 
		see \eqref{Xke}, and 
		\begin{equation} \label{ikpe}
				i^k_\perp(e) \coloneqq n_k \hat e^k_\perp \in \{0, n_k\}
		\end{equation}
		is the index corresponding to the constant (perpendicular) coordinate 
		of $e$ in the patch $k \in \cK(e)$, see \eqref{e_perp}.		
		These functions are continuous across $e$, i.e. their values
		on the adjacent patches $k^-(e)$ and $k^+(e)$ coincide on $e$. 
		Those with indices $i \in \{1, \dots, n_e-1\}$ 
		further vanish on 
		$\partial \Omega(e) \setminus e$ as mentioned above, 
		so they are actually continuous inside $\Omega$.
		Their supports $S^{e}_i \coloneqq \supp(\Lambda^{e}_i)$ 
		are of the form
		\begin{equation} \label{Sei}
		S^{e}_i = \Int ( \bar S^{e,-}_i \cup \bar S^{e,+}_i )
		\subset \Omega(e)
		\end{equation}
		where we have set
		$
		S^{e,k}_i \coloneqq \{X^k_e(\hx_\parallel, \hx_\perp) \in \Omega_k: 
				\eta^k_e(\hx_\parallel) \in S^{-}_{i}, 
				\hx_\perp \in S^{k}_{i^k_\perp(e)} \}.
		$
		
\subsubsection{Vertex-based conforming basis functions.}
	
For $\bv \in \cV$ which is either an interior vertex or a boundary vertex 
in the inhomogeneous case \eqref{dR},
we define the vertex-continuous function
\begin{equation} \label{Lv_def}
	\Lambda^{\bv}
			\coloneqq \sum_{e \in \cE(\bv)}\Lambda^{e}_{\bv} - \sum_{k\in \cK(\bv)} \Lambda^{k}_{\bv}
\end{equation}
where we have denoted 
\begin{equation} \label{Lke_v}
	\Lambda^{e}_{\bv} \coloneqq \Lambda^{e}_{i^e(\bv)}
	\qquad \text{ and } \qquad 
	\Lambda^{k}_{\bv} \coloneqq \Lambda^{k}_{\bi^k(\bv)}.
\end{equation}
Here, 
\begin{equation} \label{ikv}
\bi^k(\bv) \coloneqq n_k \hat \bv^k \in \{0,n_k\}^2
\end{equation}
is the index of $\bv$ in the patch $k$, 
with $\hat \bv^k \coloneqq F_k^{-1}(\bv)$, 
and 
$i^e(\bv) \coloneqq i^{k^-(e)}_\parallel(\bv)$ is the $e$-parallel component of that 
index for the patch $k=k^-(e)$.
The function $\Lambda^{\bv}$ is then supported in the domain
$\Omega(\bv)$, see \eqref{Omega_v},
and we claim that it is continuous on every edge $e$ contiguous to $\bv$.
Indeed, from the definition \eqref{Lambda_ei}, we observe that for two distinct edges
$e \neq e'$ contiguous to $\bv$ in a patch $k \in \cK(\bv)$, we have
\begin{equation*}
\Lambda^{e'}_\bv |_{e} = \Lambda^k_\bv |_{e}.
\end{equation*}
By applying this to \eqref{Lv_def} we find that
\begin{equation} \label{Levee}
	\Lambda^{\bv}|_e = \Lambda^{e}_{\bv}|_e \qquad \text{ for all $e \in \cE(\bv)$}
\end{equation}
which proves our claim.
As $\Lambda^{\bv}$ vanishes on every edge $e \notin \cE(\bv)$, it is actually continuous 
over the whole domain $\Omega$
and its support satisfies
\begin{equation} \label{Sv}
\supp(\Lambda^{\bv}) \subset S^{\bv} \coloneqq \big(\cup_{e \in \cE(\bv)}  S^{e}_{i^e(\bv)} \big) \bigcup 
\big(\cup_{k \in \cK(\bv)}  S^{k}_{\bi^k(\bv)} \big)\subset \Omega(\bv)
\end{equation}
see \eqref{Omega_v}. 
For later purposes we define 
\begin{equation} \label{he}
	\hat h_e \coloneqq \min_{k \in \cK(e)} \diam(\hat S^{k}_{i^k_\perp(e)})
	\quad \text{ and } \quad 
	h_e \coloneqq \min_{k \in \cK(e)} h_k
\end{equation}
and
\begin{equation} \label{hv}
	\hat h_\bv \coloneqq \min_{k \in \cK(\bv), d \in \{1,2\}} \diam(\hat S^{k}_{i^k_d(\bv)})
	\quad \text{ and } \quad 
	h_\bv \coloneqq \min_{k \in \cK(\bv)} h_k 	.
\end{equation}
Note that the local quasi-uniformity assumption \eqref{sdr} and the 
regularity \eqref{boundsDF} yields 
$\hat h_g \sim \hat h_k$, as well as $h_g \sim H_k \hat h_g$, 
for any $g = e$ or $\bv$ and any contiguous patch $k \in \cK(g)$.
Using \eqref{l_norm}, \eqref{S_size} and the scaling relations \eqref{L2_scale_pf}
we also find that both edge and vertex-continuous basis functions satisfy the a priori bounds 
\begin{equation} \label{bound_Lei_Lv}
	\norm{\Lambda^{e}_i}_{L^p} \lesssim h_e^{2/p}
	\quad \text{ and } \quad 
	\norm{\Lambda^{\bv}}_{L^p} \lesssim h_\bv^{2/p}.
\end{equation}

\subsection{Edge and vertex-based domain extension operators}

Given $\omega \subset \Omega(g)$ in an edge-domain (for $g=e \in \cE$) or 
a vertex-domain (for $g = \bv \in \cV$), we now define a domain extension 
$E_g(\omega) \subset \Omega(g)$.
The process is similar to the one used for single-patch domain extensions $E_k(\omega)$
defined in \eqref{Ek_ext}, but we add a few requirements
which will be convenient for studying the locality 
of our antiderivative operators: specifically, we ask that the extension is
(i) {\em patch-wise Cartesian} 
in the sense that each restriction $E^k_g(\omega) \coloneqq E_g(\omega)|_{\Omega_k}$, $k\in \cK(g)$,
is the image by $F_k$ of a convex Cartesian domain in $\hat \Omega$,
(ii) {\em continuous} across every edge $e'$ inside $\Omega(g)$,
in the sense that the closed domains $\bar E^k_g(\omega)$, $k \in \cK(e')$,
coincide on $e'$.
In addition, we ask that (iii) the overlapping supports of the 
{\em conforming basis functions} are also included in these extensions. 

As a result, we define 
the edge-based extension $E_e(\omega)$ 
of some domain $\omega \subset \Omega(e)$
as the smallest patch-wise Cartesian domain that 
is continuous across $e$ in the above sense, 
and that satisfies
\begin{equation} \label{Ee_ext}
	E_e(\omega) \supset \Big(\cup_{i \in \cI^e(\omega)} S^{e}_i 
			\Big) \bigcup \Big(\cup_{k \in \cK(e)} E_k(\omega)\Big)
\end{equation}
where we have denoted $\cI^e(\omega) = \{i \in \{0, \dots, n_e\} : S^{e}_i \cap \omega \neq \emptyset\}$. 
An illustration is provided in Figure~\ref{fig:e_ext}.

\begin{figure}[ht]
	\centering 
	 {
		\fontsize{18pt}{22pt}\selectfont
	 \resizebox{!}{0.25\textheight}{
\begingroup%
  \makeatletter%
  \providecommand\color[2][]{%
    \errmessage{(Inkscape) Color is used for the text in Inkscape, but the package 'color.sty' is not loaded}%
    \renewcommand\color[2][]{}%
  }%
  \providecommand\transparent[1]{%
    \errmessage{(Inkscape) Transparency is used (non-zero) for the text in Inkscape, but the package 'transparent.sty' is not loaded}%
    \renewcommand\transparent[1]{}%
  }%
  \providecommand\rotatebox[2]{#2}%
  \newcommand*\fsize{\dimexpr\f@size pt\relax}%
  \newcommand*\lineheight[1]{\fontsize{\fsize}{#1\fsize}\selectfont}%
  \ifx\svgwidth\undefined%
    \setlength{\unitlength}{315.62072754bp}%
    \ifx\svgscale\undefined%
      \relax%
    \else%
      \setlength{\unitlength}{\unitlength * \real{\svgscale}}%
    \fi%
  \else%
    \setlength{\unitlength}{\svgwidth}%
  \fi%
  \global\let\svgwidth\undefined%
  \global\let\svgscale\undefined%
  \makeatother%
  \begin{picture}(1,0.57894737)%
    \lineheight{1}%
    \setlength\tabcolsep{0pt}%
    \put(0,0){\includegraphics[width=\unitlength,page=1]{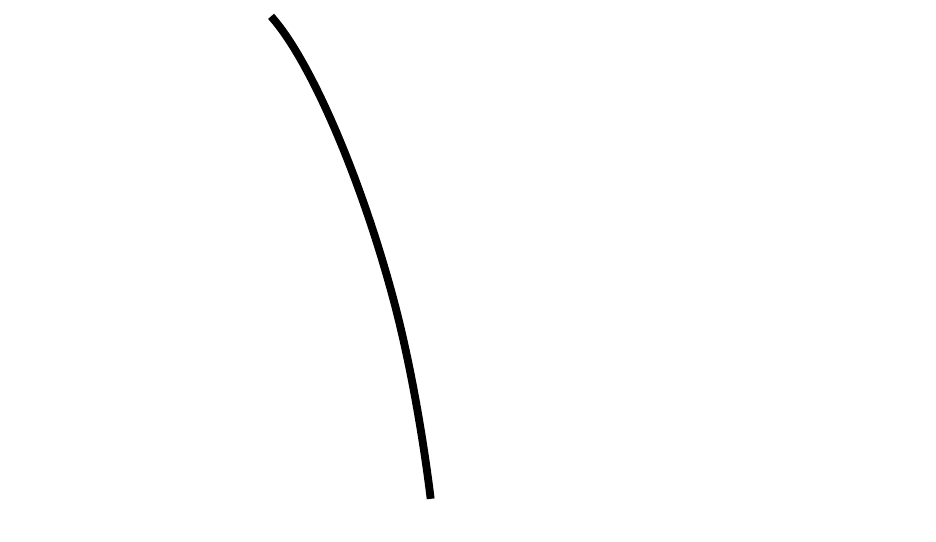}}%
    \put(0.31897897,0.54134396){\color[rgb]{0,0,0}\makebox(0,0)[lt]{\lineheight{1.25}\smash{\begin{tabular}[t]{l}$e$\end{tabular}}}}%
    \put(0,0){\includegraphics[width=\unitlength,page=2]{plots/dom_ext_e.pdf}}%
    \put(0.17927385,0.34340065){\makebox(0,0)[lt]{\lineheight{1.25}\smash{\begin{tabular}[t]{l}$\omega$\end{tabular}}}}%
  \end{picture}%
\endgroup%
}
		}
	\caption{Illustration of the edge-based extension $E_e(\omega)$ characterized by \eqref{Ee_ext}.
	 Here the overlapping supports of broken 
	 basis functions are indicated 
	 as shaded cells, and the extended domain $E_e(\omega)$ is the region delimited by the orange boundary.
 	It contains a few additional cells on the 
	 fine patch in order to be patch-wise Cartesian and continuous across the edge interface $e$.
	 }
	\label{fig:e_ext}
\end{figure}

Similarly, for $\bv \in \cV$ 
we define the vertex-based extension of a 
domain $\omega \subset \Omega(\bv)$
as the smallest patch-wise Cartesian domain $E_\bv(\omega)$ 
that is continuous across every contiguous edge $e \in \cE(\bv)$ 
in the above sense,
and that satisfies
\begin{equation} \label{Ev_ext}
	E_\bv(\omega) \supset S^{\bv} \bigcup \Big(\cup_{k \in \cK(\bv)} E_k(\omega)\Big).
\end{equation}
This domain extension is illustrated in Figure~\ref{fig:v_ext}
\begin{figure}[ht]
	\centering 
	 {
		\fontsize{18pt}{22pt}\selectfont
	 \resizebox{!}{0.25\textheight}{
\begingroup%
  \makeatletter%
  \providecommand\color[2][]{%
    \errmessage{(Inkscape) Color is used for the text in Inkscape, but the package 'color.sty' is not loaded}%
    \renewcommand\color[2][]{}%
  }%
  \providecommand\transparent[1]{%
    \errmessage{(Inkscape) Transparency is used (non-zero) for the text in Inkscape, but the package 'transparent.sty' is not loaded}%
    \renewcommand\transparent[1]{}%
  }%
  \providecommand\rotatebox[2]{#2}%
  \newcommand*\fsize{\dimexpr\f@size pt\relax}%
  \newcommand*\lineheight[1]{\fontsize{\fsize}{#1\fsize}\selectfont}%
  \ifx\svgwidth\undefined%
    \setlength{\unitlength}{315.62072754bp}%
    \ifx\svgscale\undefined%
      \relax%
    \else%
      \setlength{\unitlength}{\unitlength * \real{\svgscale}}%
    \fi%
  \else%
    \setlength{\unitlength}{\svgwidth}%
  \fi%
  \global\let\svgwidth\undefined%
  \global\let\svgscale\undefined%
  \makeatother%
  \begin{picture}(1,0.65519165)%
    \lineheight{1}%
    \setlength\tabcolsep{0pt}%
    \put(0,0){\includegraphics[width=\unitlength,page=1]{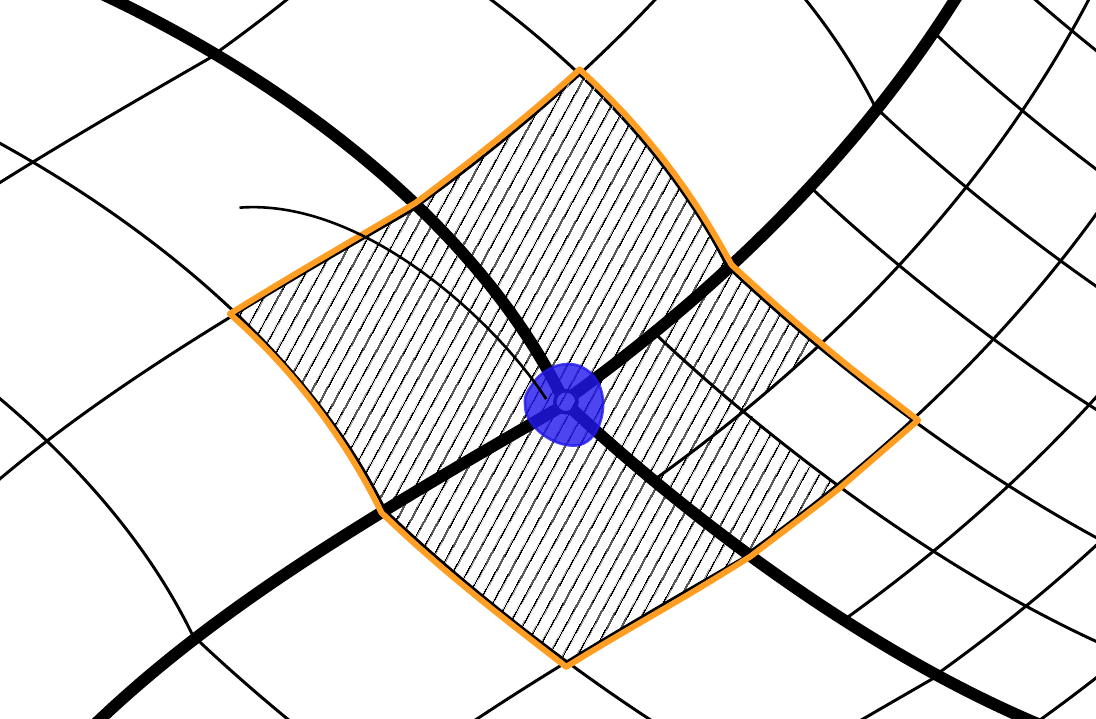}}%
    \put(0.17658191,0.455084497){\makebox(0,0)[lt]{\lineheight{1.25}\smash{\begin{tabular}[t]{l}$\omega$\end{tabular}}}}%
    \put(0.505873267,0.3269984){\makebox(0,0)[lt]{\lineheight{1.25}\smash{\begin{tabular}[t]{l}$\bv$\end{tabular}}}}%
  \end{picture}%
\endgroup%
}
		}
	\caption{
	Illustration of the vertex-based extension $E_\bv(\omega)$ characterized by \eqref{Ev_ext}.
  Here the overlapping supports of broken and vertex-continuous basis functions are indicated 
  as shaded cells, and the extended domain $E_\bv(\omega)$ is the region delimited by the orange boundary.
	It contains an additional cell on the 
  finer patch in order to be patch-wise Cartesian and continuous across the edge interface $e$.	 
	}
	\label{fig:v_ext}
\end{figure}


\begin{remark} \label{rem:stab_Pi0_ev}
	Using the nestedness properties \eqref{Ee_ext} and \eqref{Ev_ext} of 
	the edge and vertex-based domain extensions $E_e(\omega)$ and $E_\bv(\omega)$,
	we readily see that Lemma~\ref{lem:stab_Pi0k} applies to patch-wise Cartesian 
	domains $\omega \subset \Omega(g)$ with $g = e$ or $\bv$: one has
	\begin{equation} \label{stab_Pi0_ev}
		\norm{\Pi^0_\pw \phi}_{L^p(\omega)} \lesssim \norm{\phi}_{L^p(E_g(\omega))}		
	\end{equation}
	and for any constant $c$,
	\begin{equation} \label{c_Pi0_ev}
	\phi = c \text{ on } E_g(\omega) \quad \implies \quad \Pi^0_\pw \phi = c \text{ on } \omega.
	\end{equation}
\end{remark}

\subsection{Projection operators on local broken and conforming subspaces}
\label{sec:locprojs}

In order to define proper correction terms at the patch interfaces, we now introduce
several projection operators on various local subspaces of the broken space $V^0_\pw$: 
first, we define a projection on the homogeneous single-patch space $V^0_k \cap H^1_0(\Omega_k)$,
\begin{equation} \label{Ik0}
	I^k_0 : \Lambda^k_\bi \mapsto \begin{cases}
		\Lambda^k_\bi & \text{ if $\bi \in \{1,\dots, n_k-1\}^2$ }
		\\
		0 & \text{ otherwise}.
	\end{cases}	
\end{equation}

We then define two projection operators associated with an edge $e \in \cE$:
the first one is on the space spanned by the broken functions
which do not vanish identically on $e$,
\begin{equation} \label{Ie}
	I^e : \Lambda^k_\bi \mapsto \begin{cases}
		\Lambda^k_\bi \qquad & \text{ if } k \in \cK(e) \text{ and } \bi \in \cI^k_e
		\\
		0 \qquad & \text{ otherwise}
	\end{cases}
\end{equation}
where 
\begin{equation} \label{cIke}
\cI^k_e \coloneqq \{ \bi \in \cI^k : i^k_\perp = i^k_\perp(e) \}
\end{equation}
consists of the patch indices associated with the edge $e$,
see \eqref{ikpe}.
Using the reordering coordinate function $\hat X^k_e$ defined by \eqref{Xke},
we can write this index set as 
\begin{equation} \label{ike}
	\cI^k_e = \{ \bi^k_e(j) : j \in \{0, \dots, n_k\} \}
	\quad \text{ with } \quad
	\bi^k_e(j) \coloneqq \hat X^k_e(j, i^k_\perp(e)).
\end{equation}

The second edge projection is
on the space spanned by the edge-continuous basis functions:
\begin{equation} \label{Pe}
P^e : \Lambda^k_\bi \mapsto \begin{cases}
	\Lambda^e_i \qquad & \text{ if } k = k^-(e) \text{ and } \bi = \bi^k_e(i)
		\text{ with } i \in \{0, \dots, n_e\}
	\\
	0 \qquad & \text{ otherwise}
\end{cases}
\end{equation}
where we remind that $n_e \coloneqq n_{k^-(e)}$. 
Here we point out that in the homogeneous case \eqref{dR0}, 
only a fine patch $k^+(e)$ has been associated to boundary edges $e$ in Section~\ref{sec:ev}: we thus have $P^e = 0$.
Next for each vertex $\bv \in \cV$ we define projection operators on different spaces:
one on the space spanned by the broken vertex functions  
\begin{equation} \label{Iv}
I^\bv :
\Lambda^k_\bi \mapsto \begin{cases}
	\Lambda^k_\bi & \text{ if } k \in \cK(\bv) \text{ and } \bi = \bi^k(\bv)
	\\
	0 & \text{ else}
\end{cases}
\end{equation}
where $\bi^k(\bv)$ is the local index corresponding to $\bv$,
see \eqref{ikv}.

Another projection is on the single vertex-continuous basis function
\begin{equation} \label{Pv}
P^\bv : \Lambda^k_\bi \mapsto \begin{cases}
	\Lambda^\bv \qquad & \text{ if } k = k^*(\bv) \text{ and } \bi = \bi^k(\bv)
	\\
	0 \qquad & \text{ otherwise}.
\end{cases}
\end{equation}
where $k^*(\bv) \in \cK(\bv)$ is a coarse patch associated to $\bv$
as described in Section~\ref{sec:ev}. For boundary vertices 
in the homogeneous case \eqref{dR0}, where $k^*(\bv)$ is undefined, 
this leads to setting $P^\bv = 0$.

We also define a projection on the broken pieces of the vertex-continuous 
functions, namely 
\begin{equation} \label{bIv}
	\bar I^\bv : 
	\Lambda^k_\bi \mapsto \begin{cases}
	\Lambda^\bv \one_{\Omega_k} \qquad & \text{ if } k \in \cK(\bv) \text{ and } \bi = \bi^k(\bv)
	\\
	0 \qquad & \text{ otherwise}.
\end{cases}
\end{equation}
Finally, we define projection operators on different spaces spanned by 
edge-vertex functions. Again, we define three operators: one
that projects on the simple broken functions, 
\begin{equation} \label{Iev}
I^e_{\bv} : 
	\Lambda^k_\bi \mapsto \begin{cases}
	\Lambda^k_\bi \qquad & \text{ if } k \in \cK(e) \cap \cK(\bv) \text{ and } \bi = \bi^k(\bv)
	\\
	0 \qquad & \text{ otherwise}
\end{cases}
\end{equation}
one on the edge-continuous functions which do not vanish on a given vertex
\begin{equation} \label{Pev}
	P^e_{\bv} : \Lambda^k_\bi \mapsto \begin{cases}
		\Lambda^e_{\bv} \qquad & \text{ if } k = k^-(e) \in \cK(\bv) \text{ and } \bi = \bi^k(\bv)
		\\
		0 \qquad & \text{ otherwise}
	\end{cases}	
\end{equation}
(again, we note that $P^e_{\bv} = 0$ on homogeneous boundary edges 
where $k^-(e)$ is undefined)
and one on the broken pieces of the edge-continuous functions that do not vanish on the vertex:
\begin{equation} \label{bIev}
	\bar I^e_{\bv} :
	\Lambda^k_\bi \mapsto \begin{cases}
	\Lambda^e_{\bv} \one_{\Omega_k} \qquad & \text{ if } k \in \cK(e)\cap \cK(\bv) \text{ and } \bi = \bi^k(\bv)
	\\
	0 \qquad & \text{ otherwise}.
\end{cases}
\end{equation}

For later reference we observe that for all $\phi \in V^0_\pw$, 
the above edge and vertex-based projections are localized in 
edge and vertex-based supports of the form \eqref{Sei} and \eqref{Sv}:
\begin{equation} \label{supp_PIev}
	\left\{\begin{aligned}
			\big(\supp(P^e \phi) \cup \supp(I^e \phi)\big) &\subset \cup_{j=0}^{n_k} S^{e}_j
			\\
			\big(\supp(P^e_{\bv} \phi) \cup \supp(\bar I^e_{\bv} \phi)\big) 
				&\subset S^{e}_{i^e(\bv)} \subset S^{\bv}
			\\
			\big(\supp(P^\bv \phi) \cup \supp(\bar I^\bv \phi) \big)  &\subset S^{\bv}.
	\end{aligned}\right.
\end{equation}

From these definitions we can infer a few important relations. 
First, we observe that
\begin{equation} \label{sum_e_Iev}
	\sum_{e \in \cE} I^e_{\bv} = 2 I^\bv
\end{equation} 
holds for all $\bv \in \cV$.	
Multiplying \eqref{Lv_def} with $\one_{\Omega_k}$ we further obtain an equality 
relating the operators \eqref{bIv} and \eqref{bIev}
to the broken vertex-based projection \eqref{Iv}:
\begin{equation} \label{Iv_as_bc}
  \bar I^\bv = \sum_{e \in \cE} \bar I^e_{\bv} - I^\bv.
\end{equation}
Another key relation is the decomposition of 
any broken function $\phi \in V^0_\pw$ as
\begin{equation} \label{V0pw_dec}
	\phi = \Big(\sum_{k} I^k_0
		+ \sum_{e \in \cE} I^e_{0} 
		+ \sum_{v \in \cV} I^{\bv}\Big)  \phi
\end{equation}
where we have set 
\begin{equation} \label{Ie0}
I^e_{0} \coloneqq I^{e} - \sum_{\bv \in \cV} I^e_{\bv}.
\end{equation} 
By using the fact that the functions \eqref{Lk0}--\eqref{Lv} form a basis 
for the conforming space $V^0_h = V^0_\pw \cap V^0$, 
we also define a local projection on the conforming subspace,
\begin{equation} \label{P}
	P: V^0_\pw \to V^0_h, \quad \phi \mapsto \Big(\sum_{k \in \cK} I^k_0 + \sum_{e \in \cE} P^e_{0} + \sum_{\bv \in \cV} P^\bv\Big) \phi.
\end{equation}
where we have set 
\begin{equation} \label{Pe0}
	P^e_{0} \coloneqq P^e - \sum_{\bv \in \cV} P^e_{\bv}.
\end{equation} 
Note that in the homogeneous case \eqref{dR0} we have $P^e_{0} = 0$
for boundary edges and $P^\bv = 0$ for boundary vertices, so that 
$P$ is indeed a projection on the homogeneous 
conforming space $V^0_h = V^0_\pw \cap H^1_0(\Omega)$.

Several useful properties can be derived from explicit expressions of the above projections.

\begin{lemma} \label{lem:projs_exp}
Given
$\phi 
\in V^0_\pw$ and $\bx \in \Omega_k$, $k \in \cK(e)$, the edge based projections read
\begin{equation} \label{IPe_phi_px}
	\left\{\begin{aligned}
	(I^e \phi)|_{\Omega_k}(\bx)
		&= \phi|_{\Omega_k}(\bp_e(\bx))
		\lambda^{k}_{i^k_\perp(e)}(\hx^k_\perp), \\
	(P^e \phi)|_{\Omega_k}(\bx)
	&= \phi|_{\Omega^-}(\bp_e(\bx))
	\lambda^{k}_{i^k_\perp(e)}(\hx^k_\perp)	
	\end{aligned} \right.
\end{equation}
where $X^k_e(\hx^k_\parallel,\hx^k_\perp) = \bx$ as in \eqref{Xke},
$\Omega^- = \Omega_{k^-(e)}$ and
\begin{equation} \label{pex}
	\bp_e(\bx) \coloneqq X^k_e(\hx^k_\parallel,\hat e^k_\perp) 
\end{equation}
is the point on the edge $e$ that has the same parallel coordinate as $\bx \in \Omega_k$.
Similarly, the vertex based projections read
\begin{equation} \label{IPv_phi}
	\bar I^\bv \phi = \sum_{k \in \cK(\bv)} \phi|_{\Omega_{k}}(\bv) \Lambda^{\bv} \one_{\Omega_k}
	\quad \text{ and } \quad 
	P^\bv \phi 
	= \phi|_{\Omega^*}(\bv) \Lambda^{\bv}
\end{equation}
where $\Omega^* = \Omega_{k^*(\bv)}$
and the edge-vertex based projections read
\begin{equation} \label{IPev_phi}
	\bar I^e_{\bv} \phi = \sum_{k \in \cK(e)} \phi|_{\Omega_{k}}(\bv) \Lambda^{e}_\bv \one_{\Omega_k}
	\quad \text{ and } \quad 
	P^e_{\bv} \phi = \phi|_{\Omega^-}(\bv) \Lambda^{e}_\bv.
\end{equation}
In the homogeneous case \eqref{dR0} where the patches $k^-(e)$ and $k^*(\bv)$
are undefined for boundary edges and vertices, the corresponding values 
of $\phi$ can be replaced by 0.
\end{lemma}

\begin{proof}
	Write $\phi = \sum_{k \in \cK, \bi \in \cI^k} \phi^k_\bi \Lambda^{k}_\bi$.
	By definition, the projection $I^e \phi$ 
	involves broken functions $\Lambda^{k}_\bi$ 
	with $\bi = \bi^k_e(j) = X^k_e(j,i^k_\perp(e))$ 
	as in \eqref{ike}, i.e.
	\begin{equation*}	
	\Lambda^{k}_{\bi^k_e(j)}(\bx) = \lambda^{k}_{j}(\hx^k_\parallel) \lambda^{k}_{i^k_\perp(e)}(\hx^k_\perp)
	\quad \text{ for } \bx \in \Omega_k, ~ j \in \{0, \dots, n_k\}
	\end{equation*}
	while $P^e \phi$
	involves conforming functions of the form \eqref{Lambda_ei}. 
	In particular, we have 
	\begin{equation} \label{IPe_phi_coefs}
  	\left\{\begin{aligned}
  		&(I^e \phi)|_{\Omega_k}(\bx) 
				= \sum_{\bi \in \cI^k_e} \phi^k_{\bi} \Lambda^{k}_\bi(\bx)
				= \Big(\sum_{j = 0}^{n_k} \phi^k_{\bi^k_e(j)} \lambda^{k}_{j}(\hx^k_\parallel)\Big) \lambda^{k}_{i^k_\perp(e)}(\hx^k_\perp)
  	\\
  	&(P^e \phi)|_{\Omega_k}(\bx) 
  		= \sum_{j = 0}^{n_e} \phi^-_{\bi^-_e(j)} \Lambda^{e}_j(\bx)
			= \Big(\sum_{j = 0}^{n_{k^-}} \phi^-_{\bi^-_e(j)} \lambda^{-}_{j}(\eta^k_e(\hx^k_\parallel))\Big) \lambda^{k}_{i^k_\perp(e)}(\hx^k_\perp)
  	\end{aligned}\right.
  \end{equation}
	and the expressions \eqref{IPe_phi_px} 
	follow from the interpolatory property \eqref{ipe} of the basis functions in the 
	$\perp$ direction.
	For the vertex based projections we write
	\begin{equation*}
	P^\bv \phi 
	= \phi^{k^*}_{\bi^{k^*}(\bv)} \Lambda^{\bv},
	\qquad
	\bar I^\bv \phi = \sum_{k \in \cK(\bv)} \phi^{k}_{\bi^{k}(\bv)} \Lambda^{\bv} \one_{\Omega_k}
	\end{equation*}
	and 
	\begin{equation*}
	P^e_{\bv} \phi = \phi^{k^-}_{\bi^{k^-}(\bv)} \Lambda^{e}_\bv,
	\qquad 
	\bar I^e_{\bv} \phi = \sum_{k \in \cK(e)} \phi^{k}_{\bi^{k}(\bv)} \Lambda^{e}_\bv \one_{\Omega_k}.
	\end{equation*}	
	The expressions \eqref{IPv_phi} and \eqref{IPev_phi} follow from the relations 
	$\phi^k_{\bi^k(\bv)} = \phi|_{\Omega_{k}}(\bv)$, 
	which again follows from the interpolation property
	\eqref{ipe} at the patch boundaries.
\end{proof}

The following properties, which will be needed to analyze the operators $\Pi^\ell$,
are immediate corollaries of Lemma~\ref{lem:projs_exp}
\begin{lemma} \label{lem:0_corr_e}
	Let $\phi \in V^0_\pw$,
	and $e \in \cE$. The equality
	\begin{equation} \label{0_corr_e}
		P^e \phi(\bx) = I^e \phi(\bx)
	\end{equation}
	holds for all $\bx \in \Omega^-$, and all 
	$\bx \in \Omega^+$ such that 
	$\phi|_{\Omega^-}(\bp_e(\bx)) = \phi|_{\Omega^+}(\bp_e(\bx))$
	where $\bp_e(\bx)$ is the projected point on $e$, see \eqref{pex}.
\end{lemma}

\begin{lemma} \label{lem:corr_ev_ker}
	Let $\phi \in V^0_\pw$ and $\bv \in \cV$. It holds:
	\begin{equation} \label{0_corr_v}
		\text{ if } \quad \phi|_{\Omega_k}(\bv) = \phi|_{\Omega_{k'}}(\bv)
		\text{ for all } k, k' \in \cK(\bv),
		\quad \text{ then } \quad 
		(P^\bv - \bar I^\bv)\phi = 0.
	\end{equation}
	Moreover, for all $e \in \cE(\bv)$ it holds:
	\begin{equation} \label{IvIev}
		\bar I^\bv \phi = \bar I^e_{\bv} \phi  \qquad \text{ on } e
	\end{equation}
	and
	\begin{equation} \label{0_corr_ev}
		\text{ if } \quad \phi|_{\Omega_k}(\bv) = \phi|_{\Omega_{k'}}(\bv)
		\text{ for all } k, k' \in \cK(e),
		\quad \text{ then } \quad 		(P^e_{\bv} - \bar I^e_{\bv})\phi = 0.
	\end{equation}
\end{lemma}

\begin{proof}
	All these relations follow from the expressions \eqref{IPv_phi}--\eqref{IPev_phi}.
	For the equality \eqref{IvIev} we also use the relation \eqref{Levee}.
\end{proof}

Another important property is that both the broken and conforming edge projections 
preserve the invariance along the parallel direction. This partially extends the 
preservation of directional invariance of the local projections stated in Lemma~\ref{lem:dd_Pi0k}.

\begin{lemma} \label{lem:parinv}
	Let $\vp \in L^p(\Omega(e))$ such that $\nabla^e_\parallel \vp = 0$,
	where the broken parallel gradient along $e$ is defined in \eqref{nabla_ed}.
	Then,
	\begin{equation} \label{parinv}
		\nabla^e_\parallel P^e \Pi^0_\pw \vp = \nabla^e_\parallel I^e \Pi^0_\pw \vp 
		= \nabla^e_\parallel \Pi^0_\pw \vp = 0.		
	\end{equation}
\end{lemma}

\begin{proof}
	The last equality from \eqref{parinv}	follows from Lemma~\ref{lem:dd_Pi0k}.
	Apply then \eqref{IPe_phi_px} to $\phi = \Pi^0_\pw \vp$ and observe
	that $\phi|_{\Omega^k}(\bp_e(\bx))$ is constant as
	$\bp_e(\bx)$ only depends on $\hx^k_\parallel$: the result follows.
\end{proof}

We further verify that these projection operators are locally stable.

\begin{lemma}\label{lem:stab_P_I}
The broken and conforming projection operators associated 
with an edge $e$ satisfy the local bound
\begin{equation} \label{stab_P_I_e}
	\norm{Q^e \phi}_{L^p(S^{e}_j)}	\lesssim \norm{\phi}_{L^p(E_e(S^{e}_j))}
	\quad \text{ with } ~~ Q^e = I^e, P^e \text{ or } P^e_{0}
\end{equation}
where $S^{e}_j$, $j \in \{0, \dots, n_e\}$, is the local domain \eqref{Sei} 
and $E_e$ is the edge-based domain extension operator \eqref{Ee_ext}.
Similarly, 
the broken and conforming projection operators associated 
with a vertex $\bv$ satisfy
\begin{equation} \label{stab_P_I_v}
	\norm{Q_\bv \phi}_{L^p(S^{\bv})}	\lesssim \norm{\phi}_{L^p(E_\bv(S^{\bv}))}
	\quad \text{ with } ~~ Q_\bv = \bar I^\bv \text{ or } P^\bv
\end{equation}
where $S^{\bv}$ is the local domain \eqref{Sv}
and $E_\bv$ is the vertex-based domain extension operator defined in \eqref{Ev_ext}.
\end{lemma}

\begin{proof}
	Write $\omega = S^{e}_j$ and $\omega_k \coloneqq \Omega_k \cap \omega$ for $k \in \cK(e)$.
	Consider then $I^e \phi$ written as in \eqref{IPe_phi_coefs} and use $\cI^k_e \subset \cI^k$ 
	with the local stability of the basis \eqref{stab_basis}: this yields
	\begin{equation} \label{est_Ie}
		\norm{I^e \phi}_{L^p(\omega_k)}
			\le \sum_{\bi \in \cI^k(\omega_k)} \abs{\phi^k_{\bi}} \norm{\Lambda^{k}_\bi}_{L^p(\omega_k)} 
			\lesssim h_k^{2/p} \sum_{\bi \in \cI^k(\omega_k)} \abs{\phi^k_{\bi}}
			\lesssim \norm{\phi}_{L^p(E_k(\omega_k))}.
	\end{equation}
	Summing over $k \in \cK(e)$
	and using $E_k(\omega_k) \subset E_e(\omega)$, see \eqref{Ee_ext}, 
 yields \eqref{stab_P_I_e} for $Q^e = I^e$.
	For $Q^e = P^e$ we consider the two patches together
	(the argument is the same for $P^e_{0}$). 
	Using \eqref{IPe_phi_coefs} and \eqref{bound_Lei_Lv} we write
	\begin{equation*}
	\norm{P^e \phi}_{L^p(\omega)} 
	\le 
		\sum_{i \in \cI^e(\omega)} \abs{\phi^-_{\bi_e(i)}} \norm{\Lambda^{e}_i}_{L^p(\omega)}
		\lesssim h_e^{2/p} \sum_{i \in \cI^e(\omega)} \abs{\phi^-_{\bi_e(i)}}
	\end{equation*}	
	where the indices $\cI^e(\omega)$ are as in \eqref{Ee_ext}.
	We then observe that $S^{k^-}_{\bi_e(i)} \subset S^{e}_{i} \subset E_e(\omega)$ 
	for all $i \in \cI^e(\omega)$.
	With \eqref{stab_basis} and the bounded overlapping of the supports 
	this allows us to write
	\begin{equation*}
	h_e^{2/p} \sum_{i \in \cI^e(\omega)} \abs{\phi^-_{\bi_e(i)}} \lesssim \norm{\phi}_{L^p(E_e(\omega))}
	\end{equation*}
	which proves \eqref{stab_P_I_e} for $Q^e = P^e$.
	The same arguments prove the bounds \eqref{stab_P_I_v}.
\end{proof}

The next estimate is a corollary of Lemma~\ref{lem:stab_P_I}.
\begin{lemma} \label{lem:stab_P}
	The conforming projection \eqref{P} satisfies
	\begin{equation} \label{stab_P}
		\norm{P \phi}_{L^p(\Omega)}	\lesssim \norm{\phi}_{L^p(\Omega)},
		\qquad \phi \in V^0_\pw.
	\end{equation}
\end{lemma}

\begin{proof}
	The different terms in \eqref{P} can be bounded on local domains 
	which all overlap in a bounded way. For the first term, we 
	argue as in \eqref{est_Ie} and write
	\begin{equation*}
	\norm{I^k_0 \phi}_{L^p(S^{k}_\bi)} 
		\lesssim h_k^{2/p} \sum_{\bj \in \cI^k(S^k_\bi)} \abs{\phi^k_\bj} 
		\lesssim \sum_{\bj \in \cI^k(S^k_\bi)} \norm{\phi}_{L^p(S^k_\bj)}
		\lesssim \norm{\phi}_{L^p(E_k(S^k_\bi))}.
	\end{equation*}
	Summing over $\bi \in \cI^k$ then 
	yields $\norm{I^k_0 \phi}_{L^p(\Omega_k)}  \lesssim \norm{\phi}_{L^p(\Omega_k)}$
	because the extensions $E_k(S^k_\bi)$ also overlap in a bounded way according to 
	\eqref{overlap}, and in particular we have
	\begin{equation*}
	\Bignorm{\sum_k I^k_0 \phi}^2_{L^p(\Omega)} = \sum_k \norm{I^k_0 \phi}^2_{L^p(\Omega_k)}  
	\lesssim \sum_k \norm{\phi}^2_{L^p(\Omega_k)} = \norm{\phi}^2_{L^p(\Omega)}.
	\end{equation*}
	For the edge terms $\sum_e P^e_{0} \phi$ we proceed in the same way, starting with the local bounds 
	\eqref{stab_P_I_e} on the local domains $S^{e}_j$ whose union over $j = 0, \dots, n_e$ cover 
	the support of $P^e_{0} \phi$, see \eqref{supp_PIev}, and using the bounded overlapping 
	of these domains and their edge-based extensions. 
	The vertex terms are treated in the same way, using the local bounds 
	\eqref{stab_P_I_v} on the domains $S^{\bv}$.
\end{proof}


\section{$L^p$ stable antiderivative operators}
\label{sec:antider}

Our construction \eqref{Pi1_intro} for $\Pi^1$ relies on several local antiderivative operators:
\begin{itemize}
		\item a single-patch antiderivative $\Phi^k_d$ for $k \in \cK$
		and a direction $d \in \{1,2\}$,
		\item edge antiderivatives $\Phi^e_d$ 
		for $e \in \cE$
		and a relative direction $d \in \{\parallel, \perp\}$,
		\item a vertex antiderivative $\Phi^\bv$ for $\bv \in \cV$.
\end{itemize}
Given a vector valued function $\bu$, these operators take the general form 
\begin{equation} \label{Phi_a_form}
	\Phi(\bu)(\bx) = \frac{1}{\hat h} \int_0^{\hat h} \int_{\gamma(\bx,a)} \bu \cdot \dd l 
	\dd a
\end{equation}
where $\hat h$ is an averaging resolution and 
for every value of the averaging parameter $a$, 
$\gamma(\bx,a)$ is a curve connecting $\bx$ and some starting point $\gamma_*(\bx,a)$ 
which may or may not depend on $\bx$. In particular, applied to gradients these will 
satisfy a relation of the form
\begin{equation} \label{Phi_grad}
\Phi(\nabla \phi)(\bx) 
	= \frac{1}{\hat h} \int_0^{\hat h} \int_{\gamma(\bx,a)} \nabla \phi \cdot \dd l \dd a
	= \phi(\bx) - \frac{1}{\hat h} \int_0^{\hat h}\phi(\gamma_*(\bx,a))\dd a.
\end{equation}
In a similar fashion, we will define bivariate antiderivative operators of the form 
\begin{equation*}
\Psi(f)(\bx) = \frac{1}{\hat h} \int_0^{\hat h} \iint_{\sigma(\bx,a)} f \dd \bz \dd a
\end{equation*}
which will be involved in the commuting projection $\Pi^2$.

\subsection{Single-patch antiderivative operators}

In the case of single-patch antiderivative operators $\Phi^k_d$, 
the integration curve does not depend on $a$ and for $\bx \in \Omega_k$
it is fully contained in $\Omega_k$.
Writing $\hbx = F_k^{-1}(\bx)$ we parametrize it as 
\begin{equation*}
\gamma^k_d(\bx) = F_k(\hat \gamma_d(\hat \bx, [0,\hx_d]))
\quad \text{ with } \quad
\hat \gamma_d(\hat \bx, \cdot):  [0,\hx_d] \ni z \mapsto \begin{cases}
	(z, \hx_2) &\text{ if } d = 1
	\\
	(\hx_1, z) &\text{ if } d = 2.
\end{cases}
\end{equation*} 
Using the invariance of path integrals through 1-form 
pullback $(\cF^1_k)^{-1}: \bu \mapsto \hat \bu^k$, this results 
in defining the directional antiderivative operators as
\begin{equation} \label{Phi_kd_exp}
\Phi^k_1(\bu)(\bx) \coloneqq \int_{0}^{\hx_1} \hat u^k_1(z_1, \hx_2 ) \dd z_1
\quad \text{ and } \quad
\Phi^k_2(\bu)(\bx) \coloneqq \int_{0}^{\hx_2} \hat u^k_2(\hx_1, z_2 ) \dd z_2.
\end{equation}
As already mentioned, these operators play a central role in the 
tensor-product construction of \cite{buffa_isogeometric_2011}. We review their
main properties in our framework.

\begin{lemma} \label{lem:Phi_kd_nabla}
	If $\bu = \nabla \phi$ with $\phi \in C^1(\Omega)$, then
		\begin{equation} \label{Phi_kd_nabla}
			\Phi^k_d(\bu)(\bx) = \phi(\bx) - \phi(F_k(\bar \bx)) \quad \text{ on } \Omega_k,			
		\end{equation}
		where 
		$\bar x_d = 0$ and $\bar x_{d'} = \hx_{d'}$ for the other component. 
\end{lemma}
\begin{proof}
	The proof is straightforward.	
\end{proof}

\begin{lemma} \label{lem:stab_Phik}
	The single-patch antiderivative operators are stable in $L^p(\Omega(e))$,
	\begin{equation} \label{L2_bound_Phikd}
		\norm{\Phi^{k}_{d}(\bu)}_{L^p(\Omega_k)} \lesssim \norm{\bu}_{L^p(\Omega_k)}.
	\end{equation}
	Moreover, on a local domain $S^k_\bi$ of the form \eqref{Ski}, $\bi \in \cI^k$, we have 
	the bound 
	\begin{equation} \label{loc_bound_Pik1}
		\norm{\nabla^k_d \Pi^0_k \Phi^{k}_{d}(\bu)}_{L^p(S^k_\bi)} 
		\lesssim \norm{\bu}_{L^p(E_k^2(S^k_\bi))}
	\end{equation}
	where $E_k^2$ is the two-fold single-patch domain extension operator, see \eqref{Ek_ext}. 
\end{lemma}

\begin{proof}
	Using the scaling of \eqref{L2_scale_pf}, we work with the 
	pullback $\hat \Phi^k_d(\hat \bu^k)$ on the reference domain 
	$\hat \Omega$. Without loss of generality, we consider $d=1$,
	and with a H\"older 
	inequality we bound 
  \begin{equation} \label{CS_Phi_k}
		\begin{aligned}
		\norm{\hat \Phi^{k}_1(\hat \bu^k)}^p_{L^p(\hat \Omega)}
		&= \iint_{\hat \Omega} \Bigabs{\int_{0}^{\hx_1} \hat u^k_1(z,\hx_2) \dd z }^p \dd \hbx
		\\
		&\le
		\iint_{\hat \Omega} \abs{\hx_1}^{p-1} \int_{0}^{\hx_1} \abs{\hat \bu^k(z,\hx_2)}^p \dd z \dd \hbx
		\le 
		\norm{\hat \bu^k}^2_{L^p(\hat \Omega)}
		\end{aligned}
  \end{equation}
	and \eqref{L2_bound_Phikd} follows from the scaling relations \eqref{L2_scale_pf} and the bound $H_k \lesssim 1$.
	Turning to the local estimate we observe that for any fixed $\tilde x_1 \in [0,1]$, 
	the antiderivative 
	\begin{equation*}
	\tilde \Phi^{k}_1(\hat \bu^k)(\hbx) \coloneqq \int_{\tilde x_1}^{\hx_1} \hat u^k_1(z_1, \hx_2 ) \dd z_1
	\end{equation*}
	satisfies 
	$
	\hat \nabla_1\hat \Phi^{k}_1(\hat \bu^k) = \hat\nabla_1 \tilde \Phi^{k}_1(\hat \bu^k),
	$
	hence $\hat\nabla_1 \hat\Pi^0_k \hat \Phi^{k}_1(\hat \bu^k) = \hat \nabla_1 \hat\Pi^0_k \tilde \Phi^{k}_1(\hat \bu^k)$ 
	by Lemma~\ref{lem:dd_Pi0k}. 
	Using the inverse estimate \eqref{inverse_est} and the local stability \eqref{stab_Pi0k} 
	we next bound
	\begin{equation*}
	\norm{\hat\nabla_1 \hat \Pi^0_k \tilde \Phi^{k}_1(\hat \bu^k)}_{L^p(\hat S^k_\bi)} 
	\lesssim
  	\hat h_k^{-1} \norm{\hat \Pi^0_k \tilde \Phi^{k}_1(\hat \bu^k)}_{L^p(\hat E_h(\hat S^k_\bi))}
	\lesssim 
		\hat h_k^{-1} \norm{\tilde \Phi^{k}_1(\hat \bu^k)}_{L^p(\hat E^2_h(\hat S^k_\bi))}.
	\end{equation*}		
	We then observe that $\hat E^2_h(\hat S^k_\bi)$ is of diameter $\lesssim \hat h_k$,
	and we fix $\tilde x_1 \in \hat S^k_{i_1}$ which
	according to the locality properties \eqref{S_size}--\eqref{overlap}, 
	satisfies $\abs{\hx_1-\tilde x_1} \lesssim \hat h_k$
	for all $\hbx \in \hat E^2_h(\hat S^k_\bi)$. We then compute as in \eqref{CS_Phi_k}: this gives 
	\begin{equation*}
			\norm{\tilde \Phi^{k}_1(\hat \bu^k)}^p_{L^p(\hat E^2_h(\hat S^k_\bi))}
		\le
		\iint_{\hat E^2_h(\hat S^k_\bi)} \abs{\hx_1-\tilde x_1}^{p-1} \int_{\tilde x_1}^{\hx_1} \abs{\hat \bu^k(z,\hx_2)}^p \dd z \dd \hbx
		\le 
		 \hat h_k^{p}\norm{\hat \bu^k}^p_{L^p(\hat E^2_h(\hat S^k_\bi))}
	\end{equation*} 
	so that we have shown
	\begin{equation*}
	\norm{\hat \nabla_1 \hat\Pi^0_k \hat \Phi^{k}_1(\hat \bu^k)}_{L^p(\hat S^k_\bi)} 
	= \norm{\hat \nabla_1 \hat\Pi^0_k \tilde \Phi^{k}_1(\hat \bu^k)}_{L^p(\hat S^k_\bi)} 
	\lesssim \norm{\hat \bu^k}_{L^p(\hat E^2_h(\hat S^k_\bi))}.
	\end{equation*}
	Estimate \eqref{loc_bound_Pik1} then follows from the scaling 
	\eqref{L2_scale_pf} of 1-form pullbacks.
\end{proof}

\subsection{Edge-based antiderivative operators}
\label{sec:Phi_e}

In a similar way, we define edge-based antiderivative operators $\Phi^e_d$
along $d \in \{\parallel, \perp\}$, the parallel and perpendicular 
directions relative to $e$.
Both are supported in the patches adjacent to $e$ 
and the construction is summarized in Figure~\ref{fig:Phi_e}.

\begin{figure}
	\centering 
 	\subfloat[parallel case]
	 {
		\fontsize{18pt}{22pt}\selectfont
	 \resizebox{!}{0.2\textheight}{
\begingroup%
  \makeatletter%
  \providecommand\color[2][]{%
    \errmessage{(Inkscape) Color is used for the text in Inkscape, but the package 'color.sty' is not loaded}%
    \renewcommand\color[2][]{}%
  }%
  \providecommand\transparent[1]{%
    \errmessage{(Inkscape) Transparency is used (non-zero) for the text in Inkscape, but the package 'transparent.sty' is not loaded}%
    \renewcommand\transparent[1]{}%
  }%
  \providecommand\rotatebox[2]{#2}%
  \newcommand*\fsize{\dimexpr\f@size pt\relax}%
  \newcommand*\lineheight[1]{\fontsize{\fsize}{#1\fsize}\selectfont}%
  \ifx\svgwidth\undefined%
    \setlength{\unitlength}{315.62072754bp}%
    \ifx\svgscale\undefined%
      \relax%
    \else%
      \setlength{\unitlength}{\unitlength * \real{\svgscale}}%
    \fi%
  \else%
    \setlength{\unitlength}{\svgwidth}%
  \fi%
  \global\let\svgwidth\undefined%
  \global\let\svgscale\undefined%
  \makeatother%
  \begin{picture}(1,0.57894737)%
    \lineheight{1}%
    \setlength\tabcolsep{0pt}%
    \put(0,0){\includegraphics[width=\unitlength,page=1]{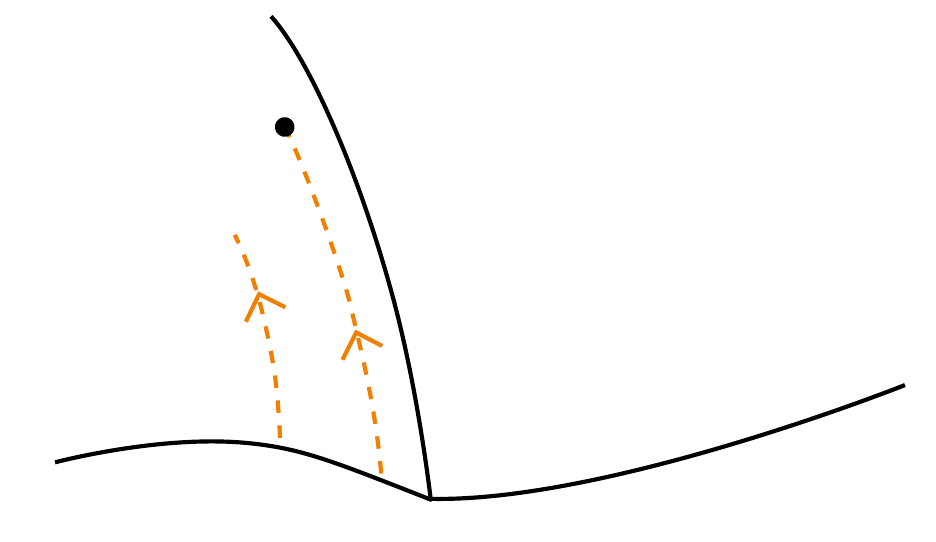}}%
    \put(0.21530711,0.48250721){\color[rgb]{0,0,0}\makebox(0,0)[rt]{\lineheight{1.25}\smash{\begin{tabular}[t]{r}$\boxed{k^-(e)}$\end{tabular}}}}%
    \put(0.21065096,0.35343452){\color[rgb]{0,0,0}\makebox(0,0)[rt]{\lineheight{1.25}\smash{\begin{tabular}[t]{r}$\bx$\end{tabular}}}}%
    \put(0.48827395,0.48159412){\color[rgb]{0,0,0}\makebox(0,0)[lt]{\lineheight{1.25}\smash{\begin{tabular}[t]{l}$\boxed{k^+(e)}$\end{tabular}}}}%
    \put(0,0){\includegraphics[width=\unitlength,page=2]{plots/Phi_e_par.pdf}}%
    \put(0.34669455,0.52631911){\color[rgb]{0,0,0}\makebox(0,0)[lt]{\lineheight{1.25}\smash{\begin{tabular}[t]{l}$e$\end{tabular}}}}%
    \put(0,0){\includegraphics[width=\unitlength,page=3]{plots/Phi_e_par.pdf}}%
  \end{picture}%
\endgroup%
}
		\label{fig:type_I}
		}
		\\
		\subfloat[perpendicular case]	
		{
			\fontsize{18pt}{22pt}\selectfont
      \resizebox{!}{0.2\textheight}{
\begingroup%
  \makeatletter%
  \providecommand\color[2][]{%
    \errmessage{(Inkscape) Color is used for the text in Inkscape, but the package 'color.sty' is not loaded}%
    \renewcommand\color[2][]{}%
  }%
  \providecommand\transparent[1]{%
    \errmessage{(Inkscape) Transparency is used (non-zero) for the text in Inkscape, but the package 'transparent.sty' is not loaded}%
    \renewcommand\transparent[1]{}%
  }%
  \providecommand\rotatebox[2]{#2}%
  \newcommand*\fsize{\dimexpr\f@size pt\relax}%
  \newcommand*\lineheight[1]{\fontsize{\fsize}{#1\fsize}\selectfont}%
  \ifx\svgwidth\undefined%
    \setlength{\unitlength}{315.62072754bp}%
    \ifx\svgscale\undefined%
      \relax%
    \else%
      \setlength{\unitlength}{\unitlength * \real{\svgscale}}%
    \fi%
  \else%
    \setlength{\unitlength}{\svgwidth}%
  \fi%
  \global\let\svgwidth\undefined%
  \global\let\svgscale\undefined%
  \makeatother%
  \begin{picture}(1,0.57894737)%
    \lineheight{1}%
    \setlength\tabcolsep{0pt}%
    \put(0,0){\includegraphics[width=\unitlength,page=1]{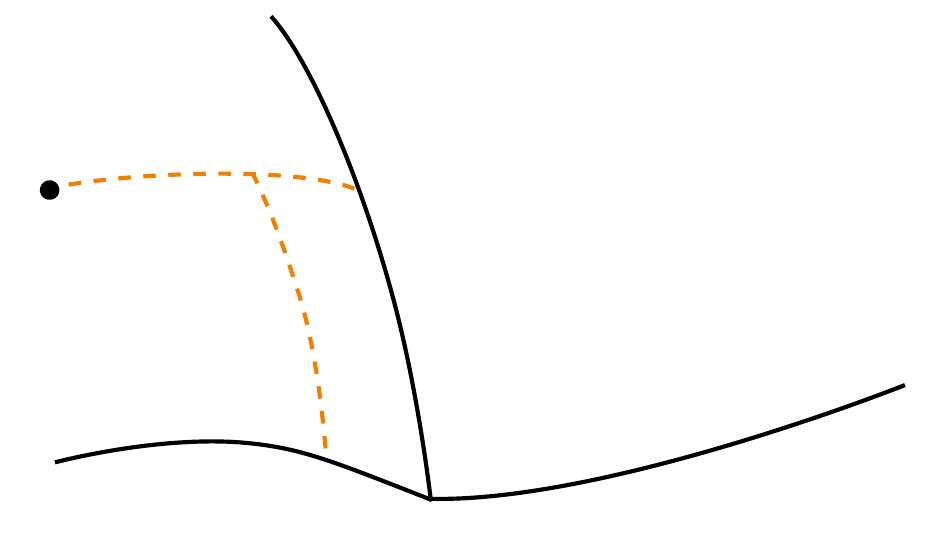}}%
    \put(0.2554971,0.01572443){\color[rgb]{0,0,0}\makebox(0,0)[lt]{\lineheight{1.25}\smash{\begin{tabular}[t]{l}$\gamma^e_{\perp, *}(a)$\end{tabular}}}}%
    \put(0,0){\includegraphics[width=\unitlength,page=2]{plots/Phi_e_perp.pdf}}%
    \put(0.21530711,0.48250721){\color[rgb]{0,0,0}\makebox(0,0)[rt]{\lineheight{1.25}\smash{\begin{tabular}[t]{r}$\boxed{k^-(e)}$\end{tabular}}}}%
    \put(0,0){\includegraphics[width=\unitlength,page=3]{plots/Phi_e_perp.pdf}}%
    \put(0.73151952,0.26372172){\color[rgb]{0,0,0}\makebox(0,0)[lt]{\lineheight{1.25}\smash{\begin{tabular}[t]{l}$\bx$\end{tabular}}}}%
    \put(0,0){\includegraphics[width=\unitlength,page=4]{plots/Phi_e_perp.pdf}}%
    \put(0.48827395,0.48159412){\color[rgb]{0,0,0}\makebox(0,0)[lt]{\lineheight{1.25}\smash{\begin{tabular}[t]{l}$\boxed{k^+(e)}$\end{tabular}}}}%
    \put(0,0){\includegraphics[width=\unitlength,page=5]{plots/Phi_e_perp.pdf}}%
    \put(0.34669455,0.52631911){\color[rgb]{0,0,0}\makebox(0,0)[lt]{\lineheight{1.25}\smash{\begin{tabular}[t]{l}$e$\end{tabular}}}}%
  \end{picture}%
\endgroup%
}
    	\label{fig:type_II}
		}
	\caption{Integration paths $\gamma^e_d(\bx,a)$ defining the edge-based 
	antiderivative operators $\Phi^e_d$. In the case $d=\parallel$ (\textsc{a})
	the curves connect various points $\bx$ to different starting points $\gamma^e_{\parallel,*}(\bx)$, represented by white squares, which depend only on the perpendicular
	component $\hx^k_\perp$ of the logical coordinate of $\bx \in \Omega_k$,
	see \eqref{ppe}.
	In the case $d=\perp$ (\textsc{b}) the curves depend 
	on the averaging parameter $a$.
	For a given value of $a \in (0,\hat h_e)$
	they connect every $\bx \in \Omega(e)$ to the same 
	starting point $\gamma^e_{\perp,*}(a)$
	represented by a white square, see \eqref{gam*eperp}.
	}
	\label{fig:Phi_e}
\end{figure}

For the parallel edge-based antiderivative operator $\Phi^e_\parallel$ 
the integration curve is similar to the single-patch one.
For $\bx \in \Omega_k$, $k \in \cK(e)$, it is fully supported in 
$\Omega_k$. Writing now $\hbx^k = F_k^{-1}(\bx)$, we define it as
\begin{equation*}
\gamma^{e,k}_\parallel(\bx) = F_k(\hat \gamma^{e,k}_\parallel(\hat \bx, [\eta_e^k(0), \hx^k_\parallel]))
\quad \text{ with } \quad
\hat \gamma^{e,k}_\parallel(\hat \bx, \cdot):  
	[\eta_e^k(0), \hx^k_\parallel] \ni z \mapsto \hat X^k_e(z, \hx^k_\perp)
\end{equation*}
where we remind that $\hx^k_\parallel$ and $\hx^k_\perp$ are the parallel and perpendicular
coordinates relative to $e$, 
$\hat X^k_e$ is the reordering function \eqref{Xke}
and $\eta_e^k$ is the edge orientation function given by \eqref{eta}.
Thus, the parallel antiderivative reads
\begin{equation} \label{Phi_par}
\Phi^{e}_{\parallel}(\bu)(\bx)
	\coloneqq
	\int_{\eta_e^k(0)}^{\hx^k_\parallel} \hat u^k_\parallel(X^k_e(z, \hx^k_\perp)) \dd z
	\quad \text{ for } \bx \in \Omega_k, \quad k \in \cK(e).
\end{equation}
For the perpendicular edge-based antiderivative we use an averaging step
\begin{equation} \label{Phi_perp}
	\Phi^{e}_{\perp}(\bu)(\bx) = \frac{1}{\hat h_e}\int_0^{\hat h_e} \Phi^{e}_{\perp,a}(\bu)(\bx) \dd a
\end{equation}
where $\hat h_e$ is defined in \eqref{he},
and for each value of the parameter $a$, 
$\gamma(\bx,a)$ is a curve defined as follows: 
For $\bx$ on the coarser patch $k = k^-(e)$ the curve
is contained in that patch, and
$\hat \gamma^{e,k}_\perp(\hat \bx, a, \cdot)$ is defined on 
$[-\hx_\parallel, \abs{\hx_\perp - \tilde a}]$ by  
\begin{equation} \label{gamma_perp_-}
	\hat \gamma^{e,-}_\perp(\hat \bx, a, \cdot): z \mapsto \begin{cases}
		X^-_e(\hx_\parallel+z, \tilde a) \quad &\text{for } ~ -\hx_\parallel \le z \le 0
		\\
		X^-_e(\hx_\parallel,\tilde a + z \sign(\hx_\perp - \tilde a)) \quad &\text{for } ~ 0 \le z \le \abs{\hx_\perp - \tilde a}.
	\end{cases}
\end{equation}
Here $\tilde a$ is the perpendicular coordinate at distance $a$ from the edge, namely
\begin{equation} \label{tilde_a}
	\tilde a \coloneqq \begin{cases}
		a &\text{ if } \hat e^-_\perp = 0
		\\
		1-a &\text{ if } \hat e^-_\perp = 1
	\end{cases}.
\end{equation}
On the coarse patch, i.e. $\bx \in \Omega_{k^-(e)}$, the resulting antiderivative reads then
\begin{equation} \label{Phi_perp_-}
\Phi^{e,-}_{\perp,a}(\bu)(\bx)
	\coloneqq
	\int_{0}^{\hx_\parallel} \hat u^-_\parallel(\hat X^-_e(z_\parallel, \tilde a)) \dd z_\parallel
	+ \int_{\tilde a}^{\hx_\perp} \hat u^-_\perp(\hat X^-_e(\hx_\parallel,z_\perp)) \dd z_\perp.
\end{equation}
Next for $\bx \in \Omega_{k^+(e)}$ we define the curve in two pieces: 
one that connects $\bx$ to its projection $\bp_e(\bx)$ on the edge, see \eqref{pex},
and another that corresponds to the coarse curve 
$\hat \gamma^{e,-}_\perp$ on that same point.
This amounts to summing 
\begin{equation} \label{Phi_perp_+}
\Phi^{e,+}_{\perp,a}(\bu)(\bx) \coloneqq \delta \Phi^{e,+}_{\perp}(\bu)(\bx) + \Phi^{e,-}_{\perp,a}(\bu)(\bp_e(\bx)) 
\end{equation}
where the first term is a path integral in the fine patch $\Omega_{k^+(e)}$,
\begin{equation} \label{dPhi_perp}
\delta \Phi^{e,+}_\perp(\bu)(\bx) \coloneqq 
\int_{\hat e^+_\perp}^{\hx_\perp} \hat u^+_\perp(X^+_e(\hx_\parallel, z_\perp)) \dd z_\perp 
\end{equation}
which corresponds to the local curve 
	$\hat \gamma^{e,+}_\perp(\hat \bx, \cdot): [\hat e^+_\perp, \hx_\perp] \ni z \mapsto 
		X^+_e(\hx_\parallel, z),
	$
and the second term is the antiderivative on the coarser patch $k^-(e)$,
evaluated at $\bp_e(\bx)$ given by \eqref{pex}, i.e.
\begin{equation} \label{foot_curve}
	\bp_e(\bx) = X^+_e(\hx_\parallel, \hat e^+_\perp) =
	X^-_e(\eta_e(\hx_\parallel), \hat e^-_\perp)
\end{equation}
This two-term definition thus corresponds to an integration path that,
for each value of the parameter $a$, connects every point $\bx \in \Omega_k$, 
$k \in \cK(e)$, to the common starting point 
\begin{equation} \label{gam*eperp}
	\gamma^e_{\perp,*}(a) = F_{k^-(e)}\big(X^-_e(0,\tilde a))\big)
\end{equation}
where we remind that $\tilde a$ is given by \eqref{tilde_a}.

For boundary edges the above construction yields 
two types of curves $\gamma^e_{\perp}(\bx, a)$: in the inhomogeneous case
\eqref{dR} where boundary patches are by convention of coarse type $k^-(e)$,
all the curves have a common starting point \eqref{gam*eperp} inside the domain
$\Omega$. The situation is different in the homogeneous case
\eqref{dR0} where boundary patches are by convention of fine type $k^+(e)$:
in this case the curves $\gamma^e_{\perp}(\bx, a)$ are all perpendicular to 
the boundary edge in the logical variables. In particular, they have many different starting points, which all lie on the boundary $\partial \Omega$.

We then have the following result.
\begin{lemma} \label{lem:Phi_e_nabla}
Let $e \in \cE$. For all $\bu = \nabla \phi$ with $\phi \in C^1(\Omega)$,
there is a function $\tilde \phi_e$ such that
\begin{equation} \label{Phi_par_nabla}
	\Phi^{e}_{\parallel}(\bu)(\bx) = \phi(\bx) - \tilde \phi_e(\bx)
	\quad \text{ and } \quad \nabla^e_\parallel \tilde \phi_e = 0
\end{equation}
holds on $\Omega(e) \coloneqq \cup_{k \in \cK(e)}\Omega_k$. 
Moreover for interior edges and boundary edges in the 
inhomogeneous case \eqref{dR}, there is a constant $\bar \phi_e$ such that
\begin{equation} \label{Phi_perp_nabla} 
	\Phi^{e}_{\perp}(\bu)(\bx) = \phi(\bx) - \bar \phi_e	
\end{equation}
holds on $\Omega(e)$. Finally, if $e$ is a boundary edge in the 
homogeneous case \eqref{dR0} and $\bu = \nabla \phi$ with 
$\phi \in C^1_0(\Omega)$, then it holds
\begin{equation} \label{Phi_perp_nabla_hom} 
	\Phi^{e}_{\perp}(\bu)(\bx) = \phi(\bx).
\end{equation}
\end{lemma}

\begin{proof}  
	For $\bu = \nabla \phi$ the parallel antiderivative 
	is the path integral of a gradient. Specifically, \eqref{Phi_par} yields 
	\eqref{Phi_par_nabla} with 
	$\tilde \phi_e(\bx) = \phi(X^k_e(\eta_e^k(0),\hx^k_\perp)$
	for $\bx \in \Omega_k$, $k \in \cK(e)$.
	For all $a \in [0,\hat h_e]$ the perpendicular antiderivative is also the 
	path integral of a gradient,
	hence for interior edges and boundary edges in the 
	inhomogeneous case, we have
	$\Phi^{e}_{\perp,a}(\bu)(\bx) = \phi(\bx) - \phi(\gamma^e_{\perp,*}(a))$
	where $\gamma^e_{\perp,*}(a)$ is the starting point defined in \eqref{gam*eperp}.
	The result follows from the averaging formula \eqref{Phi_perp}, with 
	$\bar \phi_e = \frac{1}{\hat h_e}\int_0^{\hat h_e} \phi(\gamma^e_{\perp,*}(a)) \dd a$.
	For boundary edges in the homogeneous case, the result follows from the fact that 
	all curves $\gamma^e_{\perp}(\bx,a)$ start from the boundary $\partial \Omega$
	where $\phi =0$.
\end{proof}

The following lemma states that these edge antiderivative operators 
are locally stable in $L^p$, and so are the resulting edge-correction terms
in \eqref{Pi1_intro}.

\begin{lemma} \label{lem:stab_e_corr}
	On a patch-wise mapped Cartesian domain 
	of the form
	\begin{equation} \label{omega_e}
		\omega_e = \cup_{k \in \cK(e)} F_{k}(\hat \omega^k_e)
		\quad \text{ with } \quad 
			\hat \omega^k_e
					= \hat \Omega \cap \hat X^k_e\big([0,1] \times 
						[\hat e^k_\perp - \rho\hat h_e,\hat e^k_\perp + \rho\hat h_e]\big)		
	\end{equation} 
	and
	$ 
	1 \le \rho \lesssim 1,
	$
	we have 
	\begin{equation} \label{L2_bound_Phi_e}
		\norm{\Phi^{e}_{d}(\bu)}_{L^p(\omega_e)} \lesssim \norm{\bu}_{L^p(\omega_e)}.
	\end{equation}
	Moreover, on a local domain $S^{e}_j$
	of the form \eqref{Sei}, $j \in \{0, \dots, n_e\}$,
	the bound
	\begin{equation} \label{loc_bound_corr_e}	
	\norm{\nabla^e_d (P^e - I^e)\Pi^0_\pw \Phi^{e}_{d}(\bu)}_{L^p(S^{e}_j)} 
	\lesssim \norm{\bu}_{L^p(E_e^3(S^{e}_j))}
	\end{equation}
	holds, where $E_e^3$ corresponds to a three-fold application of the edge-based domain extension operator \eqref{Ee_ext}.
\end{lemma}

\begin{proof}
	The bound \eqref{L2_bound_Phi_e} for 
	the parallel direction $(d = \parallel)$ is proven just like \eqref{L2_bound_Phikd}
	since the parallel antiderivative operator coincides with the single-patch 
	\eqref{Phi_kd_exp} along the parallel direction, 
	up to a possible change in the curve starting point and orientation.	
	For the perpendicular direction $(d=\perp)$, we first consider 
	the $k^-$ patch and again work with the pullback 
	$
	\hat \Phi^{e,-}_{\perp}(\hat \bu^-)(\hbx) \coloneqq \Phi^{e,-}_{\perp}(\bu)(\bx)
	$
	where $\hat \bu^- \coloneqq (\cF_{k^-} )^{-1}\bu$.
	For simplicity, we assume an orientation corresponding to $\hat X^{-}_e = I$,
	i.e., $\hbx = 
	(\hx_\parallel,\hx_\perp)$
	and $\hat e^-_\perp = 0$, i.e. $\tilde a = a$. 
	Using H\"older 
	inequalities we compute
  \begin{equation} \label{CS_Phi_ep-}
		\begin{aligned}
		\norm{\hat \Phi^{e,-}_{\perp}(\hat \bu^-)}^p_{L^p(\hat \omega^-_e)}
		&=  \iint_{\hat \omega^-_e} \frac{1}{\hat h_e^p} \Bigabs{
		\int_0^{\hat h_e} 
		\Big(\int_{0}^{\hx_\parallel} \hat u^-_\parallel(z_\parallel, a) \dd z_\parallel
		+
		\int_{a}^{\hx_\perp} \hat u^-_\perp(\hx_\parallel,z_\perp) \dd z_\perp
		\Big)\dd a
		}^p \dd \hbx
		\\
		&\lesssim 
		\iint_{\hat \omega^-_e} \frac{\abs{\hx_\parallel}^{p-1}}{\hat h_e} \int_0^{\hat h_e} 
		\int_{0}^{\hx_\parallel} 
					\abs{\hat u^-_\parallel(z_\parallel, a)}^p \dd z_\parallel  \dd a\dd \hbx
			\\
			& \mspace{100mu} + 
			\iint_{\hat \omega^-_e} \frac{\abs{\hx_\perp-a}^{p-1}}{\hat h_e} \int_0^{\hat h_e} 
			\int_{a}^{\hx_\perp}
					\abs{\hat u^-_\perp(\hx_\parallel,z_\perp)}^p \dd z_\perp
			 \dd a\dd \hbx
		\\
		&\lesssim 
		\rho 
					\norm{\hat u^-_\parallel}^p_{L^p(\hat \omega^-_e)}
		 + (\rho \hat h_e)^p
								\norm{\hat u^-_\perp}^p_{L^p(\hat \omega^-_e)}
		\lesssim 
									\norm{\hat \bu^-}^p_{L^p(\hat \omega^-_e)}.
		\end{aligned}
  \end{equation}
  The scaling relations \eqref{L2_scale_pf} 
  for 0-form and 1-form pullbacks yield then 
  \begin{equation} \label{bound_Phi_perp-}
	  \norm{\Phi^{e,-}_{\perp}(\bu)}_{L^p(\omega^-_e)}
	  \sim H_{k^-}^{2/p} \norm{\hat \Phi^{e,-}_{\perp}(\hat \bu^-)}_{L^p(\hat \omega^-_e)}
	  \lesssim 
		  H_{k^-}^{2/p} \norm{\hat \bu^-}_{L^p(\hat \omega^-_e)}
	  \lesssim \norm{\bu}_{L^p(\omega_e)}.		
  \end{equation} 
	On the $k^+$ patch we also assume an orientation corresponding to 
	$\hat X^{+}_e = I$, i.e. 
	$\hbx = 
	(\hx_\parallel,\hx_\perp)$,
	and $\hat e^+_\perp = 0$.
	We first consider the pullback of the integral term 
	$\delta \Phi^{e,+}_\perp$
	\eqref{Phi_perp_+}, and compute
	\begin{equation} \label{CS_dPhi_perp_+}
		\begin{aligned}
		\norm{\widehat {\delta \Phi}^{e,+}_\perp(\hat \bu^+)}^p_{L^p(\hat \omega^+_e)}
		&= 		
		\iint_{\hat \omega^+_e} \Bigabs{
			\int_{0}^{\hx_\perp} \hat u^+_\perp(\hx_\parallel, z_\perp) 
				\dd z_\perp 
		}^p \dd \hbx
		\\
		&\le \abs{\rho \hat h_e}^{p-1} \iint_{\hat \omega^+_e}
				\int_{0}^{\hx_\perp} \abs{\hat u^+_\perp(\hx_\parallel, z_\perp) }^p
				\dd z_\perp 
			\dd \hbx
		\\
		&\le \abs{\rho \hat h_e}^p \norm{\hat u^+_\perp}^p_{L^p(\hat \omega^+_e)}
		\lesssim \norm{\hat \bu^+}^p_{L^p(\hat \omega^+_e)}.
		\end{aligned}
	\end{equation}
	We next write 
	$\hat \Phi^{e,*}(\hat \bu^-)(\hbx) 
		\coloneqq 
		\Phi^{e,-}_{\perp}(\bu)(\bp_e(\bx))$,
	with $\bx = F_{k^+}(\hbx)$,
	the pullback of the 
	coarse matching term in \eqref{Phi_perp_+}.
	With our simple orientation 
	the matching point \eqref{foot_curve} is
	$\bp_e(\bx) \coloneqq F_{k^+}(\hx_\parallel,0) = F_{k^-}(\eta_e(\hx_\parallel),0)$.
	Without loss of generality we further assume the same orientation: $\eta_e(z) = z$
	so that
	$\hat\omega^-_e = \hat\omega^+_e$.
	We now have	
	\begin{equation} \label{CS_Phi_perp_*}
		\begin{aligned}
		\norm{\hat \Phi^{e,*}(\hat \bu^-)}^p_{L^p(\hat \omega^+_e)}
		&= 		
		\iint_{\hat \omega^+_e} \abs{\hat \Phi^{e,-}_{\perp}(\hat \bu^-)(\hx_\parallel,0)}^p
		\dd \hbx
		\\
		&=
		\iint_{\hat \omega^-_e} \frac{1}{\hat h_e^p} \Bigabs{
		\int_0^{\hat h_e} 
		\Big(\int_{0}^{\hx_\parallel} \hat u^-_\parallel(z_\parallel, a) \dd z_\parallel
		+
		\int_{a}^{0} \hat u^-_\perp(\hx_\parallel,z_\perp) \dd z_\perp
		\Big)\dd a
		}^p \dd \hbx		
		\\
		&\lesssim
		\iint_{\hat \omega^-_e} \frac{\abs{\hx_\parallel}^{p-1}}{\hat h_e} \int_0^{\hat h_e} 
		\int_{0}^{\hx_\parallel} 
					\abs{\hat u^-_\parallel(z_\parallel, a)}^p \dd z_\parallel  \dd a\dd \hbx
			\\
			& \mspace{100mu} + 
			\iint_{\hat \omega^-_e} \int_0^{\hat h_e} 
			\int_{a}^{0}
					\abs{\hat u^-_\perp(\hx_\parallel,z_\perp)}^p \dd z_\perp
			 \dd a\dd \hbx
			 \\
		&\lesssim
			\rho
							\norm{\hat u^-_\parallel}^p_{L^p(\hat \omega^-_e)}
				 + (\rho \hat h_e)^p
										\norm{\hat u^-_\perp}^p_{L^p(\hat \omega^-_e)}
		\lesssim \norm{\hat \bu^-}^p_{L^p(\hat \omega^-_e)}.
		\end{aligned}
	\end{equation}
	With \eqref{CS_dPhi_perp_+} and the scaling relations
	\eqref{L2_scale_pf}, this bound yields 
	\begin{equation} \label{bound_Phi_perp+}
		\begin{aligned}
			\norm{\Phi^{e,+}_{\perp}(\bu)}_{L^p(\omega^+_e)} 
			&\le \norm{\delta \Phi^{e,+}_{\perp}(\bu)}_{L^p(\omega^+_e)} 
					+ \norm{\Phi^{e,-}_{\perp}(\bu)(\bp_e(\cdot))}_{L^p(\omega^+_e)} 
			\\
			&\lesssim H_{k^+}^{2/p}
			\big(\norm{\widehat {\delta \Phi}^{e,+}_\perp(\hat \bu^+)}_{L^p(\hat \omega^+_e)}
			+\norm{\hat \Phi^{e,*}(\hat \bu^-)}_{L^p(\hat \omega^+_e)}\big)
			\\
			&\lesssim H_{k^+}^{2/p}
			(\norm{\hat \bu^+}_{L^p(\hat \omega^+_e)} + \norm{\hat \bu^-}_{L^p(\hat \omega^-_e)})
			\lesssim \norm{\bu}_{L^p(\omega_e)}.
		\end{aligned}
	\end{equation}
	Together with \eqref{bound_Phi_perp-}, this proves the stability 
	\eqref{L2_bound_Phi_e} in the perpendicular case.
	
	Turning to the local bound \eqref{loc_bound_corr_e},
	we observe that the inverse estimate \eqref{inverse_est} and the local stability of 
	$P^e$, $I^e$ and $\Pi^0_\pw$, see \eqref{stab_P_I_e} and \eqref{stab_Pi0k}, 
	also hold with edge-based domain extensions $E_e$ according
	to \eqref{Ee_ext}.
	This allows us to write
	\begin{equation} \label{est_1}
	\begin{aligned}
		\norm{\nabla^e_d (P^e - I^e)\Pi^0_\pw \Phi^{e}_{d}(\bu)}_{L^p(S^{e}_j)} 
		&\lesssim 
		h_e^{-1} \norm{(P^e - I^e)\Pi^0_\pw \Phi^{e}_{d}(\bu)}_{L^p(E_e(S^{e}_j))} 
		\\
		&\lesssim 
		h_e^{-1} \norm{\Pi^0_\pw \Phi^{e}_{d}(\bu)}_{L^p(E^2_e(S^{e}_j))} 
		\\
		&\lesssim 
		h_e^{-1} \norm{\Phi^{e}_{d}(\bu)}_{L^p(E^3_e(S^{e}_j))}		
	\end{aligned}
	\end{equation}
	so that \eqref{loc_bound_corr_e} would follow from a bound like
	$\norm{\Phi^{e}_{d}(\bu)}_{L^p(E^3_e(S^{e}_j))} \lesssim h_e\norm{\bu}_{L^p(E^3_e(S^{e}_j))}$.
	A difficulty is that this property cannot hold a priori, 
	indeed both antiderivative operators rely on integration curves that 
	are not localized in a domain of the form $E^3_e(S^{e}_j)$.
	Therefore, a localizing argument is needed. For the parallel term 
	$(d = \parallel)$ we can use a similar argument as the one that 
	we used to prove \eqref{loc_bound_Pik1} for the single-patch antiderivative:
	indeed one may again change the integration constant in 
	$\Phi^{e}_{\parallel}(\bu)$, without changing
	the function $\nabla^e_\parallel (P^e - I^e)\Pi^0_\pw \Phi^{e}_{\parallel}(\bu)$:
	here this is made possible because the invariance with respect
	to the parallel variable is preserved not only by $\Pi^0_\pw$ but also by $P^e$ and $I^e$, 
	see Lemma~\ref{lem:parinv}.
	As a result one can define a localized antiderivative
	\begin{equation*}
	\tilde \Phi^{e}_{\parallel}(\bu)(\bx)
		=
		\int_{\eta_e^k(\tilde x^k_j)}^{\hx^k_\parallel} \hat u^k_\parallel(X^k_e(z, \hx^k_\perp)) \dd z
	\end{equation*}
	with $\tilde x^k_j \in \hat S^{-}_j$ a curvilinear coordinate corresponding to the edge piece 
	$e \cap S^{e}_j$: 
	by Lemma~\ref{lem:parinv} we have
	$\nabla^e_\parallel (P^e - I^e)\Pi^0_\pw \Phi^{e}_{\parallel}(\bu) 
		= \nabla^e_\parallel (P^e - I^e)\Pi^0_\pw \tilde \Phi^{e}_{\parallel}(\bu)$
	and a local estimate for this antiderivative (derived exactly in the same way as for the 
	single-patch antiderivative) gives 
	$\norm{\tilde \Phi^{e}_{d}(\bu)}_{L^p(E^3_e(S^{e}_j))} \le h_e\norm{\bu}_{L^p(E^3_e(S^{e}_j))}$.
	This shows that the local bound \eqref{loc_bound_corr_e} holds indeed for the parallel term.

	For the perpendicular term we cannot use a similar localizing argument, as none of
	the projection operators $P^e$ or $I^e$ preserve an invariance along the 
	perpendicular direction.
	Fortunately our design for the integration curves involved in $\Phi^{e}_\perp$ 
	yields the following localizing property: 
	if $\bu = 0$ on the extended domain $E^2_e(S^{e}_j)$, then 
	\begin{equation} \label{loc_corr_perp}
		\bu = 0 ~~ \text{ on } ~~ E^2_e(S^e_j)
		\quad \implies \quad 
		(P^e - I^e)\Pi^0_\pw \Phi^{e}_\perp(\bu) = 0 ~~ \text{ on } ~~ E_e(S^e_j).	
	\end{equation}
	To establish this property we assume for simplicity that the edge $e$ has the 
	same orientation in both patches, i.e. $\eta^+_e(x_\parallel) = x_\parallel$,
	and recall that $E^2_e(S^e_j)$ is Cartesian on both patches,
	with parallel coordinate in the same interval, see \eqref{Ee_ext}. 
	Let us denote by $\alpha_\parallel$ the minimal parallel coordinate
	in both patches.
	We then see that for all $a \in [0,\hat h_e]$ and all $\bx \in E^2_e(S^{e}_j)$, 
	its parallel coordinate satisfies $\hx_\parallel \ge \alpha_\parallel$ and
	the curve $\gamma = \gamma^e_\perp(\bx)$ is made of two connected parts: 
	a first part $\Gamma^e_1(\bx,a,j)$ 
	with parallel coordinate $\hat \gamma_\parallel \le \alpha_\parallel$
	(and included in the coarse cell $\Omega_{k^-}$), 
	and a second part $\Gamma^e_2(\bx,a,j)$ with parallel coordinate 
	$\hat \gamma_\parallel > \alpha_\parallel$. 
	Because $\abs{\tilde a - \hat e^k_\perp} = a \le \hat h_e$, see \eqref{tilde_a}, 
	this latter part is included in $E^2_e(S^{e}_j)$ while the
	first part $\Gamma^e_1(\bx,a,j)$ is fully outside.
	Moreover, we observe that for all $\bx \in E^2_e(S^{e}_j)$ this first part is 
	independent of $\bx$: $\Gamma^e_1(\bx,a,j) = \Gamma^e_1(a,j)$.
	As a consequence we find that if $\bu$ vanishes on $E^2_e(S^{e}_j)$, 
	then the antiderivative takes the form 
	\begin{equation*}
	\Phi^{e}_{\perp,a}(\bu)(\bx) = \int_{\Gamma^e_1(a,j) \cup \Gamma^e_2(\bx,a,j)} \bu \cdot \dd l 
	 	= \int_{\Gamma^e_1(a,j)} \bu \cdot \dd l
	\end{equation*}
	which is a constant (say, $C(\bu,a,j)$) on $E^2_e(S^{e}_j)$.
	According to Remark~\ref{rem:stab_Pi0_ev} this shows that 
	$\Pi^0_\pw\Phi^{e}_{\perp,a}(\bu) = C(\bu,a,j)$ on $E_e(S^e_j)$.
	Using next Lemma~\ref{lem:0_corr_e} and noting that for all $\bx \in E_e(S^e_j)$ 
	the projected point $\bp_e(\bx)$ on the interface is also in $E_e(S^e_j)$,
	we find that
	$(P^e - I^e)\Pi^0_\pw \Phi^{e}_{\perp,a}(\bu) = 0$ on $E_e(S^e_j)$,
	and the property \eqref{loc_corr_perp} follows 
	by integration over $a \in [0,\hat h_e]$.
	For a general $\bu \in L^p(\Omega)$ decomposed
	as $\bu = \bu \one_{E^2_e(S^{e}_j)} + \bu (1-\one_{E^2_e(S^{e}_j)})$,
	this allows to write (by linearity)
	\begin{equation*}
	( P^e - I^e)\Pi^0_\pw \Phi^{e}_{\perp}(\bu)(\bx) 
		= ( P^e - I^e)\Pi^0_\pw \Phi^{e}_{\perp}(\bu \one_{E^2_e(S^{e}_j)})(\bx)
	\quad \text{ for } \bx \in E_e(S^e_j)
	\end{equation*}
	and to bound
	\begin{equation} \label{est_2}
		\begin{aligned}
			\norm{(P^e - I^e)\Pi^0_\pw \Phi^{e}_{\perp}(\bu)}_{L^p(E_e(S^{e}_j))}
			&= \norm{(P^e - I^e)\Pi^0_\pw \Phi^{e}_{\perp}(\bu \one_{E^2_e(S^{e}_j)})}_{L^p(E_e(S^{e}_j))}
			\\
			&	\lesssim \norm{\Pi^0_\pw \Phi^{e}_{\perp}(\bu \one_{E^2_e(S^{e}_j)})}_{L^p(E^2_e(S^{e}_j))}		
			\\
			&	\lesssim \norm{\Phi^{e}_{\perp}(\bu \one_{E^2_e(S^{e}_j)})}_{L^p(E^3_e(S^{e}_j))}		
		\end{aligned}
	\end{equation}
	by using the local $L^p$ stability of $P^e$, $I^e$ and $\Pi^0_\pw$, see \eqref{stab_P_I_e} which also holds
	with edge-based extension domains as observed above, and \eqref{stab_Pi0_ev}.
	To complete the proof we next observe that in the last antiderivative 
	the integration curve $\gamma^e_\perp(\bx,a)$ can be restricted to $E^2_e(S^{e}_j)$
	which is of diameter $\lesssim h_e$.
	By repeating the steps in \eqref{CS_Phi_ep-} and \eqref{CS_Phi_perp_*} 
	with such a localized integration over $z$, 
	we find 
	\begin{equation*}
	\norm{\Phi^{e}_{\perp}(\bu \one_{E^2_e(S^{e}_j)})}_{L^p(E^3_e(S^{e}_j))} 
	\lesssim h_e \norm{\bu \one_{E^2_e(S^{e}_j)}}_{L^p(E^3_e(S^{e}_j))}
	\lesssim h_e \norm{\bu}_{L^p(E^3_e(S^{e}_j))}.
	\end{equation*}
	Together with \eqref{est_1} and \eqref{est_2} 
	this proves the local bound \eqref{loc_bound_corr_e}
	for the perpendicular term, and completes the proof. 
\end{proof}

\subsection{Vertex-based antiderivative operators}
\label{sec:Phi_v}

The vertex-based antiderivative $\Phi^\bv(\bu)$
is defined on the patches contiguous to $\bv$ 
in a similar way as the perpendicular edge-based antiderivative 
$\Phi^e_\perp$ from Section~\ref{sec:Phi_e}.
Like the latter it involves an averaging step
\begin{equation} \label{Phi_v}
\Phi^\bv(\bu) = \frac{1}{\hat h_\bv} 
	\int_{0}^{\hat h_\bv} \Phi^\bv_a(\bu) \dd a,
	\qquad
	\Phi^\bv_a(\bu)(\bx) = \int_{\gamma^\bv(\bx,a)} \bu \cdot \dd l
\end{equation}
with $\hat h_\bv$ defined in \eqref{hv},
and parameter-dependent integration curves $\gamma^\bv(\bx,a)$ 
of the same form as the curves $\gamma^e_\perp(\bx,a)$ 
described in Section~\ref{sec:Phi_e}.
Observe that in this construction, each curve 
was fully characterized by the central edge $e$,
the choice of a coarse ($k^-(e)$) and a fine ($k^+(e)$) patch around $e$, 
and finally the choice of a starting edge on the coarse patch $k^-(e)$,
where the starting points are located.
To define the curves $\gamma^\bv(\bx,a)$ we can then specify 
these elements for each vertex, and for this we will use the decomposition  
\eqref{cK_seqs} of the contiguous patches $k \in \cK(\bv)$ in one or two 
sequences of adjacent nested patches $k = k^s_i(\bv)$ with $s \in \{l,r\}$ 
and $1 \le i \le n(\bv,s)$.
An illustration is provided in Figure~\ref{fig:Phi_v} for interior
vertices, and Figure~\ref{fig:Phi_v_bound} for boundary vertices.

On the two patches of a complete sequence, namely for $\bx \in \Omega_k$ 
with $k = k^s_i(\bv)$ such that $n(\bv,s) = 2$, we define the curve 
$\gamma^\bv(\bx,a)$ by taking (i) 
the edge $e = e^s(\bv)$ shared by the two patches $k^s_1(\bv)$ and $k^s_2(\bv)$
as the central edge, (ii)
these respective patches as the coarse and fine patches associated with edge $e$, 
and finally (iii) the second edge contiguous to $\bv$ in $k^s_1(\bv)$
as the starting edge for the curves. 
Note that if both patches adjacent to $e$ have the same resolution, 
then it is possible that $k^s_1(\bv) = k^+(e)$ and $k^s_2(\bv) = k^-(e)$:
in other words the orientation used to define the antiderivatives 
$\Phi^e_\perp$ and $\Phi^\bv$ do not need to match.
This covers the case of interior vertices, since their
contiguous patches can always be decomposed in two complete sequences.

For the case of boundary vertices associated with single-patch sequences, 
namely for $\bx \in \Omega_k$ with $k = k^s_1(\bv)$ such that $n(\bv,s) = 1$, 
we observe that $\Omega_k$ has at least one edge contiguous to $\bv$ 
which is a boundary edge $\partial \Omega$, say $e^s_b(\bv)$: we take this edge 
as the central edge $e$. In the inhomogeneous case \eqref{dR},
we take the patch $k^s_1(\bv)$ as the coarse patch (no fine patch is involved here)
and the second edge contigous to $\bv$ as the starting edge for the 
curves $\gamma^e_\perp(\bx,a)$:
we may denote this edge as $e^*(\bv)$, indeed if a second sequence of patches 
exists for $\bv$ then this starting edge must be shared by $k^l_1(\bv)$ 
and $k^r_1(\bv)$. 
In the homogeneous case \eqref{dR0} we take the patch $k^s_1(\bv)$ 
as the fine patch. Then no coarse patch is involved: the curves are 
all perpendicular to $e$ in the logical coordinates.



The construction is then similar to what we had for 
the antiderivative $\Phi^e_\perp$:
for interior vertices and boundary vertices in the inhomogeneous
case, all the curves $\gamma^\bv(\bx,a)$, $\bx \in \Omega(\bv)$, have a unique starting point 
$\gamma^\bv_*(a)$
lying on the coarse edge $e^*(\bv)$
and at a logical distance $a$ 
from the vertex $\bv$.
For boundary vertices in the homogeneous case, the curves $\gamma^\bv(\bx,a)$
may start from different points $\gamma^\bv_*(\bx,a)$ but they all lie on the 
boundary $\partial \Omega$.

%
%
%
%

\begin{figure}
	\centering 
	\fontsize{18pt}{22pt}\selectfont
	\resizebox{!}{0.25\textheight}{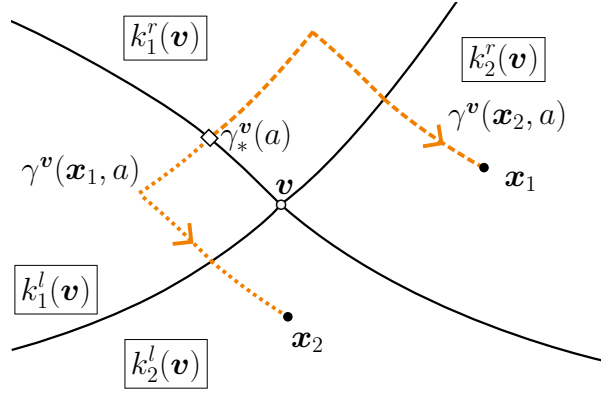}		
	\caption{Integration paths $\gamma^\bv(\bx,a)$ 
	for different points $\bx_1$ and $\bx_2$ 
	involved in the vertex-based antiderivative operator $\Phi^\bv_a$,
	for a given averaging parameter $a \in (0,\hat h_\bv)$. The common starting point $\gamma^\bv_{*}(\bx_1,a) = \gamma^\bv_{*}(\bx_2,a)$
	is represented by a white square.
	}
	\label{fig:Phi_v}
\end{figure}

\begin{figure}[ht]
	\subfloat[inhomogeneous boundary]
	{
   \fontsize{18pt}{22pt}\selectfont
	\resizebox{!}{0.25\textheight}{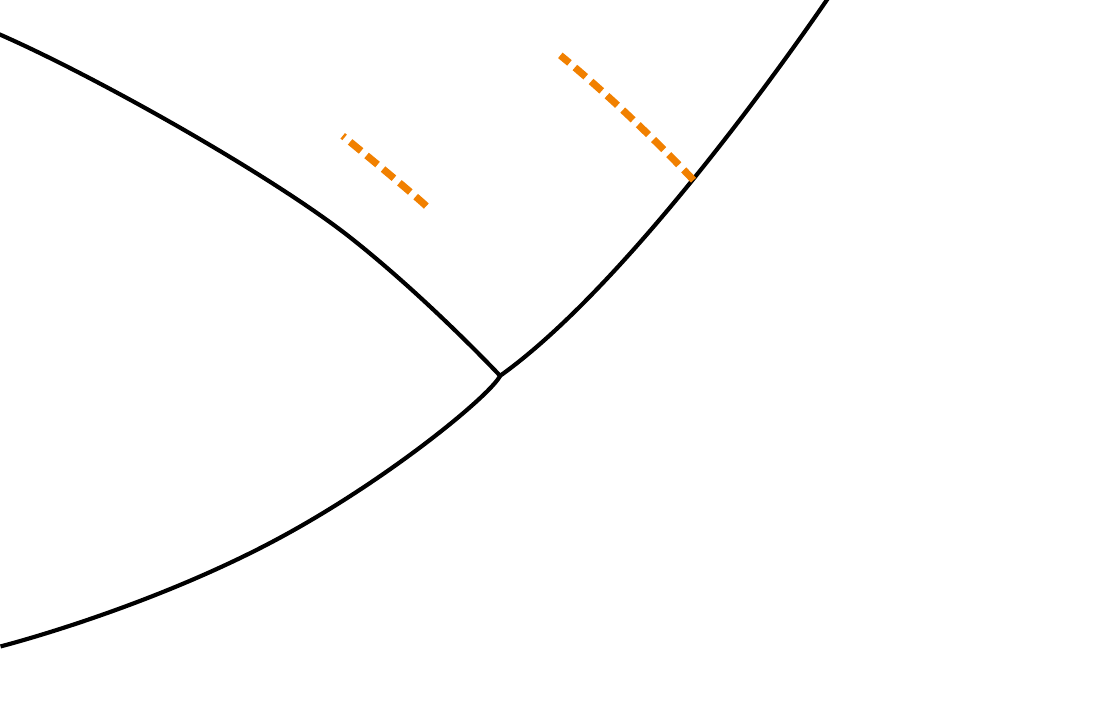}
	   }
	   \\
   \subfloat[homogeneous boundary]	
   {
   \fontsize{18pt}{22pt}\selectfont
   \resizebox{!}{0.25\textheight}{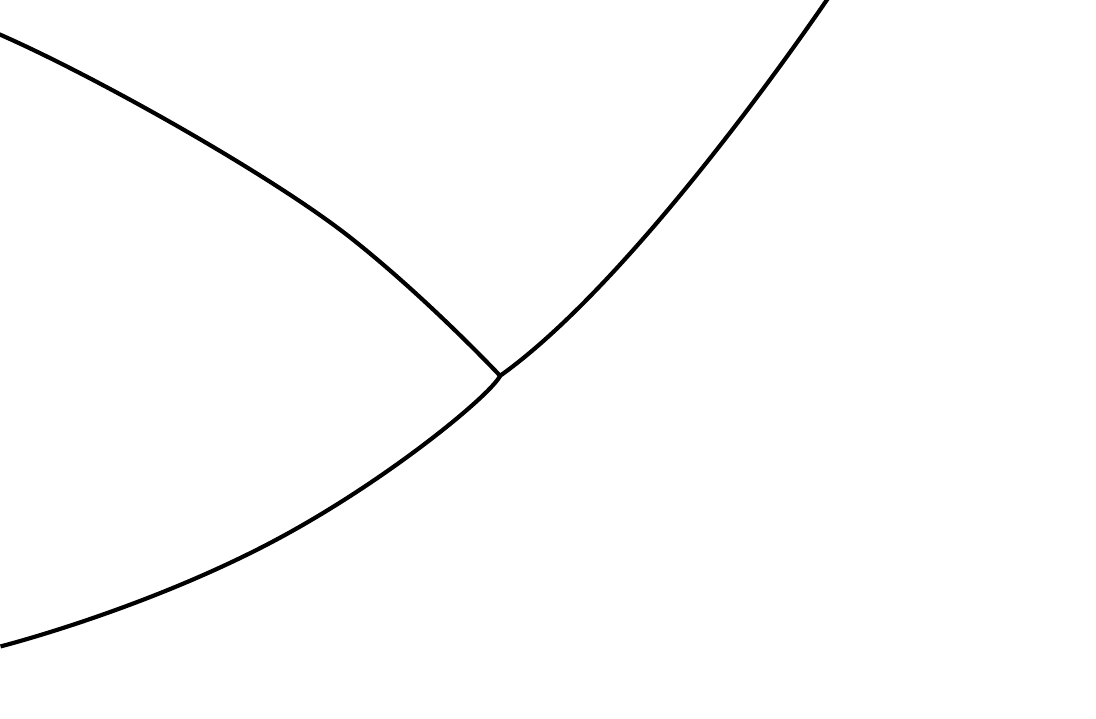}
   \label{fig:Phi_v_type_hom}
   }

   \caption{Integration paths $\gamma^\bv(\bx,a)$ defining the vertex-based antiderivative operators $\Phi^\bv(\bu)$ for boundary vertices.
	 In the inhomogeneous case, integration curves connect every $\bx \in \Omega(\bv)$
	 to a common starting point $\gamma^\bv(a)$. 
	 In the homogeneous case, different points $\bx \in \Omega(\bv)$
	 may be connected to different starting points $\gamma^\bv(\bx, a)$,
	 all on the boundary $\partial \Omega$. 
   }
   \label{fig:Phi_v_bound}
\end{figure}

This antiderivative operator satisfies several properties
which can be directly inferred from those of the perpendicular edge antiderivative.
The first one is similar to \eqref{Phi_perp_nabla} and will be useful 
to prove commuting properties.

\begin{lemma} \label{lem:Phi_v_nabla}
	Let $\bv$ be an interior vertex or a boundary vertex in the inhomogeneous case
	\eqref{dR}, and $\bu = \nabla \phi$ with $\phi \in C^1(\Omega)$. Then
		\begin{equation} \label{Phi_v_nabla}
		\Phi^\bv(\bu)(\bx) = \phi(\bx) - \phi^\bv_*
		\quad \text{ on } \Omega(\bv) 
	\end{equation}
	with a constant
	$\phi^\bv_* \coloneqq 
	\frac{1}{\hat h_\bv} \int_0^{\hat h_\bv} \phi \big(\gamma^\bv_*(a)\big) \dd a$.
	If $\bv$ is a boundary vertex in the homogeneous case
	\eqref{dR0}, and $\bu = \nabla \phi$ with $\phi \in C^1_0(\Omega)$, then
		\begin{equation} \label{Phi_v_nabla_hom}
		\Phi^\bv(\bu)(\bx) = \phi(\bx)
		\quad \text{ on } \Omega(\bv). 
		\end{equation}
\end{lemma}

The second property is a local $L^p$ stability estimate involving vertex 
neighborhoods of the form
\begin{equation} \label{omega_v}
	\omega_\bv \coloneqq \cup_{k \in \cK(\bv)} F_k(\hat \omega^k_\bv),
	\quad \text{ with } \quad 
	\hat \omega^k_\bv 
	\coloneqq \hat \Omega \cap \big(\hat \bv^k + [-\rho \hat h_\bv,\rho \hat h_\bv]^2 \big)
\end{equation}
with $1 \le \rho \lesssim 1$.

\begin{lemma} \label{lem:stab_Phi_v}
	Let $\bu \in L^p(\Omega)$. The bound
	\begin{equation*}
	\norm{\Phi^\bv(\bu)}_{L^p(\omega_\bv)} \lesssim h_\bv \norm{\bu}_{L^p(\omega_\bv)}
	\end{equation*}
	holds for any vertex neighborhood $\omega_\bv$ of the form \eqref{omega_v}.
\end{lemma}

\begin{proof}
	This estimate is proven with the same arguments as \eqref{L2_bound_Phi_e}
	in the perpendicular case, as the integration curves $\gamma^\bv(\bx,a)$
	are of the same form.
	Note that in the present case no localization argument is needed since  
	for $\bx \in \omega_\bv$ the integration curve $\gamma^\bv(\bx, a)$ 
	is fully contained in 
	the vertex neighborhood $\omega_\bv$,
	which is of diameter $\sim h_\bv$.	
\end{proof}

\begin{lemma} \label{lem:loc_bound_corr_v}
	On a domain $S^{\bv}$	of the form \eqref{Sv}, the bound
	\begin{equation} \label{loc_bound_corr_v}	
	\norm{\nabla (P^\bv - \bar I^\bv)\Pi^0_\pw \Phi^{\bv}(\bu)}_{L^p(S^{\bv})} 
	\lesssim \norm{\bu}_{L^p(E_\bv^3(S^{\bv}))}
	\end{equation}
	holds, where $E_\bv^3$ corresponds to a three-fold application of the vertex-based 
	domain extension operator \eqref{Ev_ext}.
\end{lemma}

\begin{proof}
	Using the inverse estimate \eqref{inverse_est} and the local stability of 
	$P^\bv$, $I^\bv$ and $\Pi^0_\pw$, see \eqref{stab_P_I_v} and \eqref{stab_Pi0k}
	(which also holds with vertex-based domain extensions $E_\bv$ according
	to \eqref{Ev_ext}),
 	we write
	\begin{equation*}
	\begin{aligned}
		\norm{\nabla_\pw (P^\bv - \bar I^\bv)\Pi^0_\pw \Phi^{\bv}(\bu)}_{L^p(S^{\bv})} 
		&\lesssim 
		h_\bv^{-1} \norm{(P^\bv - \bar I^\bv)\Pi^0_\pw \Phi^{\bv}(\bu)}_{L^p(E_\bv(S^\bv))} 
		\\
		&\lesssim 
		h_\bv^{-1} \norm{\Pi^0_\pw \Phi^\bv(\bu)}_{L^p(E^2_\bv(S^\bv))} 
		\\
		&\lesssim 
		h_\bv^{-1} \norm{\Phi^\bv(\bu)}_{L^p(E^3_\bv(S^\bv))}
		\\
		&\lesssim 
		\norm{\bu}_{L^p(E^3_\bv(S^\bv))}.		
	\end{aligned}
	\end{equation*}
	where the last step follows from Lemma~\ref{lem:stab_Phi_v},
	noting that $E^3_\bv(S^\bv)$ has the form of a vertex neighborhood \eqref{omega_v}
	with some bounded $\rho \ge 1$.
\end{proof}

The next property will be useful to show projection properties.
\begin{lemma} \label{lem:Phi_V1}
	If $\bu \in V^1_h$, then
	\begin{itemize}
		\item $\Phi^k_d(\bu)$ belongs to the broken space $V^0_\pw$
		\item for all $e\in \cE$, $\Phi^e_\parallel(\bu)$ and $\Phi^e_\perp(\bu)$ belong to $V^0_\pw$ 
		\item if $e$ is an interior edge, $\Phi^e_\parallel(\bu)$ and $\Phi^e_\perp(\bu)$ are continuous across $e$
		\item if $e$ is a homogeneous boundary edge, 
		$\Phi^e_\parallel(\bu)$ and $\Phi^e_\perp(\bu)$ vanish on $\partial \Omega$ 
		\item for all $\bv \in \cV$, $\Phi^\bv(\bu)$ belongs to $V^0_\pw$ and is continuous across every $e\in \cE(\bv)$
		\item if $\bv$ is a homogeneous boundary vertex, 
		$\Phi^\bv(\bu)$ vanishes on $\partial \Omega$.
	\end{itemize}
\end{lemma}

\begin{proof}
To see that antiderivatives of the form
$\Phi(\bu) = \frac{1}{\hat h}\int_0^{\hat h}\Phi_a(\bu)$ 
with $\bu \in V^1_k$ 
belong to the broken space $V^0_k$ we observe that for each $a$ and each 
$\bx \in \Omega_k$, the function $\Phi_a(\bu)(\bx)$ is a sum of at most two terms:
(i) a path integral over a curve $\gamma^k(\bx,a)$ which is in the same patch $\Omega_k$
and consists of one or two segments aligned with the patch axes and 
whose endpoints along dimension $d \in \{1,2\}$ are either constant or 
the coordinate $\hx^k_d$ of $\hbx^k = F^{-1}_k(\bx)$,
and (ii) a matching antiderivative on a coarse adjacent patch which, as a scalar function
of the parallel (interface) variable, belongs to the coarser space on the interface
and hence also to the finer one due to the nested space assumption \eqref{-+_assumption}.
This shows $\Phi_a(\bu) \in V^0_\pw$ and the inclusion $\Phi(\bu) \in V^0_\pw$
follows by linearity.
To see that these antiderivatives are continuous across the respective edges 
for $\bu \in V^1_h$,
we fix again $a$ and observe that the integration curves $\gamma(\bx,a)$ depend
continuously on $\bx$ (in the Hausdorff distance between curves).
For perpendicular curves crossing an interface this is enough to show the continuity
of the antiderivative, and for parallel curves close to an interface the continuity
follows from the continuity of the tangential component of the curl-conforming,
piecewise continuous field $\bu$.
The boundary vanishing properties of the antiderivatives in the case of homogeneous
boundary edges and vertices follow from the fact that on $\partial \Omega$ 
all integration curves involved in these antiderivative operators 
are either of zero length (being normal to the boundary and starting from it), 
or tangent to the boundary: in this latter case they integrate the vanishing
component of the function $\bu \in V^1_h \subset H_0(\curl;\Omega)$.
\end{proof}

\subsection{Edge-vertex antiderivative operators}

For the edge-vertex correction terms we also need to specify 
two additional antiderivative operators: 
a parallel one defined as
\begin{equation} \label{Phi_ev_par}
	 \Phi^{e,\bv}_\parallel(\bu) \coloneqq  \Phi^{\bv}(\bu)
\end{equation}
see \eqref{Phi_v}, and a perpendicular one defined as
\begin{equation} \label{Phi_ev_perp}
	\Phi^{e,\bv}_\perp(\bu) \coloneqq  \Phi^{e}_\perp(\bu)
\end{equation}
see \eqref{Phi_perp_-}--\eqref{foot_curve}.

\begin{lemma} \label{lem:loc_bound_corr_ev}
	On a domain $S^e_{\bv} \coloneqq S^e_{i^e(\bv)}$ of the form \eqref{Sei},
	the bounds
	\begin{equation} \label{loc_bound_corr_ev}	
		\left\{	\begin{aligned}
		&\norm{\nabla^e_\parallel (\bar I^e_{\bv}-P^e_{\bv})\Pi^0_\pw \Phi^{e,\bv}_\parallel(\bu)}_{L^p(S^e_\bv)} 
		\lesssim \norm{\bu}_{L^p(E_\bv^3(S^e_\bv))}
		\\
		&\norm{\nabla^e_\perp (\bar I^e_{\bv}-P^e_{\bv})\Pi^0_\pw \Phi^{e,\bv}_\perp(\bu)}_{L^p(S^e_\bv)} 
		\lesssim \norm{\bu}_{L^p(E_e^3(S^e_\bv))}
	\end{aligned} \right.
	\end{equation}
	hold, where $E_\bv^3$ and $E_e^3$ correspond to three-fold applications of the vertex-based 
	and edge-based domain extension operators \eqref{Ev_ext}, \eqref{Ee_ext}.
\end{lemma}

\begin{proof}
	For the parallel direction where $\Phi^{e,\bv}_\parallel(\bu) \coloneqq  \Phi^{\bv}(\bu)$
	the argument is the same as for \eqref{loc_bound_corr_v}.
	For the perpendicular direction where $\Phi^{e,\bv}_\perp(\bu) \coloneqq  \Phi^{e}_\perp(\bu)$ 
	the argument is the same as for 
	\eqref{loc_bound_corr_e} in the perpendicular direction: here the correction terms
	involve different broken and conforming projections, namely $\bar I^e_{\bv}$ and $P^e_{\bv}$ 
	but again these coincide on constants (see Lemma~\ref{lem:corr_ev_ker})
	which was the key property in the argument.
\end{proof}

\subsection{Bivariate antiderivative operators}
\label{sec:Psi}

Our projection operator on $V^2_h$ involves bivariate antiderivative
operators.
The first one is a simple single-patch operator defined as
\begin{equation} \label{Psi_k}
		\Psi^{k}(f)(\bx) 
			\coloneqq \int_{0}^{\hx_1}\int_{0}^{\hx_2} \hat f^k(z_1, z_2) \dd z_2 \dd z_1,
			\qquad \bx \in \Omega_k, ~ k \in \cK,
\end{equation}
where $\hat f^k \coloneqq (\cF^2_k)^{-1}(f)$ is the 2-form pullback
of $f$ on the patch $\Omega_k$.

Second, an edge-based bivariate antiderivative operator defined as
\begin{equation} \label{Psi_e}
		\Psi^{e}(f)(\bx) 
			\coloneqq \frac{1}{\hat h_e}\int_0^{\hat h_e} \iint_{\sigma^e(\bx,a)} f(\bz) 
					\dd \bz \dd a,
			\qquad \bx \in \Omega(e), ~ e \in \cE,
\end{equation}
where 
$\sigma^e(\bx,a) \subset \Omega(e)$ is the oriented surface whose boundary is 
the algebraic sum of three oriented curves,
\begin{equation} \label{domea}
	\partial \sigma^e(\bx,a) = \gamma^e_\perp(\bx) - \gamma^e_\parallel(\bx,a) 
			+ \tilde \gamma^e(\bx,a)
\end{equation}
where $\gamma^e_\perp(\bx,a)$ 
and $\gamma^e_\parallel(\bx)$ are the curves associated with the 
perpendicular and parallel edge-based antiderivatives in Section~\ref{sec:Phi_e},
while $\tilde \gamma^e(\bx,a)$ is a closing curve following the edges of $\Omega(e)$.

An illustration is given in Figure~\ref{fig:Psi}-(\textsc{a}) for an interior edge $e$:
For $\bx \in \Omega_k$, $k \in \cK(e)$ we remind that
$\gamma^e_\perp(\bx,a)$ connects the point $\gamma^e_{\perp,*}(a)$,
see \eqref{gam*eperp}, to $\bx$, 
while $-\gamma^e_\parallel(\bx)$ 
connects the point $\bx$ to $\gamma^e_{\parallel,*}(\bx)$.
Since $\gamma^e_{\perp,*}(a)$ in on the edge $e^-_* \coloneqq X^-_e(0,[0,1])$
and $\gamma^e_{\parallel,*}(\bx) = X^k_e(\eta_e^k(0),\hx^k_\perp)$ is on the edge 
$e^k_* \coloneqq X^k_e(\eta_e^k(0),[0,1])$,
we see that it is indeed possible to connect $\gamma^e_{\parallel,*}(\bx)$ 
to $\gamma^e_{\perp,*}(a)$ with a curve $\tilde \gamma^e(\bx,a)$ that is included 
in the edges $e^k_*$, $k \in \cK(e)$.

\begin{figure}
	\centering 
 	\subfloat[edge case]
	 {		
	\fontsize{18pt}{22pt}\selectfont
	 \resizebox{!}{0.25\textheight}{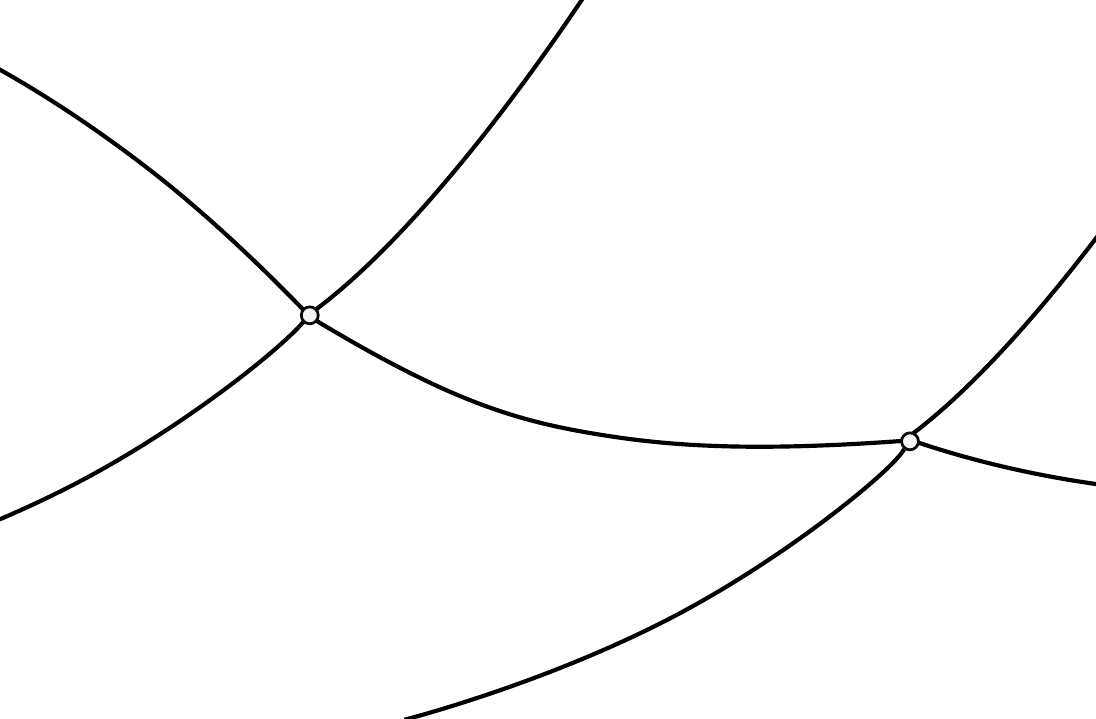}
		\label{fig:Psi_e}
		}
		\\
		\subfloat[edge-vertex case]	
		{		
		\fontsize{18pt}{22pt}\selectfont
      \resizebox{!}{0.25\textheight}{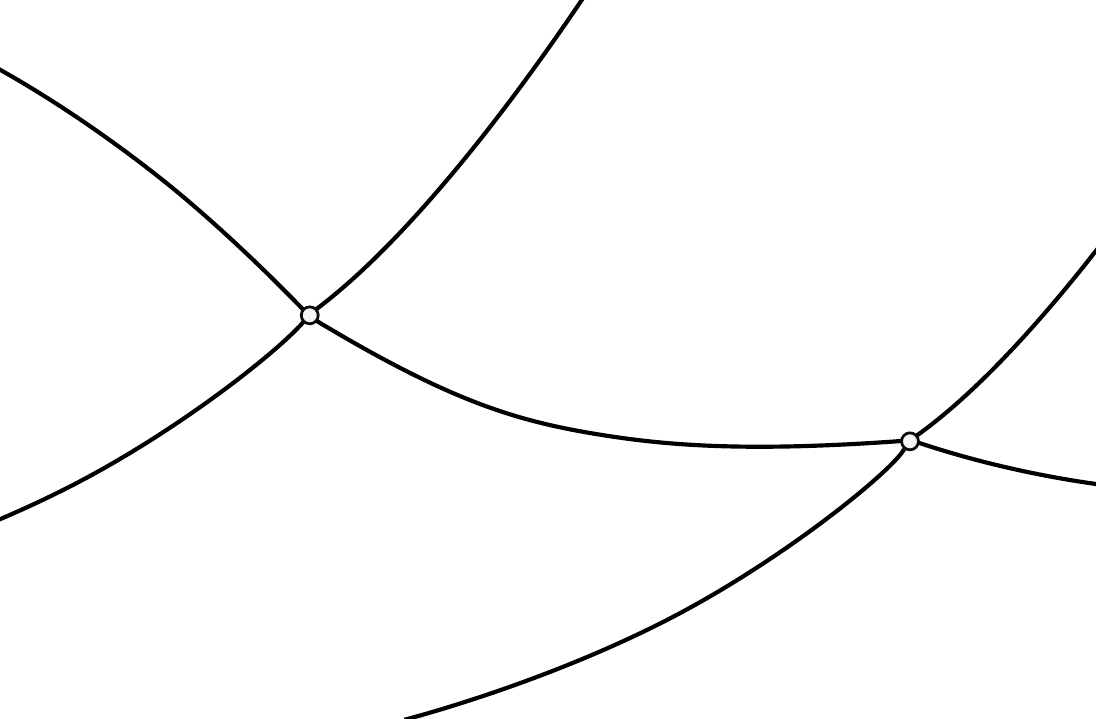}
    	\label{fig:Psi_e_v}
		}
	\caption{Oriented curves delimiting the integration domains involved in the bivariate
	antiderivative operators. On the top panel the curves for the edge-based 
	antiderivative $\Psi^e$ 
	are shown for a given averaging parameter $a \in (0, \hat h_e)$
	with starting point 
	$\gamma^e_{\perp,*}(a)$
	represented by a white square as in Figure~\ref{fig:Phi_e}.
	On the bottom panel the curves of the edge-vertex
	antiderivative $\Psi^{e,\bv}$ 
	are shown, for a given averaging parameter $ \bar a \in (0,1)$. Again, we denote the starting points
	$\gamma^e_{\perp,*}(\hat h_e \bar a)$ and $\gamma^\bv_{*}(\hat h_\bv \bar a)$ of the respective perpendicular edge-based and vertex based antiderivative integration curves
	by white squares as in Figure~\ref{fig:Phi_e}
	and \ref{fig:Phi_v}.  
	}
	\label{fig:Psi}
\end{figure}

Finally, we define an edge-vertex based bivariate antiderivative operator, as
\begin{equation} \label{Psi_ev}
		\Psi^{e,\bv}(f)(\bx) 
			\coloneqq 
				\int_0^{1} \iint_{\sigma^{e,\bv}(\bx,\bar a)} f(\bz) 
					\dd \bz \dd \bar a,
			\qquad \bx \in \Omega(\bv), ~ \bv \in \cV, e \in \cE(\bv) 
\end{equation}
where $\sigma^{e,\bv}(\bx,\bar a)$
is the domain which boundary is again the algebraic sum of three oriented curves, namely
\begin{equation} \label{domeva}
	\partial \sigma^{e,\bv}(\bx,\bar a) = \gamma^e_\perp(\bx,\hat h_e \bar a) - \gamma^\bv(\bx,\hat h_\bv \bar a) + \tilde \gamma^{e,\bv}(\bar a).
\end{equation}
An illustration is given in Figure~\ref{fig:Psi}-(\textsc{b}), 
again for an interior edge $e$:
Here, $\gamma^e_\perp(\bx,\hat h_e \bar a)$ connects as above the point 
$\gamma^e_{\perp,*}(\hat h_e \bar a)$
(defined by \eqref{gam*eperp}) to $\bx$,
$-\gamma^\bv(\bx,\hat h_\bv \bar a)$ connects 
$\bx$ to the point $\gamma^\bv_*(\hat h_\bv \bar a)$
and finally $\tilde \gamma^{e,\bv}(\bar a)$ is the curve 
that connects $\gamma^\bv_*(\hat h_\bv \bar a)$ to 
$\gamma^e_{\perp,*}(\hat h_e \bar a)$ and is included 
in the edges of the patches contiguous to $\bv$.
Observe that the latter curve does not depend on $\bx$.

The following properties will be useful.

\begin{lemma} \label{lem:Psi_curl}
For $\bu \in C^1(\Omega)$, we have 
\begin{equation*} 
\Psi^{e}(\curl \bu) = 
	\Phi^{e}_{\parallel}(\bu) - \Phi^{e}_{\perp}(\bu) + \tilde \Phi^{e}(\bu)
	\quad \text{ with } \quad 
	\tilde \Phi^{e}(\bu)(\bx) \coloneqq 
			\frac{1}{\hat h_e}\int_0^{\hat h_e} \int_{\tilde \gamma^e(\bx,a)} \mspace{-10mu} \bu \cdot \dd l \dd a
	\end{equation*}
and 
\begin{equation*} 
\Psi^{e,\bv}(\curl \bu) = 
	\Phi^{e,\bv}_{\perp}(\bu) - \Phi^{e,\bv}_{\parallel}(\bu) + \tilde \Phi^{e,\bv}(\bu)
	~~ \text{ with } ~~
	\tilde \Phi^{e,\bv}(\bu)(\bx) 
		\coloneqq \int_0^1 \int_{\tilde \gamma^{e,\bv}(\bar a)} \mspace{-10mu}\bu \cdot \dd l \dd \bar a.
\end{equation*}
Moreover, the relations
\begin{equation} \label{Psi_tilde_terms}	
	\nabla^e_\parallel \tilde \Phi^{e}(\bu) = 0
	\qquad \text{ and } \qquad
	\nabla_\pw \tilde \Phi^{e,\bv}(\bu) = 0
\end{equation}
hold on every interior edge $e$. They also hold on boundary edges in the 
inhomogeneous case \eqref{dR}. For the integration curves constructed in the 
homogeneous case \eqref{dR0}, they hold if $\bu \in C^1_0(\Omega)$.
\end{lemma}
\begin{proof}
By Stokes theorem, we have
\begin{equation*} 
		\iint_{\sigma^e(\bx,a)} \curl \bu (\bz) \dd \bz
		= \int_{\gamma^e_\parallel(\bx)} \bu \cdot \dd l 
				- \int_{\gamma^e_{\perp}(\bx,a)} \bu \cdot \dd l 
						+ \int_{\tilde \gamma^e(\bx,a)} \bu \cdot \dd l
\end{equation*}
which shows the first equation, and 
\begin{equation*}
		\iint_{\sigma^{e,\bv}(\bx,\bar a)} \curl \bu (\bz) \dd \bz
		= \int_{\gamma^{e}_\perp(\bx, \hat h_e \bar a)} \bu \cdot \dd l 
			- \int_{\gamma^\bv(\bx,\hat h_\bv \bar a)} \bu \cdot \dd l 
			+ \int_{\tilde \gamma^{e,\bv}(\bar a)} \bu \cdot \dd l
\end{equation*}
which shows the second one 
using \eqref{Phi_ev_par} and \eqref{Phi_ev_perp}.
To complete the proof we observe that for interior edges
the circulation
$
\tilde \Phi^{e}(\bu)(\bx)
$
does not depend on $\hx^k_\parallel$ and 
$\tilde \Phi^{e,\bv}(\bu)(\bx)$ is a constant.
For boundary edges in the homogeneous case the circulation 
$\tilde \Phi^{e}(\bu)(\bx)$ also contains a contribution along $e$
which depends on $\hx^k_\parallel$: for $\bu \in C^1_0(\Omega)$ this 
contribution vanishes. The same argument can be used for 
$\tilde \Phi^{e,\bv}(\bu)(\bx)$.
\end{proof}

\begin{lemma}\label{lem:Pi2_V2h}
If $f \in V^2_h$, then 
both $\Psi^{e}(f)$ and $\Psi^{e,\bv}(f)$ belong to $V^0_\pw$
and they are continuous across any interior edge $e$.
Moreover, in the homogeneous case \eqref{dR0}, they vanish 
on any boundary edge.
\end{lemma}

\begin{proof}
	This result follows from the fact that on a patch $\Omega_k$
	both antiderivatives, for instance
	$\Psi^{e,\bv}(f)$,
	are a sum of integrals 
	of the form 
	\begin{equation*}
	\Psi^{e,\bv}_{\bar a}(f)(\bx) 
		= \iint_{\sigma^e_k(\bx,\bar a)}f 
					+ \sum_{k'} \iint_{\sigma^e_{k'}(\bx,\bar a)}f 
		= \iint_{\hat \sigma^e_k(\hbx^k,\bar a)}\hat f^k 
					+ \sum_{k'} \iint_{\hat \sigma^e_{k'}(\hbx^k,\bar a)}\hat f^{k'}
	\end{equation*}
	where $\hat \sigma^e_k(\hbx^k,\bar a)$ is a Cartesian patch 
	which corners are either constant or aligned with $\hbx^k$
	(so that the corresponding integral is in $V^0_k$)
	and the sum over $k'$ involves Cartesian domains 
	$\hat \sigma^e_{k'}(\hbx^k,\bar a)$ corresponding to
	adjacent patches whose perpendicular, resp. parallel length
	(relative to the edge $e'$ 
	shared with $k$) is a constant, resp. affine function of 
	$\hx^k_{\parallel(e')}$ so that the corresponding integral is in $\VV^{0}_{k'}$ 
	as a function of this parallel coordinate. Because such an adjacent patch $k'$ 
	must be of lower resolution than $k$, it is also in $\VV^{0}_{k}$ and hence in 
	$V^0_k$ as a function of $\bx \in \Omega_k$.
	The continuity over the edges follows from the fact that the Cartesian domains
	depend continuously on $\bx$.
	In the homogeneous case, we observe that our definitions of integration 
	curves associated with boundary edges and vertices lead to 
	domains $\sigma^e(\bx,a)$ and $\sigma^{e,\bv}_k(\bx,\bar a)$ which are of 
	zero measure for $\bx \in e$ when $e \subset \partial \Omega$.
	Hence, $\Psi^{e}(f) = \Psi^{e,\bv}(f) = 0$ on $e$.
\end{proof}

Finally, the $\Phi^2$ (bivariate) antiderivatives satisfy local $L^p$ stability bounds 
similar to the $\Phi^1$ (path integral) antiderivative operators, see  
Lemma~\ref{lem:stab_e_corr} and \ref{lem:loc_bound_corr_ev}.

\begin{lemma} \label{lem:stab_Psi_corr}
	On patch-wise mapped Cartesian domain $\omega_e$
	and $\omega_\bv$ of the form \eqref{omega_e} and
	\eqref{omega_v} with
	$ 
	1 \le \rho \lesssim 1,
	$
	we have 
	\begin{equation} \label{L2_bound_Psi_e}
		\norm{\Psi^{e}(f)}_{L^p(\omega_e)} \lesssim \norm{f}_{L^p(\Omega(e))}.
	\end{equation}
	and 
	\begin{equation} \label{L2_bound_Psi_ev}
		\norm{\Psi^{e,\bv}(f)}_{L^p(\omega_\bv)} \lesssim \norm{f}_{L^p(\Omega(\bv))}.
	\end{equation}
	Moreover, on a local domain $S^{e}_j$
	of the form \eqref{Sei}, $j \in \{0, \dots, n_e\}$,
	the bounds
	\begin{equation} \label{loc_bound_corr_Pi2_e}
	\norm{D^{2,e}(P^e - I^e)\Pi^0_\pw \Psi^{e}(\bu)}_{L^p(S^{e}_j)} 
	\lesssim \norm{f}_{L^p(E_e^3(S^{e}_j))}
	\end{equation}
	and 
	\begin{equation} \label{loc_bound_corr_Pi2_ev}	
		\norm{D^{2,e}(\bar I^e_{\bv}-P^e_{\bv})\Pi^0_\pw \Psi^{e,\bv}(f)}_{L^p(S^e_\bv)} 
		\lesssim \norm{f}_{L^p(E_\bv^3(S^e_\bv))}
	\end{equation}
	hold with $E_e$ and $E_\bv$ the edge-based and vertex-based 
	domain extension operators \eqref{Ee_ext}, \eqref{Ev_ext}.
\end{lemma}

\begin{proof}
	The $L^p$ stability 
	is easily verified:
	using that 
	$\sigma^e(\bx,a) \subset \omega_e$ for all $\bx \in \omega_e$ and $a \in [0,\hat h_e]$,
	we estimate
	\begin{equation} \label{bounds_Psi}
	\begin{aligned}
		\norm{\Psi^{e}(f)}_{L^p(\omega_e)}^p
		&= \iint_{\omega_e} \frac{1}{\hat h_e^p} \Bigabs{\int_0^{\hat h_e}\iint_{\sigma^e(\bx,a)} f(\bz)\dd \bz\dd a}^p
			\dd \bx
		\\
		&\le \iint_{\omega_e}  
				\frac{\abs{\omega_e}^{p-1}}{\hat h_e}\int_0^{\hat h_e}\iint_{\sigma^e(\bx,a)} \abs{f(\bz)}	^p \dd \bz \dd a
			\dd \bx
		\le \abs{\omega_e}^p \norm{f}_{L^p(\Omega(e))}^p		
	\end{aligned}
	\end{equation} 
	where $\abs{\omega_e}$ denotes the measures of
	the respective domains. The bound \eqref{L2_bound_Psi_e} follows from 
	$\abs{\omega_e} \lesssim 1$ and
	\eqref{L2_bound_Psi_e} is verified with the same arguments
	(using now $\sigma^{e,\bv}(\bx,\bar a) \subset \Omega(\bv)$ for 
	$\bx \in \Omega(\bv)$ and $\bar a \in [0,1]$).
	To show the local bounds \eqref{loc_bound_corr_Pi2_e} and \eqref{loc_bound_corr_Pi2_ev}
	we use the same arguments as for \eqref{loc_bound_corr_e} and \eqref{loc_bound_corr_ev}.
	Here the inverse estimate \eqref{inverse_est} applied to the broken 
	mixed derivative $D^{2,e}$
	yields quadratic blow-up factors of order $h_e^{-2}$,
	which can be absorbed by a localized version of the bounds in \eqref{bounds_Psi}: 
	the localizing arguments are then essentially the same as those used 
	for the path integrals in \eqref{loc_bound_corr_e} and \eqref{loc_bound_corr_ev}.	
\end{proof}


\section{Commuting projection operators}
\label{sec:cpo}

In this section, we finalize the construction of our commuting projection operators
and state our main results.

\subsection{Projection operator on $V^0_h$}

Our projection on the first conforming multipatch space simply combines
the local $L^p$ stable projection operator \eqref{Pi0pw} on the individual patches 
with the discrete conforming projection operator $P : V^0_\pw \to V^0_h$ defined 
in Section~\ref{sec:conf}. Remind that this construction handles both the
inhomogeneous case \eqref{dR} and the homogeneous one \eqref{dR0}.
The resulting projection is
\begin{equation} \label{Pi0}
	\Pi^0 \coloneqq P \Pi^0_\pw.
\end{equation}	

\begin{lemma} \label{lem:stab_Pi0}
	The operator \eqref{Pi0} is a projection on the space $V^0_h$, 
	and it satisfies
	\begin{equation} \label{stab_Pi0}
		\norm{\Pi^0 \phi}_{L^p(\Omega)} \lesssim \norm{\phi}_{L^p(\Omega)}.
	\end{equation}	
\end{lemma}

\begin{proof}
	The projection property is a direct consequence of the fact that $\Pi^0_\pw$ is 
	a projection on the patch-wise space $V^0_\pw$ and that $P$ is a projection on the conforming 
	subspace $V^0_h$.
	The $L^p$ stability \eqref{stab_Pi0} follows from the bounds \eqref{stab_Pi0_pw} and \eqref{stab_P}.
\end{proof}

As described in Section~\ref{sec:approach}, the projection $\Pi^1$ then involves 
single-patch projections which commute with the broken derivatives.


\subsection{Single-patch commuting projection operators}

On each patch a projection operator on $V^1_k$ 
is defined following the tensor-product approach of \cite{buffa_isogeometric_2011}, as 
\begin{equation} \label{Pi1k}
\Pi^1_k \bu \coloneqq \sum_{d \in \{1,2\}} \nabla^k_d \Pi^0_k \Phi^k_d(\bu)
\end{equation}
where $\nabla^k_d$ is the patch-wise directional gradient \eqref{nabla_kd}
and $\Phi^k_d$ is the single-patch directional antiderivative \eqref{Phi_kd_exp}.
This definition corresponds to setting 
\begin{equation} \label{hPi1k}
	\Pi^1_k = \cF^1_k \hat \Pi^1_k (\cF^1_k)^{-1}
	\quad \text{ with } \quad
	\hat \Pi^1_k 
	\hat \bu \coloneqq
	\begin{pmatrix}
		\partial_1 \hat \Pi^0_k \big( \int_{0}^{\hx_1} \hat u_1(z_1, \hx_2) \dd z_1 \big) \\
		\partial_2 \hat \Pi^0_k \big( \int_{0}^{\hx_2} \hat u_2(\hx_1, z_2) \dd z_2 \big)
	\end{pmatrix}
	\in \hat V^1_k.
\end{equation}

Similarly, a projection operator on $V^2_k$ is defined as
\begin{equation} \label{Pi2k}
\Pi^2_k f \coloneqq D^{2,k} \Pi^0_k \Psi^{k}(f)
\end{equation}
where $D^{2,k}$ and $\Psi^{k}$ are the single-patch mixed derivative and 
bivariate antiderivative operators, see \eqref{D2k} and \eqref{Psi_k}.
This amounts to writing
\begin{equation} \label{hPi2k}
	\Pi^2_k = \cF^2_k \hat \Pi^2_k (\cF^2_k)^{-1}
	\quad \text{ with } \quad
	\hat \Pi^2_k \hat f \coloneqq
	\partial_1 \partial_2 
			\hat \Pi^0_k \big( \int_{0}^{\hx_1}\int_{0}^{\hx_2} \hat f(z_1, z_2) \dd z_2 \dd z_1 \big)
	\in \hat V^2_k.
\end{equation}

These single-patch operators satisfy some key properties 
which essentially follow from \cite{buffa_isogeometric_2011}.
\begin{lemma} \label{lem:Pik}
	The operators \eqref{Pi1k} and \eqref{Pi2k} are projections 
	on the space $V^1_k$ and $V^2_k$, and they satisfy
	\begin{equation} \label{stab_Pik}
		\norm{\Pi^1_k \bu}_{L^p(\Omega_k)} \lesssim \norm{\bu}_{L^p(\Omega_k)},
		\qquad 
		\norm{\Pi^2_k f}_{L^p(\Omega_k)} \lesssim \norm{f}_{L^p(\Omega_k)}
	\end{equation}
	and 
	\begin{equation} \label{cpk_1}
		\nabla^k \Pi^0_k \phi = \Pi^1_k \nabla^k \phi 
			\quad \text{ for all } \phi \in H^1(\Omega_k)
	\end{equation}
	as well as
	\begin{equation} \label{cpk_2}
		\curl^k \Pi^1_k \bu = \Pi^2_k \curl^k \bu 
			\quad \text{ for all } \bu \in H(\curl;\Omega_k).
	\end{equation}
\end{lemma}
\begin{proof}
	The $L^p$ bound follows by summing up the estimates \eqref{loc_bound_Pik1}
	on the supports $S^k_\bi$ which cover the patch $\Omega_k$, and using the 
	bounded overlapping of the extensions $E^2_k(S^k_\bi) \cup \Omega_k$.
	To show the commuting property \eqref{cpk_1} we consider $\phi \in C^1(\Omega_k)$
	and observe that the path integral \eqref{Phi_kd_exp} of $\bu = \nabla \phi$ along 
	a logical dimension $d$ reads
		$\Phi^{k}_d(\bu) = \phi - \phi^*$
	where $\phi^*(\bx) \coloneqq \phi(F_k(\hbx^*))$ with $\hx^*_d \coloneqq 0$ and $\hx^*_{d'} \coloneqq \hx_{d'}$
	for the other component.
	In particular $\nabla^k_d \phi^* = 0$ and the preservation 
	of directional invariance by $\Pi^0_k$ (see Lemma~\ref{lem:dd_Pi0k}) 
	yields 
	$
	\nabla^k_d \Pi^0_k \Phi^{k}_d(\bu) = \nabla^k_d \Pi^0_k \phi.
	$
	This shows \eqref{cpk_1} for $\phi \in C^1(\Omega)$ and the result follows by density.
	The $L^p$ stability of $\Pi^2_k$, as well as the commuting relation \eqref{cpk_2},
	are shown with similar arguments. We also refer to \cite{buffa_isogeometric_2011} for more details.
\end{proof}

Summing over the patches, we obtain projection operators 
$\Pi^1_\pw$ and $\Pi^2_\pw$ on the patch-wise spaces,
see \eqref{Pipw}.
These operators are obviously stable in $L^p$ and they 
commute with the patch-wise, broken gradient and curl:
\begin{equation} \label{pw_grad_cp}
\Pi^1_\pw \nabla \phi = \nabla_\pw \Pi^0_\pw \phi, \qquad \phi \in H^1(\Omega)
\end{equation}
and 
\begin{equation} \label{pw_curl_cp}
\Pi^2_\pw \curl \bu = \curl_\pw \Pi^1_\pw \bu, \qquad \bu \in H(\curl;\Omega).
\end{equation}
Our next task is to modify $\Pi^1_\pw$ so that it becomes a projection on 
the conforming space $V^1_h$ with commuting properties involving the projection 
$\Pi^0$ defined by \eqref{Pi0}.

\subsection{The commuting projection operator on $V^1_h$.}

A suitable projection operator on $V^1_h$ is obtained by adding correction terms
to the single-patch projections \eqref{Pi1k}. Thanks to the
specific definition of the local projections and antiderivative operators,
we can handle both the inhomogeneous and homogeneous cases in the same way.
Thus, we set
\begin{equation}
	\label{Pi1}
	\Pi^1 \coloneqq 
		\sum_{k \in \cK} \Pi^1_k 
			+ \sum_{e \in \cE} \tilde \Pi^1_e
					+ \sum_{\bv \in \cV} \tilde \Pi^1_\bv
						+ \sum_{\bv \in \cV, e \in \cE(\bv)} \tilde \Pi^1_{e,\bv}
\end{equation}
with edge correction terms 
\begin{equation}
	\label{tPi1_e}
	\tilde \Pi^1_e : 
	\left\{\begin{aligned}
	&L^p(\Omega) \to V^1_\pw,
	\\
 	&\bu \mapsto \sum_{d \in \{\parallel, \perp\}} 
	 				\nabla^e_d (P^e - I^e) \Pi^0_\pw \Phi^e_{d}(\bu)
	\end{aligned}
\right.
\end{equation}
that involve the edge antiderivative operators
\eqref{Phi_par} and \eqref{Phi_perp}--\eqref{foot_curve},
the patch-wise projection \eqref{Pi0k} on $V^0_\pw$, 
the local (edge-based) conforming and broken projection operators \eqref{Ie}, \eqref{Pe}
and the edge-directional broken gradient operator \eqref{nabla_ed},
vertex correction terms
\begin{equation}
	\label{tPi1_v}
	\tilde \Pi^1_\bv : 
	\left\{\begin{aligned}
	&L^p(\Omega) \to V^1_\pw,
	\\
 	&\bu \mapsto 
	 				\nabla_\pw (P^{\bv} - \bar I^{\bv}) \Pi^0_\pw \Phi^\bv(\bu)
	\end{aligned}
\right.
\end{equation}
that involve the vertex antiderivative operator \eqref{Phi_v},
the vertex-based conforming and broken projection operators \eqref{Iv}, \eqref{Pv}
and the patch-wise gradient operator \eqref{nabla_pw}.
Finally, the last terms are edge-vertex corrections
\begin{equation}
	\label{tPi1_ev}
	\tilde \Pi^1_{e,\bv} : 
	\left\{\begin{aligned}
	&L^p(\Omega) \to V^1_\pw,
	\\
 	&\bu \mapsto 
					\sum_{d \in \{\parallel, \perp\}} 
						\nabla^e_d (\bar I^e_{\bv} - P^e_{\bv})\Pi^0_\pw \Phi^{\bv,e}_{d}(\bu)
	\end{aligned}
\right.
\end{equation}
that involve the edge-vertex antiderivative operators \eqref{Phi_ev_par}
and \eqref{Phi_ev_perp},
the edge-vertex broken and conforming projection operators \eqref{Iev} and \eqref{Pev},
and again the edge-directional broken gradient operator \eqref{nabla_ed}.

\subsection{The commuting projection operator on $V^2_h$.}

A commuting projection $V^2_h$ is also obtained by adding correction terms
to the single-patch projections \eqref{Pi2k}. Specifically, it is defined as 
\begin{equation}
	\label{Pi2}
	\Pi^2 \coloneqq
		\sum_{k \in \cK} \Pi^2_k 
			+ \sum_{e \in \cE} \tilde \Pi^2_e
						+ \sum_{\bv \in \cV, e \in \cE(\bv)} \tilde \Pi^2_{e,\bv}
\end{equation}
with edge correction terms 
\begin{equation}
	\label{tPi2_e}
	\tilde \Pi^2_e : 
	\left\{\begin{aligned}
	&L^p(\Omega) \to V^2_\pw,
	\\
 	&f \mapsto 
	 				D^{2,e} (P^e - I^e) \Pi^0_\pw \Psi^{e}(f)
	\end{aligned}
\right.
\end{equation}
and edge-vertex corrections
\begin{equation}
	\label{tPi2_ev}
	\tilde \Pi^2_{e,\bv} : 
	\left\{\begin{aligned}
	&L^p(\Omega) \to V^2_\pw,
	\\
 	&f \mapsto 
					D^{2,e} 
						(\bar I^e_{\bv} - P^e_{\bv})\Pi^0_\pw \Psi^{\bv,e}(f).
	\end{aligned}
\right.
\end{equation}
These terms involve the bivariate edge and edge-vertex antiderivatives
\eqref{Psi_e} and \eqref{Psi_ev},
the patch-wise projection \eqref{Pi0k} on $V^0_\pw$, 
the local broken and conforming projection operators \eqref{Ie}, \eqref{Pe}, \eqref{Iev} and \eqref{Pev}
and the broken mixed derivative operator \eqref{D2e}.

\subsection{The main result}
\label{sec:main}

We are now in a position to state our main result for the inhomogeneous and homogeneous 
$\nabla$-$\curl$ sequences \eqref{dR} and \eqref{dR0}. 
Our result for the $\bcurl$-$\Div$ sequences \eqref{cD} and \eqref{cD0} will be presented immediately afterwards. 
We remind that for the sequence \eqref{dR0} the homogeneous boundary conditions 
are handled by removing boundary basis functions from the conforming basis of $V^0_h$ (with a consistent adaptation of the conforming projection $P$ and its localized counterparts $P^e, P^\bv, P^e_\bv$ in Section~\ref{sec:locprojs}), 
and by defining the antiderivative operators $\Phi^e_\perp$ and $\Phi^\bv$ 
through integration curves that start from the homogeneous boundaries 
in Sections~\ref{sec:Phi_e} and \ref{sec:Phi_v}. 
\begin{theorem} \label{thm:main_result}
	The operators $\Pi^\ell$ defined in 
	\eqref{Pi0}, \eqref{Pi1} and \eqref{Pi2}
	are local projections on the respective spaces $V^\ell_h$, $\ell = 0, 1, 2$.
	They satisfy the a priori bounds
	\begin{equation} \label{stab_Pi}
		\norm{\Pi^\ell v}_{L^p(\Omega)} \lesssim \norm{v}_{L^p(\Omega)}
	\end{equation}
	for $1 \le p \le \infty$, with constants that only depend on the smoothness parameters 
	$\kappa_1, \dots, \kappa_8$
	described in Section~\ref{sec:broken}. Moreover, the commuting relations
	\begin{equation} \label{cp}
		\nabla \Pi^0 \phi = \Pi^1 \nabla \phi
		\quad  \text{ and } \quad
		\curl \Pi^1 \bu = \Pi^2 \curl \bu 
	\end{equation}	
	hold for all $\phi \in H^1(\Omega)$ and all $\bu \in H(\curl;\Omega)$
	in the inhomogeneous case \eqref{dR}, and for 
	all $\phi \in H^1_0(\Omega)$ and all $\bu \in H_0(\curl;\Omega)$
	in the homogeneous case \eqref{dR0}.
\end{theorem}

\begin{remark} \label{rem:L1cd}
	The commuting relations \eqref{cp} actually hold on larger spaces: in 
	the inhomogeneous case \eqref{dR} they hold for all $\phi \in W^{1,1}(\Omega)$ and all 
	$\bu \in W^{1}(\curl;\Omega)$ where the latter space consists of the functions 
	$\bu \in L^1$ with $\curl \bu \in L^1$, equipped with the norm 
	$\norm{\bu}_{1,\curl} \coloneqq \norm{\bu}_{L^1} + \norm{\curl \bu}_{L^1}$.
	In the homogeneous case \eqref{dR0} they hold for all $\phi \in W^{1,1}_0(\Omega)$ and all 
	$\bu \in W^{1}_{0}(\curl;\Omega)$ where the latter space is the closure of 
	$C^1_0(\Omega)$ in $W^{1}(\curl;\Omega)$.
\end{remark}

To handle the $\bcurl$-$\Div$ sequences \eqref{cD} and \eqref{cD0}, we observe that
	\begin{equation} \label{startdiffeq}
		\bcurl \phi = R\nabla \phi 
		 \quad \text{ and }  \quad 
		\Div \bv = \curl R^{-1} \bv 
		 \quad \text{ with}  \quad 
		R = \begin{pmatrix} 0 & 1 \\ -1 & 0  \end{pmatrix}.
	\end{equation}
	 This leads to defining
	\begin{equation}	\label{Pikstar}
		 \Pi^{0,*}   \coloneqq  \Pi^{0} ,
		\qquad 
		 \Pi^{1,*}  \coloneqq R  \Pi^{1}  R^{-1} ,
		\qquad 
		 \Pi^{2,*}  \coloneqq  \Pi^{2} ,
	\end{equation}
	which allows to transfer the good properties of the $\Pi^\ell$ projections 
	to the $\bcurl$-$\Div$ sequence.
	
	Analogously, the pushforward operators for the $\nabla$-$\curl$ sequence in 
	\eqref{pf_gc} are also related 
	to the ones of the 	$\bcurl$-$\Div$ sequence, namely
	\begin{equation} \label{pf_cd_rk}
	  \left\{
	  \begin{aligned}
	  &\cF^{0,*}_k : \hat \phi \mapsto \phi \coloneqq \hat \phi \circ F_k^{-1}
	  \\
	  &\cF^{1,*}_k : \hat \bu \mapsto \bu \coloneqq  \big(J_{F_k}^{-1} DF_k \hat \bu \big)\circ F_k^{-1}
	  \\
	  &\cF^{2,*}_k : \hat f \mapsto f \coloneqq  \big(J_{F_k}^{-1} \hat f \big)\circ F_k^{-1}.
	  \end{aligned}
	  \right.
	\end{equation}
	Using the matrix relation $R DF_k^{-T} = J_{F_k}^{-1} DF_k R$, we find indeed
	\begin{equation}
	\cF^{0,*}_k = \cF^{0}_k, \qquad
	\cF^{1,*}_k = R \cF^1_k R^{-1}, \qquad
	\cF^{2,*}_k = \cF^{2}_k.
	\end{equation}

	Following Section \ref{tensorprodspaces}, we define the reference patch finite element spaces as
\begin{equation*}
		\hat V^{0,*}_{k}  \coloneqq \hat V^{0}_{k} = \VV^{0}_k \otimes \VV^{0}_k,
		\quad
		\hat V^{1,*}_{k}  \coloneqq R \hat V^{1}_{k} = \begin{pmatrix}
			\VV^{0}_k\otimes \VV^{1}_{k}
		  \\ \noalign{\smallskip}
		  \VV^{1}_{k} \otimes \VV^{0}_k		\end{pmatrix},
		\quad
		\hat V^{2, *}_{k}  \coloneqq \hat V^{2}_{k} = \VV^{1}_{k} \otimes \VV^{1}_{k},
		\end{equation*}
		and upon pushing forward the patch-wise spaces with
		$V^{\ell, *}_k = \cF^{\ell,*}_k( \hat V^{\ell, *}_k)$, we define the global conforming spaces as
	\begin{equation*} 
		V^{0,*}_h = V^{0,*}_\pw \cap H(\bcurl;\Omega),
		\quad 
		V^{1,*}_h = V^{1,*}_\pw \cap H(\Div;\Omega),
		\quad 
		V^{2,*}_h = V^{2,*}_\pw \cap L^2(\Omega) = V^2_\pw.
	\end{equation*}
	
These definitions allow us to extend the main theorem \ref{thm:main_result} to this sequence:
\begin{theorem} \label{thm:main_result_star}
	The operators $\Pi^{\ell, *}$ defined in \eqref{Pikstar} 
	are local projections on the respective spaces $V^{\ell,*}_h$, $\ell = 0, 1, 2$.
	They satisfy the a priori bounds
	\begin{equation} \label{stab_Pi_star}
		\norm{\Pi^{\ell,*} v}_{L^p(\Omega)} \lesssim \norm{v}_{L^p(\Omega)}
	\end{equation}
	for $1 \le p \le \infty$, with constants that only depend on the smoothness parameters 
	$\kappa_1, \dots, \kappa_8$
	described in Section~\ref{sec:broken}. Moreover, the commuting relations
	\begin{equation} \label{cp_star}
		\bcurl \Pi^{0,*} \phi = \Pi^{1,*} \bcurl \phi
		\quad  \text{ and } \quad
		\Div \Pi^{1,*} \bu = \Pi^{2,*} \Div \bu 
	\end{equation}	
	hold for all $\phi \in H(\bcurl; \Omega)$ and $\bu \in H(\Div;\Omega)$
	in the inhomogeneous case \eqref{cD}, and 
	hold for all $\phi \in H_0(\bcurl; \Omega)$ and $\bu \in H_0(\Div;\Omega)$
	in the homogeneous case \eqref{cD0}.
\end{theorem}

\begin{remark} Again, these commuting relations actually hold in 
	larger spaces characterized by $L^1$
	integrability, defined similarly as in Remark~\ref{rem:L1cd}.
\end{remark}

\begin{proof}

The stability follows from Theorem ~\ref{thm:main_result} and the fact that $R$ and $R^{-1}$ are isometries in any $L^p$. 

The range and projection properties for $\ell = 0, 2$ follow from 
Theorem~\ref{thm:main_result} as the projectors and conforming spaces are equal. For $\ell=1$, we first realize that the spaces $H(\curl; \Omega)$ and $H(\Div; \Omega)$ have similar, but rotated, conformity conditions, i.e. continuity along tangent and normal vectors respectively. Since $V^{1,*}_k = \cF^{1,*}_k( \hat V^{1, *}_k) = R \cF^1_k (R^{-1} R \hat V^1_k) = R V^1_k$, we have 
\begin{equation} \label{RV1}
	V^{1,*}_h = R V^{1}_\pw \cap H(\Div; \Omega) = R V^1_h.
\end{equation}
 Thus, the range and projection properties are direct consequences of \eqref{RV1} and definition of $\Pi^{1,*}$.

The commuting properties are again a consequence of Theorem~\ref{thm:main_result} and the relations of differential operators in \eqref{startdiffeq}, i.e. 
 \begin{equation*}
 \Pi^{1,*}  \bcurl  \phi = 
R   \Pi^{1} R^{-1}  \bcurl  \phi =
R   \Pi^{1}  \nabla  \phi 
= R  \nabla  \Pi^{0}  \phi = \bcurl  \Pi^{0,*}  \phi
\end{equation*}
and
\begin{equation*}	
		 \Pi^{2,*}  \Div  \bv 
	=  \Pi^{2}  \curl R^{-1}  \bv 
	=  \curl \Pi^{1}  R^{-1}  \bv
	=  \curl R^{-1}  \Pi^{1,*}  \bv
	=  \Div \Pi^{1,*}  \bv.
\end{equation*}
\end{proof}	

We conclude this section by listing a few corollaries of 
Theorems~\ref{thm:main_result} and \ref{thm:main_result_star}.
The first one is a direct consequence of the $L^p$ stability and commuting sequence properties.

\begin{corollary}
The operators defined in \eqref{Pi0} and \eqref{Pi1} satisfy the bounds
\begin{equation*}
	\norm{\Pi^\ell_h v}_{V^\ell,L^p} \lesssim \norm{v}_{V^\ell,L^p}
\end{equation*}
for $\ell = 0, 1$, with the norms
\begin{equation*}
	\left\{ \begin{array}{l}
		\norm{\phi}_{V^0,L^p} = (\norm{\phi}^p_{L^p} + \norm{\nabla \phi}^p_{L^p})^{\frac 1p} 
		\\
		\norm{\bv}_{V^1,L^p} = (\norm{\bv}^p_{L^p} + \norm{\curl \bv}^p_{L^p})^{\frac 1p}
	\end{array} \right. , 
\end{equation*}
and constants that only depend on the smoothness parameters $\kappa_1, \dots, \kappa_8$
described in Section~\ref{sec:broken}.
Similarly, the operators defined in \eqref{Pikstar} satisfy the bounds
\begin{equation*}
	\norm{\Pi^{\ell,*} v}_{V^\ell,L^p} \lesssim \norm{v}_{V^\ell,L^p}
\end{equation*}
for $\ell = 0, 1$, with the same constants and the norms
\begin{equation*}
	\left\{ \begin{array}{l}
		\norm{\phi}_{V^0,L^p} = (\norm{\phi}^p_{L^p} + \norm{\bcurl \phi}^p_{L^p})^{\frac 1p} 
		\\
		\norm{\bv}_{V^1,L^p} = (\norm{\bv}^p_{L^p} + \norm{\Div \bv}^p_{L^p})^{\frac 1p}
	\end{array} \right. .
\end{equation*}
\end{corollary}

A second stability result follows by reasoning as in \cite[Theorem 3.6]{Arnold.Falk.Winther.2010.bams}: 
\begin{corollary} 
If the spaces $V^\ell$ in the $\nabla$-$\curl$ sequence \eqref{dR} 
satisfy a Poincaré-Friedrichs inequality 
\begin{equation} \label{PI}
	\| v \|_{L^p} \le c_\mathrm{P} \| d^\ell v \|_{L^p}, \quad v \in K
\end{equation}
with $K =  (\ker d^\ell )^\perp$, $d^0 = \nabla, \ d^1 = \curl$ and additionally the following bound holds
\begin{equation} \label{norm-bd}
	\| v \|_{L^p} \le c_\mathrm{s} \sup_{w \in K_h} \frac{\sprod
	{v}{w}}{\norm{w}_{L^q}}, \quad v \in K_h
\end{equation}
where $K_h = V^\ell_h \cap (\ker d^\ell|_{V^\ell_h} )^\perp$ and $\frac1p + \frac1q = 1$, then the discrete spaces $V^\ell_h$
satisfy a discrete Poincaré-Friedrichs inequality 
\begin{equation}\label{PIh}
	\| v \|_{L^p} \le c_\mathrm{P} c_\mathrm{s} c_\Pi \| d^\ell v \|_{L^p}, \quad v \in K_h
\end{equation}
where $c_\Pi$ only depends on the parameters $\kappa_1, \dots, \kappa_8$ from Section~\ref{sec:broken}.
Similarly, if the spaces $V^\ell$ in the $\bcurl$-$\Div$ sequence \eqref{cD} 
satisfy a Poincaré-Friedrichs inequality of the form \eqref{PI} with 
$d^0 = \bcurl, \ d^1 = \Div$, then the discrete spaces $V^{\ell,*}_h$ 
satisfy a discrete Poincaré-Friedrichs inequality of the form \eqref{PIh}.
Note that for $p=2$ the bound \eqref{norm-bd} always holds.
\end{corollary}
Another corollary of Theorems~\ref{thm:main_result} and \ref{thm:main_result_star} is 
the well-posedness and a priori error estimates for FEEC approximations 
of Hodge Laplacian source problems of the form
\begin{equation} \label{hod-lap}
	\cL \bu = \bsf 
	\quad \text{ where } \quad \cL = -\nabla \Div + \bcurl \curl
\end{equation}
either in the curl-conforming space $V^1_h$, or in the div-conforming space $V^{1,*}_h$, 
see Theorem 3.8, 3.9 and 3.11 of \cite{Arnold.Falk.Winther.2010.bams}.
One further application regards the associated eigenvalue problem.
\begin{corollary}
	If the continuous sequence $V$ satisfies the compactness property, then the FEEC approximation to the eigenvalue problem 
	$\cL \bu = \lambda \bu$ converges towards the exact one in the sense of 
	Theorems~3.19 and 3.21 in \cite{Arnold.Falk.Winther.2010.bams}.
\end{corollary}


\section{Proof of the main result}
\label{sec:proof}

This section is devoted to the proof of Theorem ~\ref{thm:main_result},
which we decompose in several Lemmas.

\subsection{$L^p$ stability}

\begin{lemma} \label{lem:stab_Pi12}
	The projection operators \eqref{Pi1} and \eqref{Pi2} satisfy
\begin{equation} \label{stab_Pi1}
		\norm{\Pi^1 \bu}_{L^p(\Omega)} \lesssim \norm{\bu}_{L^p(\Omega)}
\end{equation}
and
\begin{equation} \label{stab_Pi2}
		\norm{\Pi^2 f}_{L^p(\Omega)} \lesssim \norm{f}_{L^p(\Omega)}.
\end{equation}
\end{lemma}

\begin{proof}
	We assemble the local estimates for the different terms: 
	the stability of the patch-wise projection $\Pi^1_\pw$ has been 
	established in Lemma~\ref{lem:Pik}. 
	For the edge-based correction terms we use
  Lemma~\ref{lem:stab_e_corr}: since the domains $S^e_j$, $j \in \cI^e$ 
	cover the support of the term $\tilde \Pi^1_e$ according to \eqref{supp_PIev}, 
	we have
	\begin{equation*}
	\norm{\tilde \Pi^1_e \bu}^2_{L^p(\Omega)} 
		\le \sum_{j \in \cI^e} \norm{\tilde \Pi^1_e \bu}^2_{L^p(S^e_j)}
		 \lesssim \sum_{j \in \cI^e} \norm{\bu}^2_{L^p(E^3_e(S^e_j))}
		 \lesssim \norm{\bu}^2_{L^p(\Omega(e))}
	\end{equation*}	
	where we remind that $\Omega(e) = \cup_{k \in \cK(e)} \Omega_k$
	and we have used the bounded overlapping of the extensions $E^3_e(S^e_j)$, 
	which are all in $\Omega(e)$. 
	For the vertex and edge-vertex correction terms $\tilde \Pi^1_\bv \bu$ and 
	$\tilde \Pi^1_{e,\bv} \bu$
	which are all supported in $S^\bv$ according to \eqref{supp_PIev}, we use the fact that 
	vertex-based extensions $E_\bv(\omega)$ are always included in the vertex domain 
	$\Omega(\bv)$, so that the bounds 
	\ref{loc_bound_corr_v} and \ref{loc_bound_corr_ev} readily provide us with the bounds
	\begin{equation*}
	\norm{\tilde \Pi^1_\bv \bu}_{L^p(\Omega(\bv))} 
		+ \sum_{e \in \cE(\bv)} \norm{\tilde \Pi^1_{e,\bv} \bu}_{L^p(\Omega(\bv))}
		\lesssim \norm{\bu}_{L^p(\Omega(\bv))}.
	\end{equation*}
	Estimate \eqref{stab_Pi1} is then obtained by summing the above estimates 
	over the edges and vertices,
	and using the bounded overlapping of the domains $\Omega(e)$ and $\Omega(\bv)$.
	Estimate \eqref{stab_Pi2} is proven with the same arguments, using 
	Lemma~\ref{lem:stab_Psi_corr}.
\end{proof}

\subsection{Range property}

\begin{lemma} \label{lem:range}
	For all $\phi, \bu, f \in L^p(\Omega)$, $\Pi^0 \phi$, $\Pi^1 \bu$ and 
	$\Pi^2 f$ belong to the respective spaces $V^0_h$, $V^1_h$ and $V^2_h$.
\end{lemma}

\begin{proof}
The property for $\Pi^0$ has been established in Lemma~\ref{lem:stab_Pi0}.
By construction, it is clear that $\Pi^1$ and $\Pi^2$ maps into the broken 
spaces $V^{1}_\pw$ and $V^{2}_\pw$. For $\Pi^2$ this is enough since
$V^{2}_h = V^{2}_\pw$, while for $\Pi^1$
we need to show that it also maps in $H(\curl;\Omega)$ in the inhomogeneous case \eqref{dR}, and respectively
$H_0(\curl;\Omega)$ in the 
homogeneous case \eqref{dR0}.
This amounts to verifying that the tangential component of $\Pi^1 \bu$ is continuous across any edge 
$e \in \cE$, and in the homogeneous case \eqref{dR0} that it further vanishes on boundary edges.
For this we consider some unit tangent vector $\btau_e$ and $k \in \cK(e)$. 
Denoting by ${\cdot}|^k_e$ the restriction on $e$ of the $\Omega_k$ piece of some broken
field, we write
\begin{equation*}
\btau_e \cdot (\Pi^1 \bu)|^k_e = A^k_e  + B^k_e  + C^k_e + D^k_e
\quad \text{ with } \quad
\left\{
\begin{aligned}
	A^k_e 
		&= \btau_e \cdot (\Pi^1_k \bu) |^k_e
	\\	
	B^k_e 
		&= \btau_e \cdot \sum_{e' \in \cE(k)} (\tilde \Pi^1_{e'} \bu)|^k_e
		\\
	C^k_e 
			&= \btau_e \cdot \sum_{\bv \in \cV(e)} (\tilde \Pi^1_{\bv} \bu)|^k_e
	\\
	D^k_e 
			&= \btau_e \cdot \sum_{\substack{\bv \in \cV(e) \\ e' \in \cE(\bv)}} (\tilde \Pi^1_{e',\bv} \bu)|^k_e ~.
\end{aligned}
\right.
\end{equation*}
Here, we have restricted the vertex sums over $\cV(e)$ (the vertex contiguous to $e$), 
since all the vertex and edge-vertex
projection operators map into functions which vanish on $e$ for $\bv \notin \cV(e)$
(this follows from the interpolation property of the basis functions at the patch boundaries).
Using \eqref{sum_nabla_ked}, i.e., $\nabla_\pw = \nabla^e_\parallel + \nabla^e_\perp$ 
on $\Omega(e)$, and $\btau_e \cdot \nabla^e_\perp = 0$, we compute 
\begin{equation*}
A^k_e 
	= \btau_e \cdot \Pi^1_k \bu |^k_e = \btau_e \cdot (\nabla^e_\parallel \Pi^0_k \Phi^k_{\parallel(e)}(\bu))|^k_e
	= \btau_e \cdot (\nabla^e_\parallel I^e \Pi^0_k \Phi^k_{\parallel(e)}(\bu))|^k_e
\end{equation*}
where the third equality follows from the fact that basis functions vanishing on $e$ have also a vanishing parallel 
gradient on $e$. Here the antiderivative \eqref{Phi_kd_exp} is taken in the direction parallel to $e$, that is 
\begin{equation*}
\Phi^k_{\parallel(e)}(\bu)(\bx) 
	= \int_{0}^{\hx^k_\parallel} \hat u^k_{\parallel(e)}(\hat X^k_e(z_\parallel, \hx^k_\perp)) \dd z_\parallel
\quad \text{ with } \quad \hbx = \hat X^k_e(\hx^k_\parallel, \hx^k_\perp) = (F_k)^{-1}(\bx).
\end{equation*}
Then,
\begin{equation*}
\begin{aligned}
	B^k_e 
		&= 
		\sum_{e' \in \cE(k)} \btau_e \cdot (\tilde \Pi^1_{e'} \bu)|^k_e 
	\\
		&= \btau_e \cdot \sum_{e' \in \cE(k)} \sum_{d \in \{\parallel, \perp\}} 
		\nabla^{e'}_d (P^{e'} - I^{e'}) \Pi^0_\pw \Phi^{e'}_{d}(\bu) |^k_e 
	\\
	&= 
	\bar B^k_e - \tilde A^k_e + \tilde B^k_e 
\end{aligned}
\end{equation*}
with 
\begin{equation*}
\left\{
\begin{aligned}
	\bar B^k_e &= \btau_e \cdot \nabla^e_\parallel P^{e} \Pi^0_\pw \Phi^{e}_{\parallel}(\bu)	
	\\
	\tilde A^k_e &= \btau_e \cdot \nabla^e_\parallel I^{e} \Pi^0_\pw \Phi^{e}_{\parallel}(\bu)
	\\
	\tilde B^k_e &= 
	\btau_e \cdot \sum_{\bv \in \cV(e)} \mspace{0mu} \nabla^{e'(\bv)}_\perp (P^{e'(\bv)} - I^{e'(\bv)}) \Pi^0_\pw \Phi^{e'(\bv)}_{\perp}(\bu)|^k_e
\end{aligned}
\right.
\end{equation*}
where $e'(\bv)$ is the edge $e'\neq e$ contiguous to $\bv$,
and we used $\btau_e \cdot \nabla^{e'(\bv)}_\parallel = 0$.
For an interior edge we see that $\bar B^k_e$ is continuous across $e$ 
(in the sense that $\bar B^-_e = \bar B^+_e$) as the tangential derivative of a 
function continuous across $e$, and for a homogeneous boundary edge we have $\bar B^k_e =0$ 
since $P^{e} = 0$.
By observing that  
$
\Phi^{e}_{\parallel}(\bu)(\bx) - \Phi^k_{\parallel(e)}(\bu)(\bx)
= \int_{\eta_e^k(0)}^{0} \hat u^k_{\parallel(e)}(X^k_e(z_\parallel, \hx^k_\perp)) \dd z_\parallel
$
is a function of $\hx^k_\perp$ only, we further infer from Lemma~\ref{lem:parinv} that 
\begin{equation*}
\tilde A^k_e - A^k_e = \nabla^e_\parallel I^e \Pi^0_k (\Phi^{e}_{\parallel}(\bu) - \Phi^k_{\parallel(e)}(\bu)) = 0.
\end{equation*}
For the third term we compute, using again $\btau_e \cdot \nabla_\pw = \btau_e \cdot \nabla^e_\parallel$,
\begin{equation*}
\begin{aligned}
C^k_e 
	&= \btau_e \cdot \sum_{\bv \in \cV(e)} (\tilde \Pi^1_{\bv} \bu)|^k_e
	\\
	&=  \btau_e \cdot \sum_{\bv \in \cV(e)}
	\nabla^e_\parallel P^{\bv} \Pi^0_\pw \Phi^\bv(\bu) |^k_e
	- \btau_e \cdot \sum_{\bv \in \cV(e)}
	\nabla^e_\parallel \bar I^{\bv} \Pi^0_\pw \Phi^\bv(\bu) |^k_e
	\\
	&=: \bar C^k_e - \tilde C^k_e
\end{aligned}
\end{equation*}
and for the last one we write, using \eqref{Phi_ev_par} and \eqref{Phi_ev_perp},
\begin{equation*}
\begin{aligned}
	D^k_e &= 
		\btau_e \cdot \sum_{\bv \in \cV(e), e' \in \cE(\bv)} (\tilde \Pi^1_{e',\bv} \bu)|^k_e
	\\
	&= \btau_e \cdot \sum_{\bv \in \cV(e)} \nabla^e_\parallel \bar I^e_{\bv} \Pi^0_\pw \Phi^{\bv}(\bu) |^k_e
	- \btau_e \cdot \sum_{\bv \in \cV(e)}	\nabla^e_\parallel P^e_{\bv} \Pi^0_\pw \Phi^{\bv}(\bu) |^k_e
		\\
		& \qquad \qquad + \btau_e \cdot \sum_{\bv \in \cV(e)} \nabla^{e'(\bv)}_\perp 
				(\bar I^{e'(\bv)}_{\bv} - P^{e'(\bv)}_{\bv}) \Pi^0_\pw \Phi^{e'(\bv)}_\perp(\bu) |^k_e
		\\
		&	
		=: \tilde D^k_e - \bar D^k_e + \check D^k_e.
\end{aligned}
\end{equation*}
According to \eqref{IvIev} the equality $\bar I^\bv \phi = \bar I^e_{\bv} \phi$ holds on $e$: this yields
$
\tilde C^k_e = \tilde D^k_e,
$
moreover for $e' = e'(\bv)$ we have $P^{e'} \phi = P^{e'}_{\bv} \phi$ 
and $I^{e'} \phi = \bar I^{e'}_{\bv} \phi$ on $e$
(this holds both in the inhomogeneous and homogeneous cases).
This yields
$
\tilde B^k_e = -\check D^k_e.
$
Thus, we obtain that 
\begin{equation*}
\btau_e \cdot (\Pi^1 \bu)|^k_e = \bar B^k_e + \bar C^k_e - \bar D^k_e
\end{equation*}
where these three terms are tangential derivatives of 
fields which are continuous across $e$ (and hence are also continuous across $e$) 
if the latter is an interior edge,
or vanish if $e$ is a homogeneous boundary edge.
This shows that $\Pi^1 \bu \in H(\curl;\Omega)$, resp. $H_0(\curl;\Omega)$ in the homogeneous case
\eqref{dR0}, and completes the proof.
\end{proof}

\subsection{Projection property}

\begin{lemma} \label{lem:proj}
	For all $\bu \in V^1_h$ and $f \in V^2_h$, we have $\Pi^1 \bu = \bu$
	and $\Pi^2 f = f$.
\end{lemma}

\begin{proof}
	We first consider $\Pi^1$ and 
	observe that for all $k$, the restriction $\bu|_{\Omega_k}$ belongs to the local space 
	$V^1_k$. Hence, the projection property of the local projection operator 
	gives $(\Pi^1_k \bu)|_{\Omega_k} = \bu|_{\Omega_k}$: it follows that
	\begin{equation*}
	\sum_{k \in \cK} \Pi^1_k \bu = \bu.
	\end{equation*}
	We thus need to show that the correction terms $\tilde \Pi^1_{e} \bu$, $\tilde \Pi^1_{\bv} \bu$
	and $\tilde \Pi^1_{e,\bv} \bu$ all vanish for $\bu \in V^1_h$.
	As for the first term we know from Lemma~\ref{lem:Phi_V1} that the parallel and perpendicular antiderivatives 
	$\Phi^e_\parallel(\bu)$ and $\Phi^e_\perp(\bu)$ belong to $V^0_\pw$, hence they are left unchanged by the patch-wise projection $\Pi^0_\pw$. 
	Moreover, they are continuous across any interior edge $e$ so that
	Lemma~\ref{lem:0_corr_e} and allows us to write
	\begin{equation} \label{ker_par}	
		(P^e - I^e) \Pi^0_\pw \Phi^e_d(\bu) = 0, \qquad d \in \{\parallel, \perp\}.
	\end{equation}
	We further observe that this equality also holds on boundary edges: 
	in the inhomogeneous case this follows from the fact that $P^e = I^e$,
	and in the homogeneous case (where $P^e = 0$) it follows from the 
	fact that $\Phi^e_d(\bu) = 0$ (again by \eqref{lem:Phi_V1}).
	As a result the edge correction terms vanish: 
	$\tilde \Pi^1_{e} \bu = 0$ for $\bu \in V^1_h$.
	The same reasoning applies to the vertex correction term: 
	according again to Lemma~\ref{lem:Phi_V1}, the antiderivative $\Phi^\bv(\bu)$ belongs to $V^0_\pw$ and it is continuous across any interior edge $e \in \cE(\bv)$ (and in the homogeneous case it vanishes on the boundary). 
	Then Lemma~\ref{lem:corr_ev_ker} applies, which yields
	\begin{equation}
		(P^\bv - \bar I^\bv) \Pi^0_\pw \Phi^\bv(\bu) = 0
	\end{equation}
	and hence $\tilde \Pi^1_{\bv} \bu = 0$.
	Turning to the edge-vertex correction terms we infer from \eqref{Phi_ev_par} and \eqref{Phi_ev_perp}
	that both $\Phi^{e,\bv}_\parallel(\bu)$ and $\Phi^{e,\bv}_\perp(\bu)$ are in $V^0_\pw$ and continuous across interior edges $e$ (and vanish on boundary edges in the homogeneous case). Applying again Lemma~\ref{lem:corr_ev_ker} yields then
	\begin{equation*}
	(P^e_{\bv} - \bar I^e_{\bv}) \Pi^0_\pw \Phi^{e,\bv}_d(\bu) = 0, \qquad d \in \{\parallel, \perp\},
	\end{equation*}
	which shows that $\tilde \Pi^1_{e,\bv} \bu = 0$ and finishes the proof.
	To show that $\Pi^2$ is a projection, we use a similar argument based on 
	Lemma~\ref{lem:Pi2_V2h}.
\end{proof}

\subsection{Commuting property}
 
\begin{lemma}
 	For the projection operators built in the inhomogeneous
	case \eqref{dR}, the equality
		\begin{equation} \label{cd_grad}
		\Pi^1 \nabla \phi = \nabla \Pi^0 \phi
		\end{equation}
		holds for all $\phi \in H^1(\Omega)$, and 
		for all $\phi \in H^1_0(\Omega)$ in the homogeneous
		case \eqref{dR0}.
\end{lemma}

\begin{remark}
	The commuting relation \eqref{cd_grad} also holds in the respective (larger) spaces $W^{1,1}(\Omega)$ and $W^{1,1}_0(\Omega)$, as mentioned 
	in Remark~\ref{rem:L1cd}.
\end{remark}
 
\begin{proof}
We consider $\bu = \nabla \phi$, with $\phi \in C^1(\Omega)$ in the 
inhomogeneous case and $\phi \in C^1_0(\Omega)$ in the homogeneous case: 
the result will then follow by a density argument, 
using the $L^1$ stability of the projection operators.  
Throughout this proof we write $\phi_h = \Pi^0_\pw \phi \in V^0_\pw$.
For the volume terms, 
we have seen in \eqref{pw_grad_cp} that the commutation of the patch-wise 
projection operators yield
\begin{equation} \label{CP_k}
	\sum_{k \in \cK} \Pi^1_k \bu = \nabla_\pw \phi_h.
\end{equation}
For the parallel edge correction terms, \eqref{Phi_par_nabla} reads
\begin{equation*}
\Phi^e_\parallel(\bu)(\bx) = \phi(\bx) - \tilde \phi_e(\bx) 
\end{equation*}
on $\Omega(e)$, $e\in \cE$, for some function $\tilde \phi_e$ independent of $\hx^k_\parallel$.
Hence, Lemma~\ref{lem:parinv} yields
\begin{equation} \label{CP_e_par}
\nabla^e_\parallel ( P^e - I^e) \Pi^0_\pw \Phi^e_\parallel(\bu) = 
\nabla^e_\parallel ( P^e - I^e) \phi_h.
\end{equation}
Next for the perpendicular edge correction terms, 
\eqref{Phi_perp_nabla} and \eqref{Phi_perp_nabla_hom} yield
\begin{equation}	\label{Phi_phi_cst}	
	\Phi^e_{\perp, a}(\bu)(\bx) = \phi(\bx) - \bar \phi_e
\end{equation}	
(on $\Omega(e)$) with some constant value $\bar \phi_e$
for interior and boundary edges.
In the homogeneous case we further have $\bar \phi_e = 0$
on boundary edges.
For interior and inhomogeneous boundary edges we can use the fact that 
$\Pi^0_\pw$ preserve global constants (which are continuous across $e$),
so that Lemma~\ref{lem:0_corr_e} gives
$(P^e - I^e) \Pi^0_\pw \bar \phi_e = (P^e - I^e) \bar \phi_e = 0$.
Hence, in both cases we have for any $e \in \cE$
\begin{equation} \label{CP_e_perp}
\nabla^e_\perp (P^e - I^e) \Pi^0_\pw \Phi^e_\perp(\bu) = 
\nabla^e_\perp (P^e - I^e) \phi_h.
\end{equation}
Using again that $\nabla^e_\parallel + \nabla^e_\perp = \nabla_\pw$ on $\Omega(e)$, it follows that 
\begin{equation} \label{CP_e}
\tilde \Pi^1_e = \sum_d \nabla^e_d (P^e - I^e) \Pi^0_\pw \Phi^e_d(\bu)
	= \nabla_\pw (P^e - I^e) \phi_h.
\end{equation}
Relations similar to \eqref{Phi_perp_nabla} and \eqref{Phi_perp_nabla_hom}, 
namely \eqref{Phi_v_nabla} and \eqref{Phi_v_nabla_hom},
hold for the vertex antiderivative, hence on $\bv \in \cV$ we have
\begin{equation} \label{CP_v}
\tilde \Pi^1_\bv = \nabla_\pw (P^\bv-\bar I^\bv) \Pi^0_\pw \Phi^\bv_\perp(\bu) 
	= \nabla_\pw (P^\bv-\bar I^\bv) \phi_h.
\end{equation}
Finally, for the edge-vertex correction, the respective antiderivative 
operators \eqref{Phi_ev_par} and \eqref{Phi_ev_perp} are of vertex 
and edge perpendicular type, hence they also satisfy
a relation of the form \eqref{Phi_phi_cst}. It follows that 
\begin{equation} \label{CP_ev}
\tilde \Pi^1_{e,\bv} \bu = 
	\sum_{d \in \{\parallel, \perp\}} 
		\nabla^e_d (\bar I^e_{\bv} - P^e_{\bv})\Pi^0_\pw \Phi^{e,\bv}_d(\bu)
	= \nabla_\pw (\bar I^e_{\bv} - P^e_{\bv}) \phi_h.
\end{equation}
With the decomposition \eqref{V0pw_dec}, i.e.
$
\phi_h = \Big(\sum_k I^k_0 + \sum_{e} I^e_{0} + \sum_{\bv} I^\bv\Big)\phi_h
$,
this allows us to write $\Pi^1 \bu = \nabla_\pw \psi_h$ with 
\begin{equation*}
\begin{aligned}
	\psi_h
		&= \phi_h + \Big(\sum_e (P^e - I^e) + \sum_\bv (P^\bv - \bar I^\bv)
		+ \sum_{e,\bv} (\bar I^e_{\bv} - P^e_{\bv}) \Big) \phi_h
		\\
		&= \Big(\sum_k I^k_0 + \sum_{e} (I^e_{0} + P^e - I^e)
			+ \sum_\bv (I^\bv - P^\bv - \bar I^\bv) 
			+ \sum_{e,\bv} (\bar I^e_{\bv} - P^e_{\bv}) \Big)\phi_h.
\end{aligned}
\end{equation*}
We then observe that \eqref{sum_e_Iev}, \eqref{Ie0} yield
$\sum_{e} (I^e_{0} - I^e) = -\sum_{e,\bv} I^e_{\bv} = - \sum_\bv 2 I^\bv$,
while \eqref{Iv_as_bc} is 
$\bar I^\bv = \sum_e \bar I^e_{\bv} - I^\bv$.
With \eqref{Pe0}, i.e., $P^e - \sum_\bv P^e_{\bv} = P^e_{0}$, and the 
decomposition \eqref{P}, this gives
\begin{equation*}
\psi_h
	= \Big(\sum_k I^k_0 + \sum_{e} P^e
		+ \sum_\bv P^\bv 
		- \sum_{e,\bv} P^e_{\bv} \Big)\phi_h = P \phi_h.
\end{equation*}
This shows that $\psi_h$ is continuous on $\Omega$, hence 
$\nabla_\pw \psi_h = \nabla \psi_h$, and finally we find
$\Pi^1 \nabla \phi = \nabla \psi_h = \nabla P \phi_h = 
\nabla P \Pi^0_\pw \phi = \nabla \Pi^0 \phi$,
which completes the proof.
\end{proof}

\begin{lemma}
	For the projection operators built in the inhomogeneous
	case \eqref{dR}, the equality
		\begin{equation} \label{cd_curl}
		\Pi^2 \curl \bu = \curl \Pi^1 \bu
		\end{equation}
		holds for all $\bu \in H(\curl;\Omega)$,
		and 
		for all $\bu \in H_0(\curl;\Omega)$ in the homogeneous
		case \eqref{dR0}.
\end{lemma}

\begin{remark}
The commuting relation \eqref{cd_curl} also holds in the respective (larger) spaces 
$W^{1}(\curl;\Omega)$ and $W^{1}_0(\curl;\Omega)$ defined
in Remark~\ref{rem:L1cd}.
\end{remark}

\begin{proof}
By a density argument we may consider $\bu \in C^1(\Omega)$
and $\bu \in C^1_0(\Omega)$ in the respective inhomogeneous and homogeneous cases.
According to Lemma~\ref{lem:Pik} we know that 
the single-patch projections commute with the patch-wise curl operator, namely
\begin{equation*}
\curl^k \Pi^1_k \bu = \Pi^2_k \curl \bu. 
\end{equation*}
Since every vertex correction term 
\eqref{tPi1_v} is a patch-wise gradient, we also have
\begin{equation*}
\curl_\pw \tilde \Pi^1_{\bv}\bu = 0.
\end{equation*}
For the edge correction terms \eqref{tPi1_e}, we use Lemma~\ref{lem:D2e}
with $\psi_d = (P^e - I^e) \Pi^0_\pw \Phi^e_d(\bu)$
and compute
\begin{equation*}
\begin{aligned}
	\curl_\pw \tilde \Pi^1_e \bu
	&= D^{2,e} (P^e - I^e) \Pi^0_\pw 
			\big(\Phi^e_{\perp}(\bu) - \Phi^e_{\parallel}(\bu)\big)
	\\
	&= D^{2,e} (P^e - I^e) \Pi^0_\pw 
			\big(\Phi^e_{\perp}(\bu) - \Phi^e_{\parallel}(\bu) + \tilde \Phi^e(\bu)\big)
	\\
	&= D^{2,e} (P^e - I^e) \Pi^0_\pw 
			\Psi^{e}(\curl \bu)
\end{aligned}
\end{equation*}
where the second and third equalities follow from Lemma~\ref{lem:Psi_curl} 
and the parallel invariance preserving property of the operators
$\Pi^0_\pw$, $P^e$ and $I^e$, see Lemma~\ref{lem:parinv}:
note that an invariance along $\hx_\parallel$ leads indeed to
the cancellation of the mixed derivative $D^{2,e}$ on each patch.
For the edge-vertex correction terms \eqref{tPi1_ev} we use again
Lemma~\ref{lem:D2e} and write
\begin{equation*}
\begin{aligned}
	\curl_\pw \tilde \Pi^1_{e,\bv} \bu
	&= D^{2,e} (\bar I^e_{\bv} - P^e_{\bv}) \Pi^0_\pw 
			\big(\Phi^{e,\bv}_{\perp}(\bu) - \Phi^{e,\bv}_{\parallel}(\bu)\big)
	\\
	&= D^{2,e} (\bar I^e_{\bv} - P^e_{\bv}) \Pi^0_\pw 
			\big(\Phi^{e,\bv}_{\perp}(\bu) - \Phi^{e,\bv}_{\parallel}(\bu) 
			+ \tilde \Phi^{e,\bv}(\bu)\big)
	\\
	&= D^{2,e} (\bar I^e_{\bv} - P^e_{\bv}) \Pi^0_\pw 
			\Psi^{e}(\curl \bu)
\end{aligned}
\end{equation*}
where the second and third equalities follow from Lemma~\ref{lem:Psi_curl}
and the preservation of constants by the operator $\Pi^0_\pw$,
which are in the kernel of $\bar I^e_{\bv} - P^e_{\bv}$, see
Lemma~\ref{lem:corr_ev_ker}.
Gathering the computations above and using the form of the 
$\Pi^2$ projection, we find	
\begin{equation*}
\begin{aligned}
	\curl \Pi^1 \bu 
	&= \curl_\pw \Pi^1 \bu 
	\\
	&= \sum_{k \in \cK} \curl^k \Pi^1_k  \bu
				+ \sum_{e \in \cE} \curl_\pw \tilde \Pi^1_e\bu
						+ \sum_{\bv \in \cV} \curl_\pw \tilde \Pi^1_\bv\bu
							+ \sum_{\substack{\bv \in \cV \\ e \in \cE(\bv)}} \curl_\pw \tilde \Pi^1_{e,\bv}\bu
	\\
	&= \sum_{k \in \cK} \Pi^2_k \curl \bu
				+ \sum_{e \in \cE} \tilde \Pi^2_e \curl \bu
							+ \sum_{\substack{\bv \in \cV \\ e \in \cE(\bv)}} 
								\tilde \Pi^2_{e,\bv}\curl \bu
	= \Pi^2 \curl \bu.
\end{aligned}
\end{equation*}
\end{proof}

\section{Conclusion}
 \label{sec:conclusion}

 In this article we have proposed a new approach for constructing
 $L^2$ stable commuting projection operators on de Rham sequences
 of multipatch spaces, 
 which applies to locally refined patches
 with tensor-product structure. 

 Our construction involves single-patch projections that rely on the tensor-product 
structure of the single-patch spaces, and correction terms for the interfaces.
Like the single-patch projections, the correction terms are composed of partial 
derivatives, local projections and antiderivative operators: the specifity of 
the latter is to involve projections on the local conforming and broken spaces
associated with an interface.

Being local, our construction naturally yields projection operators which are stable 
in any $L^p$ norm with $p \in [1,\infty]$.
It also applies to de Rham sequences with homogeneous boundary conditions.

Looking ahead, we believe that our construction can be extended to 3D domains and this 
will be the subject of a separate work.
Applying these theoretical findings to the design of stable numerical schemes 
is a work in progress. Preliminary experiments conducted on curl-curl eigenvalue 
problems have yielded promising results, particularly when employing broken-FEEC 
schemes \cite{campos-pinto_broken-feec_2022}: these and further studies will be described in a forthcoming article.

\bibliographystyle{amsplain}
\bibliography{refs}

\end{document}